\documentclass{amsart}
\usepackage{amsmath,amsthm,enumerate,mathrsfs,graphicx,makecell,amssymb,booktabs}

\newtheorem{theorem}{Theorem}
\newtheorem{proposition}{Proposition}[section]

\newtheorem{lemma}[proposition]{Lemma}

\newtheorem{definition}[proposition]{Definition}

\theoremstyle{definition}

\newcommand{\vertle}[1]{{\left\vert\kern-0.25ex\left\vert\kern-0.25ex\left\vert #1 
    \right\vert\kern-0.25ex\right\vert\kern-0.25ex\right\vert}}

\DeclareMathOperator{\tr}{tr}
\DeclareMathOperator{\dimaff}{dim_{\mathsf{aff}}}
\DeclareMathOperator{\Int}{Int}
\DeclareMathOperator{\diam}{diam}

\DeclareMathOperator{\sign}{sign}

\DeclareMathOperator{\Fix}{Fix}

\newcommand{\iii}{\mathtt{i}}
\newcommand{\jjj}{\mathtt{j}}

\title[Approximating the affinity dimension]{Fast approximation of the affinity dimension for dominated affine iterated function systems}
\begin{document}
\author{Ian D. Morris}
\begin{abstract}
In 1988 K. Falconer introduced a formula which predicts the value of the Hausdorff dimension of the attractor of an affine iterated function system. The value given by this formula -- sometimes referred to as the \emph{affinity dimension} -- is known to agree with the Hausdorff dimension both generically and in an increasing range of explicit cases. It is however a nontrivial problem to estimate the numerical value of the affinity dimension for specific iterated function systems. In this article we substantially extend an earlier result of M. Pollicott and P. Vytnova on the computation of the affinity dimension. Pollicott and Vytnova's work applies to planar invertible affine contractions with positive linear parts under several additional conditions which among other things constrain the affinity dimension to be between 0 and 1. We extend this result by passing from planar self-affine sets to self-affine sets in arbitrary dimensions, relaxing the positivity hypothesis to a domination condition, and removing all other constraints including that on the range of values of the affinity dimension. We provide explicit examples of two- and three-dimensional affine iterated function systems for which the affinity dimension can be calculated to more than 30 decimal places.
\end{abstract}
\maketitle

\section{Introduction}\label{se:one}

\subsection{Background and context}

If $T_1,\ldots,T_N \colon \mathbb{R}^d \to \mathbb{R}^d$ are contractions it is well-known that there exists a unique nonempty compact set $X \subset \mathbb{R}^d$ such that $X=\bigcup_{i=1}^NT_iX$. In this case $(T_1,\ldots,T_N)$ is called an \emph{iterated function system} and the set $X$ its \emph{attractor}. When each transformation $T_i$ is a similitude with contraction ratio $r_i \in (0,1)$ and the distinct images $T_iX \cap T_jX$ do not overlap too strongly it is classical that the box dimension and Hausdorff dimension of the attractor are both equal to the unique real number $s>0$ such that $\sum_{i=1}^Nr_i^s=1$ (see for example \cite[Theorem 9.3]{Fa14} or the original article \cite{Hu81}). In the case where each $T_i$ is instead an affine map $T_ix=A_ix+v_i$ the Hausdorff dimension and box dimension of the attractor $X$ -- which in this context we call a \emph{self-affine set} -- are more challenging to calculate. The problem of determining the Hausdorff dimension of such sets, even implicitly, has been an active topic of research since the 1980s and has received particularly intense research interest within the last decade (see for example the classic articles \cite{Be84,Ed92,Fa88,Fa92,HuLa95,Mc84} and more recent contributions such as   \cite{Ba08,BaHoRa18,BoMo18,DaSi17,FaKe18,FeSh14,Fr12,KaSh09,MoSh17}).  In the landmark article \cite{Fa88} K. Falconer defined an implicit formula which is known to give the correct value for the Hausdorff dimension of a wide variety of self-affine sets. The subject of this article is the numerical estimation of the value predicted by Falconer's formula.

In order to define Falconer's formula we require a few preliminary definitions. Let $M_d(\mathbb{R})$ denote the set of all real $d \times d$ matrices. If $A \in M_d(\mathbb{R})$ we recall that the \emph{singular values} of $A$ are defined to be the square roots of the eigenvalues of the positive semidefinite matrix $A^\top A$. We denote the singular values of $A \in M_d(\mathbb{R})$ by $\sigma_1(A),\ldots,\sigma_d(A)$ in decreasing order of absolute value. For each $A \in M_d(\mathbb{R})$ and $s\geq 0$ let us define
\[\varphi^s(A):=\left\{\begin{array}{cl}\sigma_1(A)\cdots \sigma_{\lfloor s \rfloor}(A) \sigma_{\lceil s \rceil}(A)^{s-\lfloor s \rfloor}&\text{if }0 \leq s \leq d,\\
|\det A|^{\frac{s}{d}}&\text{if }s \geq d.\end{array}\right.\]
It was shown in \cite{Fa88} that for each $s \geq 0$ we have $\varphi^s(AB) \leq \varphi^s(A)\varphi^s(B)$ for all $A,B \in M_d(\mathbb{R})$. The \emph{affinity dimension} of the iterated function system $T_ix:=A_ix=v_i$, where $1 \leq i \leq N$, is then defined to be the quantity
\[\dimaff (T_1,\ldots,T_N) := \inf \left\{s>0 \colon\sum_{n=1}^\infty \sum_{i_1,\ldots,i_n=1}^N \varphi^s(A_{i_1}\cdots A_{i_n})<\infty\right\}.\]
Since $\dimaff (T_1,\ldots,T_N)$ depends only on $A_1,\ldots,A_N$ and not on the additive part of the transformations $T_i$ we will also denote it by $\dimaff (A_1,\ldots,A_N)$.
If the matrices $A_1,\ldots,A_N$ are assumed to be invertible and contracting with respect to some norm on $\mathbb{R}^d$ then the affinity dimension  is the unique $s>0$ such that the quantity
\[ P(A_1,\ldots,A_N;s):= \lim_{n \to \infty }\frac{1}{n}\log \sum_{i_1,\ldots,i_n=1}^N \varphi^s(A_{i_1}\cdots A_{i_n})\]
is equal to zero.

Let $\|\cdot\|$ denote the Euclidean norm on $\mathbb{R}^d$. It was shown in \cite{Fa88} that when $\max_{1 \leq i \leq N}\|A_i\|<1$ the affinity dimension $\dimaff (A_1,\ldots,A_N)$ is well-defined and is an upper bound for the box dimension of the attractor. (This argument may easily be adapted to the case where $\max_{1 \leq i \leq N}\vertle{A_i}<1$ in the operator norm induced by some norm $\vertle{\cdot}$ on $\mathbb{R}^d$.)  It was additionally shown that when matrices $A_1,\ldots,A_N$ satisfying $\max_{1 \leq i \leq N}\|A_i\|<\frac{1}{3}$ are fixed, then for Lebesgue-a.e. choice of $(v_1,\ldots,v_N) \in (\mathbb{R}^d)^N$ the attractor of the affine transformations $T_1,\ldots,T_N$ given by $T_ix:=A_ix+v_i$ has Hausdorff dimension equal to $\min\{d,\dimaff(A_1,\ldots,A_N)\}$. Subsequent research focused on providing explicit examples for which the Hausdorff dimension of the attractor equals the affinity dimension of the defining iterated function system, with explicit special cases being given in articles such as \cite{FaKe18,Fr12,HuLa95,MoSh17}. Recently, B. Bara\'ny, M. Hochman and A. Rapaport have shown that the Hausdorff dimension of a planar self-affine set is always equal to the affinity dimension of the defining iterated function system as long as the matrices $A_i$ are invertible, the affine transformations satisfy the strong open set condition, and the matrices $|\det A_i|^{-1/2}A_i$ neither belong to a compact subgroup of $GL_2(\mathbb{R})$ nor preserve a finite subset of $\mathbb{RP}^1$. At the present time, however, results on higher-dimensional self-affine sets additional to that of Falconer are essentially unavailable.

Despite its prominent r\^ole in the dimension theory of self-affine sets, the properties of the affinity dimension itself have been investigated only very recently. In the 2014 article \cite{FeSh14} D.-J. Feng and P. Shmerkin showed for the first time that the affinity dimension $\dimaff (A_1,\ldots,A_N)$ depends continuously on the entries of the matrices $A_1,\ldots,A_N$, and in \cite{Mo16} it was shown that the affinity dimension is computable in principle in the sense that for any given $\varepsilon>0$ we may algorithmically compute an explicit approximation to $\dimaff(A_1,\ldots,A_N)$ which is guaranteed to be accurate to within the prescribed error $\varepsilon$ and which requires only finitely many arithmetical operations to calculate. However, the method of \cite{Mo16} does not result in an algorithm which is fast enough to be useful in practical computations. Further general properties of the affinity dimension were investigated in \cite{BoMo18,KaMo16}. 

At the present time there are very few practical techniques available for the computation of the affinity dimension. In the article \cite{Mo19} the author gave a simple closed-form expression for the affinity dimension in the very special case where the matrices $A_i$ are generalised permutation matrices, that is, matrices having exactly one nonzero entry in every row and column. Closed-form expressions are also available in the case of diagonal and upper-triangular matrices \cite{FaMi07,KaMo16}. To the best of the author's knowledge there so far exists only one result in the literature which is powerful enough to be able to estimate the affinity dimension for a nonempty open set of examples in a practicable time frame. The following result was proved by M. Pollicott and P. Vytnova in \cite{PoVy16}. Here and throughout this article $\rho(A)$ denotes the spectral radius of the matrix or linear operator $A$.
\begin{theorem}\label{th:one}
Let $A_1,\ldots,A_N$ be $2\times 2$ matrices which satisfy the following conditions:
\begin{enumerate}[(i)]
\item
We have $\sigma_1(A_i)^2< \sigma_2(A_i)<1$ for all $i=1,\ldots,N$.
\item
If $\mathcal{Q}_2$ is defined to be the open second quadrant $\{(x,y)\in \mathbb{R}^2 \colon x<0<y\}$, then the sets $A_1^{-1}\mathcal{Q}_2,\ldots,A_N^{-1}\mathcal{Q}_2$ are subsets of $\mathcal{Q}_2$ and have pairwise disjoint closures in $\mathcal{Q}_2$.
\item
All entries of the matrices $A_i$ are strictly positive\footnote{This hypothesis is invoked in Pollicott and Vytnova's section 3 but is not explicitly stated in their introduction. It does not follow automatically from the other hypotheses unless the determinants are assumed positive.}.
\end{enumerate}
For each $n \geq 1$ and $s \in \mathbb{C}$ define
\[t_n(s)=\sum_{i_1,\ldots,i_n=1}^N \frac{\rho(A_{i_1}\cdots A_{i_n})^{2+s}}{\rho(A_{i_1}\cdots A_{i_n})^2 - \det A_{i_1}\cdots A_{i_n}},\]
\[a_n(s):=\sum_{k=1}^n \frac{(-1)^k}{k!} \sum_{\substack{(n_1,\ldots,n_k) \in \mathbb{N}^k \\ \sum_{i=1}^k n_i=n}} \prod_{i=1}^k \frac{t_{n_i}(s)}{n_i}\]
and $a_0(s):=1$, and for each $n\geq 1$ let $s_n \in \mathbb{R}$ denote the smallest positive real number $s$ such that $\sum_{i=0}^n a_i(s)=0$. Then $\dimaff(A_1,\ldots,A_N) \in (0,1)$, $s_n$ is well-defined for all sufficiently large $n$, and there exists $\gamma>0$ such that 
\[\left|\dimaff (A_1,\ldots,A_N)-s_n\right|=O\left(\exp(-\gamma n^2)\right).\]
\end{theorem}
\emph{Remark.} The quantity $a_n(s)$ may be alternatively characterised as
\[ \frac{(-1)^n}{n!}\det\begin{pmatrix}
t_1(s)& n -1&  0&\cdots &0 &0\\
t_2(s)&t_1(s)&n-2 &\cdots &0 &0\\
t_3(s)&t_2(s)&t_1(s) &\ddots &0 &0\\
\vdots & \vdots & \vdots & \ddots  &\ddots& \vdots\\
t_{n-1}(s) &t_{n-2}(s)&t_{n-3}(s)&\cdots &t_1(s) &1\\
t_n(s) &t_{n-1}(s)&t_{n-2}(s)&\cdots &t_2(s) &t_1(s)
\end{pmatrix},\]
and we will prefer this format in our exposition.

The methods underlying the proof of Theorem \ref{th:one} will be described in more detail in the following section. We remark that condition (i) above implies that the matrices are invertible, and the combination of the three conditions implies $0<\dimaff(A_1,\ldots,A_N)<1$ (see \cite{HuLa95} for details). 

In fact the only condition which is really essential to Pollicott and Vytnova's argument is that the matrix entries are positive, although in cases where we have $\dimaff(A_1,\ldots,A_N)\in (1,2)$ the formula for $t_n(s)$ must be replaced with
\[t_n(s):=\sum_{i_1,\ldots,i_n=1}^N \frac{\rho(A_{i_1}\cdots A_{i_n})^{4-s}|\det A_{i_1}\cdots A_{i_n}|^{s-1}}{\rho(A_{i_1}\cdots A_{i_n})^2 - \det A_{i_1}\cdots A_{i_n}}.\]
 In this article we aim to prove as comprehensive as possible an extension of Theorem \ref{th:one}. In particular, as well as removing hypotheses (i)--(ii) from Theorem \ref{th:one} we will establish a version of that theorem which is valid for affine iterated function systems in dimensions higher than two, in which $\dimaff(A_1,\ldots,A_N)$ may take any value in the range $(0,d)$, and in which the hypothesis of positivity is weakened to one of domination. In order to state our results in full we will require a number of definitions, which relate to multilinear algebra, to positivity and to domination.

\subsection{Multilinear algebra}

In extending Theorem \ref{th:one} one of our concerns will be to allow matrices of arbitrary dimension. Whereas in two dimensions the function $\varphi^s(A)$ admits the simple characterisation
\[\varphi^s(A)=\left\{\begin{array}{cl} \|A\|^s&\text{if }0 \leq s \leq 1,\\
|\det A|^{s-1}\|A\|^{2-s}&\text{if }1 \leq s \leq 2,\end{array}\right.\]
when $s>1$ and $d>2$  the analogous formula involves exterior powers of the matrix $A$. In order to study the singular value function $\varphi^s$ in dimensions higher than two we therefore need to recall some concepts and notation from multilinear algebra.

Recall that when $1 \leq k \leq d$ the real vector space $\wedge^k \mathbb{R}^d$ is the vector space spanned by the formal expressions $\{v_1\wedge v_2 \wedge \cdots \wedge v_k \colon v_1,\ldots v_k \in \mathbb{R}^d\}$ subject to the identifications
\[\lambda(v_1 \wedge v_2\wedge  \cdots \wedge v_k) = (\lambda v_1)\wedge v_2\wedge \cdots \wedge v_k,\]
\[(u_1\wedge v_2\wedge \cdots \wedge v_k)+(v_1\wedge v_2\wedge \cdots \wedge v_k)=(u_1+v_1)\wedge v_2\wedge \cdots \wedge v_k,\]
\[v_1 \wedge v_2 \wedge \cdots \wedge v_k = (-1)^{\sign (\pi)} v_{\pi(1)}\wedge v_{\pi(2)}\wedge \cdots v_{\pi(k)}\]
for all $v_1,\ldots,v_k,u_1 \in \mathbb{R}^d$, $\lambda \in \mathbb{R}$ and permutations $\pi \colon \{1,\ldots,k\}\to \{1,\ldots,k\}$. The vector space $\wedge^k \mathbb{R}^d$ is ${d\choose k}$-dimensional and if $v_1,\ldots,v_d$ is any basis for $\mathbb{R}^d$ then $\{v_{i_1}\wedge \cdots \wedge v_{i_k} \colon 1 \leq i_1<i_2<\cdots<i_k\leq d\}$ is a basis for $\wedge^k\mathbb{R}^d$. The ${d \choose k}$-dimensional vector space $\wedge^k\mathbb{C}^d$ may be constructed analogously.

The space $\wedge^k\mathbb{R}^d$ inherits an inner product $\langle\cdot,\cdot\rangle_{\wedge^k\mathbb{R}^d}$ from the standard inner product $\langle\cdot,\cdot\rangle$ on $\mathbb{R}^d$ which satisfies
\[\langle u_1\wedge \cdots \wedge u_k,v_1\wedge \cdots \wedge v_k\rangle_{\wedge^k\mathbb{R}^d} = \det \left(\left[\langle u_i,v_j\rangle\right]_{i,j=1}^k\right).\]
If $A \in M_d(\mathbb{R})$ then we may define a linear map $A^{\wedge k} \colon \wedge^k \mathbb{R}^d \to \wedge^k\mathbb{R}^d$ by $A^{\wedge k}(v_1 \wedge \cdots \wedge v_k)=Av_1 \wedge \cdots \wedge Av_k$. If $v_1,\ldots,v_d$ is a basis for $\mathbb{C}^d$ consisting of eigenvectors and generalised eigenvectors for $A$ then the vectors $v_{i_1}\wedge \cdots \wedge v_{i_k}$ form a basis for $\wedge^k\mathbb{C}^d$ and it is not hard to see that if $\lambda_1,\ldots,\lambda_d$ are the eigenvalues of $A$ then the eigenvalues of $A^{\wedge k}$ are precisely the ${d \choose k}$ different products $\lambda_{i_1}\cdots \lambda_{i_k}$ with $1 \leq i_1<\cdots<i_k\leq d$. It is clear from the definition of the inner product on $\wedge^k\mathbb{R}^d$ that $(A^{\wedge k})^\top =(A^\top )^{\wedge k}$. Combining these observations we may easily see that
\[\left\|A^{\wedge k}\right\|_{\wedge^k\mathbb{R}^d}=\rho\left(\left({A^\top}A\right)^{\wedge k}\right)^{\frac{1}{2}}=\sigma_1(A)\cdots \sigma_k(A)\]
for all $A \in M_d(\mathbb{R})$. By convention we also define $\wedge^0 \mathbb{R}^{d}=\mathbb{R}$ and $A^{\wedge 0}=1$. It follows easily that we may write
\[\varphi^s(A)=\left\|A^{\wedge\lfloor s\rfloor}\right\|^{1+s-\lfloor s\rfloor}\left\|A^{\wedge \lceil s \rceil}\right\|^{\lceil s \rceil-s}\]
for all $A \in M_d(\mathbb{R})$ and $s \in [0,d]$.

\subsection{Positivity and domination}

As well as increasing the dimension of the matrices to be considered in our extension of Theorem \ref{th:one} we would like to weaken as much as possible the hypothesis that the matrices have positive entries. To this end we introduce the following definition:
\begin{definition}\label{de:mc}
Let $\mathsf{A} \subset M_d(\mathbb{R})$ be  nonempty. We say that $(\mathcal{K}_1,\ldots,\mathcal{K}_m)$ is a \emph{multicone} for $\mathsf{A}$ if the following properties hold:
\begin{enumerate}[(i)]
\item
Each $\mathcal{K}_j$ is a closed, convex subset of $\mathbb{R}^d$ with nonempty interior such that $\lambda \mathcal{K}_j \subseteq \mathcal{K}_j$ for every non-negative real number $\lambda$.
\item
There exists a unit vector $w \in \mathbb{R}^d$ such that $\langle u,w\rangle> 0$ for all nonzero vectors $u \in \bigcup_{j=1}^m \mathcal{K}_j$. In particular $\mathcal{K}_j\cap-\mathcal{K}_j=\{0\}$ for all $j=1,\ldots,m$.
\item
For every $A \in \mathsf{A}$ and $j \in \{1,\ldots,m\}$ there exists $\ell =\ell(j,A) \in \{1,\ldots,m\}$ such that $A(\mathcal{K}_j\setminus \{0\})\subset (\Int \mathcal{K}_\ell) \cup (-\Int \mathcal{K}_\ell)$.
\item
For all distinct $j_1,j_2 \in \{1,\ldots,m\}$ we have $\mathcal{K}_{j_1}\cap \mathcal{K}_{j_2}=\{0\}$.
\end{enumerate} 
When (ii) holds we say that $w$ is a \emph{transverse-defining vector} for $(\mathcal{K}_1,\ldots,\mathcal{K}_m)$ since the hyperplane normal to $w$ is transverse to $\bigcup_{j=1}^m \mathcal{K}_j$. If a multicone for $\mathsf{A}$ exists then we say that $\mathsf{A}$ is \emph{multipositive}.
\end{definition}
We shall say that a set $\mathsf{A} \subset M_d(\mathbb{R})$ is \emph{$k$-multipositive} if the set $\{A^{\wedge k} \colon A \in \mathsf{A}\}$ is multipositive. By abuse of notation we shall say that a tuple of matrices is $k$-multipositive if and only if the corresponding set is. We observe that a tuple of $d \times d$ matrices with all entries positive is multipositive since we may take $m=1$ and $\mathcal{K}_1$ to be the closed positive orthant in $\mathbb{R}^d$. It follows that every tuple of $d \times d$ matrices is $0$-multipositive. We also observe that every tuple of $d \times d$ invertible matrices is $d$-multipositive.

In generalising Theorem \ref{th:one} we will adopt the hypothesis that $(A_1,\ldots,A_N)$ is  $k$-multipositive for certain integers $k$ depending on $\dimaff(A_1,\ldots,A_N)$. In the invertible case this hypothesis may be related to the concept of \emph{domination} as follows. 
If $1 \leq k < d$ then a tuple of invertible matrices $(A_1,\ldots,A_N) \in GL_d(\mathbb{R})^N$ is called \emph{$k$-dominated} if there exist $C,\gamma>0$ such that
\[\sigma_{k+1}(A_{i_1}\cdots A_{i_n}) \leq Ce^{-\gamma n}\sigma_k(A_{i_n}\cdots A_{i_1})\]
for all $i_1,\ldots,i_n \in \{1,\ldots,n\}$ and $n \geq 1$. By convention we will say that every  $(A_1,\ldots,A_N) \in GL_d(\mathbb{R})^N$  is both $0$- and $d$-dominated. It is not difficult to show using the observations made in the previous subsection that $(A_1,\ldots,A_N)$ is $k$-dominated if and only if $(A_1^{\wedge k},\ldots,A_N^{\wedge k})$ is $1$-dominated. Various characterisations of domination -- in terms of invariant splittings, singular values, contraction on projective spaces and contraction on Grassmannians -- were explored by J. Bochi, N. Gourmelon, M. Barnsley and A. Vince in \cite{BoGo09,BaVi12}. In particular it was shown in \cite{BoGo09} that a compact set of invertible  matrices is $1$-dominated if and only if it satisfies a slightly weakened form of multipositivity in which the criteria of Definition \ref{de:mc} all hold except that the sets $\mathcal{K}_j$ are not assumed to be convex. By repeating iteratively the operations of replacing each set $\mathcal{K}_j$ with its convex hull (which may introduce overlaps) and uniting overlapping sets $\mathcal{K}_j$ (which may introduce non-convexity but reduces the number of sets $\mathcal{K}_j$ to be considered) one may prove the following result by inductive descent on the number of sets $\mathcal{K}_j$:
\begin{proposition}[\cite{BaVi12}]\label{pr:app}
Let $\mathsf{A}\subset M_d(\mathbb{R})$ be compact, and suppose that every $A \in \mathsf{A}$ is invertible. Then $\mathsf{A}$ is $1$-dominated if and only if it is multipositive. 
\end{proposition}
An obvious consequence of this proposition is that for every $k=0,\ldots,d$ every compact set of invertible matrices $\mathsf{A}\subset M_d(\mathbb{R})$ is $k$-dominated if and only if it is $k$-multipositive.

\subsection{The main theorem}

In order to state our main theorem we require just a few more items of notation. For each $N \geq 1$ let us define
\[\Sigma_N^*:=\bigcup_{n=1}^\infty \{1,\ldots,N\}^n.\]
If $\iii=(i_k)_{k=1}^n \in \Sigma_N^*$ we write $|\iii|=n$ and refer to $|\iii|$ as the \emph{length} of $\iii$. If $\iii,\jjj \in \Sigma_N^*$ we let $\iii\jjj \in \Sigma_N^*$ denote the sequence of length $|\iii|+|\jjj|$ obtained by running first through the symbols of $\iii$ and then through those of $\jjj$ in the obvious fashion. Clearly $\Sigma_N^*$ is a semigroup with respect to the operation $(\iii,\jjj) \mapsto \iii\jjj$. If $A_1,\ldots,A_N \in M_d(\mathbb{R})$ and $\iii=(i_k)_{k=1}^n \in \Sigma_N^*$ then we write $A_\iii:=A_{i_n}\cdots A_{i_1}$. We observe that $A_\iii A_\jjj =A_{\jjj \iii}$ for all $\iii,\jjj \in \Sigma_N^*$. 

If $B$ is a linear transformation of a finite-dimensional real vector space we let $\lambda_1(B),\ldots,\lambda_d(B)$ denote the eigenvalues of $B$ listed with repetition according to multiplicity and listed in decreasing order of absolute value. While this notation \emph{a priori} introduces ambiguities when distinct eigenvalues of the same modulus exist, we will see that this consideration does not affect the statements of our results.

We may now present the following generalisation of Pollicott and Vytnova's result:
\begin{theorem}\label{th:main}
Let $d,N \geq 2$, let $(A_1,\ldots,A_N) \in M_d(\mathbb{R})^N$ and let $0 \leq k <d$. Suppose that $(A_1,\ldots,A_N)$ is both $k$-multipositive and $(k+1)$-multipositive. For each integer $n \geq 1$ and $s \in \mathbb{R}$ define
\[t_n(s):=\sum_{|\iii|=n} \frac{\lambda_1\left(A_\iii^{\wedge k}\right)^{{d \choose k} -1} \lambda_1\left(A_\iii^{\wedge\left(k+1\right)}\right)^{{d \choose k+1} -1} \rho\left(A_\iii^{\wedge k}\right)^{k+1-s} \rho\left(A_\iii^{\wedge \left(k+1\right)}\right)^{s-k} }{ p'_{A_\iii^{\wedge k}}\left(\lambda_1\left(A_\iii^{\wedge k}\right)\right)p'_{A_\iii^{\wedge \left(k+1\right)}}\left(\lambda_1\left(A_\iii^{\wedge \left(k+1\right)}\right)\right) }
\]
where $p_B'(x_0)$ denotes the first derivative of the characteristic polynomial $p_B(x):=\det (xI-B)$ evaluated at the point $x_0$. Define also
\[a_n(s):= \frac{(-1)^n}{n!}\det\begin{pmatrix}
t_1(s)& n -1&  0&\cdots &0 &0\\
t_2(s)&t_1(s)&n-2 &\cdots &0 &0\\
t_3(s)&t_2(s)&t_1(s) &\ddots &0 &0\\
\vdots & \vdots & \vdots & \ddots  &\ddots& \vdots\\
t_{n-1}(s) &t_{n-2}(s)&t_{n-3}(s)&\cdots &t_1(s) &1\\
t_n(s) &t_{n-1}(s)&t_{n-2}(s)&\cdots &t_2(s) &t_1(s)
\end{pmatrix}\]
for all $n \geq 1$, and $a_0(s):=1$. 
For each $s \in [k,k+1]$ let $r_n(s)$ denote the smallest positive real root of the polynomial $p_{n,s}(x):=\sum_{i=0}^n a_n(s)x^i$. Then there exists $n_0 \in \mathbb{N}$ such that $r_n(s)$ is well-defined for all $s \in [k,k+1]$ and $n\geq n_0$, and we have
\[\left|e^{P(A_1,\ldots,A_N;s)}-\frac{1}{r_n(s)}\right| \leq K\exp\left(-\gamma n^\alpha\right)\]
for some constants $K,\gamma>0$ not depending on $s \in [k,k+1]$, where 
\[\alpha:=\frac{{d+1\choose k+1}-1}{{d+1\choose k+1}-2}>1.\]
Suppose additionally that there is a norm $\vertle{\cdot}$ on $\mathbb{R}^d$ such that $\max_{1 \leq i \leq N} \vertle{A_i}<1$, and that $\dimaff (A_1,\ldots,A_N) \in (k,k+1)$. Then for all sufficiently large $n$ the function $s \mapsto 1/r_n(s)$ is strictly decreasing and convex on $[k,k+1]$ and there exists a unique $s_n \in [k,k+1]$ such that $r_n(s_n)=1$. 
There exist constants $K',\gamma'>0$ depending on $A_1,\ldots,A_N$ such that for all such $n$ we have
\[\left|\dim_{\mathsf{aff}} (A_1,\ldots,A_N)- s_n\right| \leq K'\exp\left(-\gamma' n^\alpha\right).\]
\end{theorem}
Since every matrix tuple is $0$-multipositive, in the case $k=0$ the hypothesis of Theorem \ref{th:main} reduces to the requirement that $\dimaff (A_1,\ldots,A_N)$ is $1$-multipositive and $\dimaff(A_1,\ldots,A_N) \in (0,1)$. Since $B^{\wedge 0}$ is the identity map on $\mathbb{R}$ the expressions involving $A_\iii^{\wedge k}$ reduce to $1$ in the case $k=0$, resulting in the formula
\[t_n(s):=\sum_{|\iii|=n} \frac{\lambda_1\left(A_\iii\right)^{d-1} \rho\left(A_\iii\right)^s}{p'_{A_\iii}\left(\lambda_1\left(A_\iii\right)\right) }.
\]
In particular when $d=2$, $k=0$ and the matrices $A_i$ have positive entries we may recover the conclusion of  Theorem \ref{th:one}. Similarly, since every tuple in $GL_d(\mathbb{R})^N$ is $d$-multipositive and $B^{\wedge d}=\det B$, the expressions involving $A_\iii^{\wedge(k+1)}$ simplify when $k=d-1$ yielding
\[t_n(s):=\sum_{|\iii|=n} \frac{\lambda_1\left(A_\iii^{\wedge (d-1)}\right)^{d-1} \rho\left(A_\iii^{\wedge(d-1)}\right)^{d-s} |\det A_\iii|^{s+1-d}}{p'_{A_\iii^{\wedge(d-1)}}\left(\lambda_1\left(A_\iii^{\wedge(d-1)}\right)\right) }.
\]
and the hypotheses are reduced to the requirement that $(A_1,\ldots,A_N)$ is $(d-1)$-multipositive and $\dimaff(A_1,\ldots,A_N)\in (d-1,d)$. We remark that hypotheses of domination and positivity analogous to those in Theorem \ref{th:main} have been a feature of numerous recent works on affine iterated function systems such as \cite{BaKa17,BaKaKo17,BaRa18,FaKe17,FaKe18} as well as the older article \cite{HuLa95}. 

If it is known that the tuple $(A_1^{\wedge k},\ldots,A_N^{\wedge k})$ preserves a single cone in $\wedge^k\mathbb{R}^d$ and similarly $(A_1^{\wedge (k+1)},\ldots,A_N^{\wedge (k+1)})$ preserves a single cone in $\wedge^{k+1}\mathbb{R}^d$ then the condition $\dimaff (A_1,\ldots,A_N) \in (k,k+1)$ may be easily checked. A theorem of V. Yu. Protasov \cite{Pr10} implies that if $B_1,\ldots,B_N$ preserve a cone then
\[\lim_{n \to \infty} \left(\sum_{i_1,\ldots,i_n=1}^N \left\|B_{i_1}\cdots B_{i_n}\right\|\right)^{\frac{1}{n}}=\rho\left(\sum_{i=1}^N B_i\right),\]
and so in this case
\[\lim_{n \to \infty} \left(\sum_{i_1,\ldots,i_n=1}^N \varphi^k\left(A_{i_1}\cdots A_{i_n}\right)\right)^{\frac{1}{n}}=\rho\left(\sum_{i=1}^N A_i^{\wedge k}\right),\]
\[\lim_{n \to \infty} \left(\sum_{i_1,\ldots,i_n=1}^N \varphi^{k+1}\left(A_{i_1}\cdots A_{i_n}\right)\right)^{\frac{1}{n}}=\rho\left(\sum_{i=1}^N A_i^{\wedge (k+1)}\right)\]
using the identity $\varphi^\ell(B)=\|B^{\wedge \ell}\|$ for $\ell=0,\ldots,d$. 
It follows that in this situation Theorem \ref{th:main} is applicable if
\[\rho\left(\sum_{i=1}^N A_i^{\wedge (k+1)}\right)<1<\rho\left(\sum_{i=1}^N A_i^{\wedge k}\right).\]
An example of this situation is presented in \S\ref{se:apps} below.

In the situation where $(A_1,\ldots,A_N)$ fails to be both $k$- and $(k+1)$-multipositive we believe it to be unlikely that any analogue of Theorem \ref{th:main} can be proved. The precise role of the multipositive hypothesis is discussed in more detail in the following section, and in the final section \S\ref{se:nondom}.

\section{Overview of the method and statement of the main technical theorem}
  
The method underlying Theorem \ref{th:main} is, like Theorem \ref{th:one}, based on Fredholm determinants of transfer operators, and in broad terms resembles many other arguments of this type such as \cite{JePo02,JePo05,JePo16,Po10,PoJe00,PoVy16,PoWe08}. Both in order to give a sense of the organisation of this article and to indicate those complications present in the proof of Theorem \ref{th:main} which do not occur in the context of Theorem \ref{th:one} let us briefly describe this strategy. For simplicity we will specialise our description to the situation in which the transfer operators act on a Hilbert space, although this is not a strict requirement.

 We recall that an operator $\mathscr{L}$ on an infinite-dimensional Hilbert space is called trace-class if the sequence of approximation numbers
\[\mathfrak{s}_n(\mathscr{L}):=\inf\left\{\|\mathscr{L}-\mathscr{F}\| \colon \text{rank }\mathscr{F}<n\right\}\]
is summable; we observe in particular that such an operator is compact (being a limit in the norm topology of a sequence of finite-rank operators) and cannot be invertible. We also observe that clearly $\mathfrak{s}_n(\mathscr{L}^\ell) \leq \|\mathscr{L}^{\ell-1}\|  \mathfrak{s}_n(\mathscr{L})$ for every $n,\ell \geq 1$ and consequently every power of a trace-class operator is also trace-class. The notion of trace-class operator is reviewed in detail for the reader's convenience in \S\ref{se:prelims}.
 Suppose then that  $\mathscr{H}$ is a separable complex Hilbert space and $\mathscr{L} \colon \mathscr{H} \to \mathscr{H}$ a trace-class linear operator, and let $(\lambda_\ell)_{\ell=1}^\infty$ be the sequence of nonzero eigenvalues of $\mathscr{L}$ listed with repetition according to their algebraic multiplicity. (If only $M<\infty$ nonzero eigenvalues exist then define $\lambda_\ell=0$ for $\ell>M$.) 
 It is a classical fact that the function $z \mapsto \det(I-z\mathscr{L})$ which may be defined\footnote{The Fredholm determinant is more usually defined first by its power series and shown later to equal the infinite product given here, see e.g. \cite{Si79}; we adopt this characterisation for simplicity of presentation and because of its more direct connection with the problems being studied.} by
\[\det(I-z\mathscr{L}):=\prod_{\ell=1}^M \left(1-z\lambda_\ell\right)\]
is an entire function from $\mathbb{C}$ to $\mathbb{C}$, and moreover one may show that in the power series $\det(I-z\mathscr{L})=\sum_{n=0}^\infty a_\ell z^\ell$ the coefficients are given by $a_0=1$ and
\[a_\ell= \frac{(-1)^\ell}{\ell!}\det\begin{pmatrix}
\tr \mathscr{L} & \ell -1&  0&\cdots &0 &0\\
\tr \mathscr{L}^2&\tr \mathscr{L} &\ell-2 &\cdots &0 &0\\
\tr \mathscr{L}^3&\tr \mathscr{L}^2&\tr \mathscr{L}  &\ddots &0 &0\\
\vdots & \vdots & \vdots & \ddots  &\ddots& \vdots\\
\tr \mathscr{L}^{\ell-1} &\tr \mathscr{L}^{\ell-2}&\tr \mathscr{L}^{\ell-3}&\cdots &\tr \mathscr{L}  &1\\
\tr \mathscr{L}^{\ell} &\tr \mathscr{L}^{\ell-1}&\tr \mathscr{L}^{\ell-2}&\cdots &\tr \mathscr{L}^2 &\tr \mathscr{L} 
\end{pmatrix}\]
for $\ell \geq 1$. If we write
\[\sum_{\ell=0}^\infty a_\ell z^\ell=\det(I-z\mathscr{L})=\prod_{\ell=1}^M \left(1-z\lambda_\ell\right)\]
then by equating coefficients of $z^n$ we find (at least informally) that also
\begin{equation}\label{eq:fast}a_n= (-1)^n\sum_{i_1<i_2<\cdots <i_n} \lambda_{i_1}\cdots \lambda_{i_n}. \end{equation}
for each $n \geq 1$. Suppose now that we wished to calculate the spectral radius $\rho(\mathscr{L})$, knowing the values of the traces $\mathscr{L}^\ell$ for $\ell=1,\ldots,n$, say, and knowing also that the spectral radius is an eigenvalue of $\mathscr{L}$. The roots of $\det(I-z\mathscr{L})$ are precisely the reciprocals of the eigenvalues of $\mathscr{L}$ and therefore $\rho(\mathscr{L})^{-1}$ is the smallest positive root of $\sum_{\ell=0}^\infty a_\ell z^\ell$. In particular, the smallest positive root of $\sum_{\ell=0}^n a_\ell z^\ell$ should be a good approximation to $\rho(\mathscr{L})^{-1}$ as long as $\sum_{\ell=n+1}^\infty |a_\ell|$ is small. But if we are able to show that the eigenvalues $(\lambda_n)$ decay exponentially (or even just stretched-exponentially) in $n$, then the expression \eqref{eq:fast} implies a super-exponential decay estimate for the coefficients $a_n$. Such an estimate will hold in particular if the approximation numbers of $\mathscr{L}$ decay stretched-exponentially. In such a situation we may therefore reasonably hope that the approximation procedure just outlined provides an estimate which becomes super-exponentially more accurate as $n$ increases.

In order to implement this line of reasoning we need therefore to construct, for each $s \in [k,k+1]$, a trace-class operator $\mathscr{L}_s$ on a Hilbert space $\mathscr{H}$ such that $e^{P(A_1,\ldots,A_N;s)}$ is an eigenvalue of $\mathscr{L}_s$ and is equal to the spectral radius of $\mathscr{L}_s$, such that $\mathscr{L}_s$ is trace-class, such that the sequence of approximation numbers of $\mathscr{L}_s$ decays rapidly to zero, and such that the sequence of traces $\tr \mathscr{L}_s^n$ is easy to compute. Once such a family of operators has been constructed the result follows by relatively straightforward manipulations which, while they do not correspond precisely to any prior work, share a degree of familial resemblance with calculations occurring in numerous earlier articles such as \cite{BaJePo14,JePo01,JePo02,JePo05,JePo07,JePo16,JePoVy17,KaPo15,Po11,PoFe14,PoJe00,PoVy15,PoVy16,PoWe08}. 

If $V$ is a finite-dimensional real vector space let $PV$ denote the real projective space of lines through the origin in $V$. Intuitively, in order to construct an operator $\mathscr{L}_s$ with spectral radius $e^{P(A_1,\ldots,A_N;s)}$, we might consider an operator acting on some space of continuous functions $P(\wedge^k\mathbb{R}^d)\times P(\wedge^{k+1}\mathbb{R}^d) \to \mathbb{C}$ defined by
\[\left(\mathscr{L}_sf\right)(\overline{u},\overline{v}) = \sum_{i=1}^N \left(\frac{\left\|A^{\wedge k}_iu\right\|}{\|u\|}\right)^{k+1-s}\left(\frac{\left\|A^{\wedge(k+1)}_iv\right\|}{\|v\|}\right)^{s-k} f\left(\overline{A^{\wedge k}_iu},\overline{A^{\wedge (k+1)}_iv}\right)\]
where for $v \in V$ the notation $\overline{v}$ represents the one-dimensional subspace spanned by the vector $v$. Since we would then have
\[\left(\mathscr{L}_s^nf\right)(\overline{u},\overline{v}) = \sum_{|\iii|=n} \left(\frac{\left\|A^{\wedge k}_\iii u\right\|}{\|u\|}\right)^{k+1-s}\left(\frac{\left\|A^{\wedge(k+1)}_\iii v\right\|}{\|v\|}\right)^{s-k} f\left(\overline{A^{\wedge k}_\iii u},\overline{A^{\wedge (k+1)}_\iii v}\right)\]
for each $n\geq 1$ we might then reasonably expect that
\[\lim_{n \to \infty} \left\|\mathscr{L}_s^n\right\|^{\frac{1}{n}}= \lim_{n\to \infty}\left(\sum_{|\iii|=n} \left\|A^{\wedge k}_\iii\right\|^{k+1-s}\left\|A^{\wedge(k+1)}_\iii\right\|^{s-k}\right)^{\frac{1}{n}}= \lim_{n\to \infty}\left(\sum_{|\iii|=n}\varphi^s(A_\iii)\right)^{\frac{1}{n}}\]
so that $e^{P(A_1,\ldots,A_N;s)}$ is equal to the spectral radius of $\mathscr{L}_s$. Indeed, such operators were successfully constructed by Guivarc'h and Le Page on spaces of H\"older continuous functions $P(\wedge^k\mathbb{R}^d)\times P(\wedge^{k+1}\mathbb{R}^d) \to \mathbb{C}$ in the article \cite{GuLe04}.

However, notwithstanding the (rather minor) additional complications posed by the fact that the spaces defined above are not Hilbert, there is no reason to believe that $\mathscr{L}_s$ acting on such a space should have a summable sequence of approximation numbers $\mathfrak{s}_n(\mathscr{L}_s)$. Indeed, $\mathscr{L}_s$ as constructed is equal to a sum of weighted composition operators $f \mapsto g\cdot f \circ T$ where $T$ is an invertible transformation of $P(\wedge^k\mathbb{R}^d)\times P(\wedge^{k+1}\mathbb{R}^d)$ and $g$ is nowhere zero. Such an operator might reasonably be expected to be invertible, and there is certainly no reason to believe that $\mathscr{L}_s$ should be trace-class.

The problem is thus to define $\mathscr{L}_s$ approximately as above in such a way that it is a sum of trace-class, non-invertible operators. It is here that the hypothesis of $k$- and $(k+1)$-multipositivity becomes relevant: this hypothesis implies that for $\ell=k,k+1$ the matrices $A_1^{\wedge \ell},\ldots,A_N^{\wedge \ell}$ map a finite union of patches of $P(\wedge^\ell \mathbb{R}^d)$ strictly inside itself. By taking $\mathscr{H}$ to be a suitable Hilbert space of functions defined only on the patches, composition with the projective action of the matrices should then induce an operator which is non-invertible and hopefully trace-class. It transpires that composition operators on spaces of holomorphic functions are reliably trace-class subject to moderate geometrical conditions, and as such our strategy will involve passing to a space of holomorphic functions defined on complex extensions of the patches in real projective space. Once we have verified that such an extension can be constructed in such a way that the operator $\mathscr{L}_s$ is well-defined on the patches we may proceed to prove Theorem \ref{th:main} along the lines outlined above. 

In the two-dimensional context of Theorem \ref{th:one} the construction of these complex patches is very straightforward. Since Theorem \ref{th:one} is restricted to affine transformations whose linear parts contract the positive cone in $\mathbb{R}^2$, it is sufficient to consider the projective action of those linear maps on the interval $\{(x,1-x)\colon x \in [0,1]\}$, which is an action by linear fractional transformations. A finite collection of linear fractional transformations each of which maps an interval strictly inside itself can easily be shown to also map a corresponding complex disc inside itself, and this complex disc can be used as the domain of the holomorphic functions on which the operator $\mathscr{L}_s$ acts. In higher dimensions and using multicones instead of cones, the corresponding problem is to understand (in place of one-dimensional intervals) a family of $(d-1)$-dimensional sections of cones in $\mathbb{R}^d$ -- in effect, a finite collection of arbitrary $(d-1)$-dimensional convex bodies -- and a collection of linear fractional transformations between them, and to contrive a system of extensions of those convex bodies into $\mathbb{C}^{d-1}$ which is also preserved by the same family of linear fractional transformations. This much more involved procedure is undertaken in \S\ref{se:complex-cones} and lays the foundation for following technical theorem which is obtained subsequently: 
\begin{theorem}\label{th:opter}
Let $d,N \geq 2$ and let $(A_1,\ldots,A_N)\in M_d(\mathbb{R})^N$ be both $k$-multipositive and $(k+1)$-multipositive, where $0\leq k<d$.  Then there exist a separable complex Hilbert space $\mathscr{H}$ and a family of bounded linear operators $\mathscr{L}_s \colon \mathscr{H} \to \mathscr{H}$ defined for all $s \in \mathbb{C}$ with the following properties:
\begin{enumerate}[(i)]
\item
There exist $C,\kappa,\gamma>0$ such that for all $s \in \mathbb{C}$ and $n \geq 1$ we have $\mathfrak{s}_n(\mathscr{L}) \leq C\exp\left(\kappa|s|-\gamma n^\beta\right)$, where
\[\beta:=\frac{1}{{d+1 \choose k+1}-2} \in (0,1].\]
In particular each $\mathscr{L}_s$ is trace-class.
\item
For every $s \in \mathbb{C}$ and $n \geq 1$ we have
\[\tr \mathscr{L}_s^n=\sum_{|\iii|=n} \frac{\lambda_1\left(A_\iii^{\wedge k}\right)^{{d \choose k} -1} \lambda_1\left(A_\iii^{\wedge\left(k+1\right)}\right)^{{d \choose k+1} -1} \rho\left(A_\iii^{\wedge k}\right)^{k+1-s} \rho\left(A_\iii^{\wedge \left(k+1\right)}\right)^{s-k} }{ p'_{A_\iii^{\wedge k}}\left(\lambda_1\left(A_\iii^{\wedge k}\right)\right)p'_{A_\iii^{\wedge \left(k+1\right)}}\left(\lambda_1\left(A_\iii^{\wedge \left(k+1\right)}\right)\right) }
\]
where $p_B(x):=\det(xI-B)$ denotes the characteristic polynomial of $B$ and $p'_B(x_0)$ its derivative evaluated at $x_0$.
\item
For every $s \in \mathbb{R}$ the spectral radius of $\mathscr{L}_s$ is equal to
\[\lim_{n \to \infty} \frac{1}{n}\log \sum_{|\iii|=n} \left\|A_{\iii}^{\wedge k}\right\|^{k+1-s}  \left\|A_{\iii}^{\wedge (k+1)}\right\|^{s-k}.\]
In particular the above limit exists for all $s \in \mathbb{R}$, and for every $s \in [k,k+1]$ the spectral radius of $\mathscr{L}_s$ is equal to $e^{P(A_1,\ldots,A_N;s)}$. For all $s \in \mathbb{R}$ the spectral radius of $\mathscr{L}_s$ is a simple eigenvalue of $\mathscr{L}_s$ and there are no other eigenvalues of the same modulus.
\end{enumerate}
\end{theorem}
Theorem \ref{th:opter} is a special case of a slightly more general result, Theorem \ref{th:topaff}, which will be proved later. Theorem \ref{th:topaff} is also applied in the sequel article \cite{Mo20} to the estimation of a related invariant of tuples of matrices.

The remainder of this article is structured as follows. In \S\ref{se:complex-cones} we undertake the construction of the complex extensions of the patches in real projective space. We then review in \S\ref{se:prelims} the properties of trace-class operators which will be needed in this article and extend a standard result from this context in view of the fact that we will be working with spaces of holomorphic functions defined on a non-connected region. We then proceed in \S\ref{se:operators} to establish the properties of the operator $\mathscr{L}_s$ and deduce Theorem \ref{th:opter}. In \S\ref{se:proofs} we derive Theorem \ref{th:main} from Theorem \ref{th:opter} above. Some examples of the application of Theorem \ref{th:main} are presented in \S\ref{se:apps}. In \S\ref{se:nondom} we consider the problem of calculating the affinity dimension in situations where the hypotheses of Theorem \ref{th:main} do not apply.

We remark that sections \ref{se:proofs}--\ref{se:nondom} depend only on the statement of Theorem \ref{th:opter} and the material presented in sections 1 and 2 and as such may be read independently of sections \ref{se:complex-cones}--\ref{se:operators} in which the proof of Theorem \ref{th:opter} is prepared for and presented.


\section{Complex domains for linear semigroups acting on a multicone}\label{se:complex-cones}

Our first task in proving Theorem \ref{th:main} is to translate the matter from the context of linear maps between real cones to the context of holomorphic maps between complex domains. We will prove the following:
\begin{theorem}\label{th:multicones}
Let $d \geq 1$ and let $(\mathcal{K}_1,\ldots,\mathcal{K}_m)$, $(\mathcal{K}_1',\ldots,\mathcal{K}_{m}')$ be multicones in $\mathbb{R}^d$, both with transverse-defining vector $w \in \mathbb{R}^d$, such that $\mathcal{K}_j' \setminus\{0\} \subset \Int \mathcal{K}_j$ for each $j=1,\ldots,m$. Define
\[\mathfrak{A}:=\left\{A \in M_d(\mathbb{R}) \colon A\left(\bigcup_{j=1}^m \mathcal{K}_j\right) \subseteq \bigcup_{j=1}^{m} \left(\mathcal{K}_j' \cup -\mathcal{K}_j'\right)\right\}\]
and let $\mathfrak{A}^*$ denote the set of all nonzero elements of $\mathfrak{A}$. We observe that $\mathfrak{A}$ is a semigroup.

Then there exists a subset $\Omega$ of the complex hyperplane $\{z \in \mathbb{C}^d \colon \langle z,w\rangle=1\}$ such that the following properties are satisfied by $\mathfrak{A}^*$ and $\Omega$:
\begin{enumerate}[(i)]
\item\label{it:aly}
There is a constant $\tau>0$ such that $\|A_1 A_2 \| \geq \tau \|A_1\|\cdot\|A_2\|$ for every $A_1,A_2 \in \mathfrak{A}$. In particular $\mathfrak{A}^*$ is a subsemigroup of $\mathfrak{A}$.
\item\label{it:omega}
The set $\Omega$ is open and bounded and is symmetric with respect to complex conjugation. Every connected component of $\Omega$ intersects $\mathbb{R}^{d_i}$. The closures of the connected components of $\Omega$ are disjoint.
\item\label{it:relower}
There exists $C>0$ such that 
\[C^{-1}\|A\|\leq \left|\Re(\langle Az,w\rangle)\right| \leq |\langle Az,w\rangle| \leq C \|A\|\]
for all $A \in \mathfrak{A}$ and $z \in \Omega$.
\item\label{it:induce}
Every $A \in \mathfrak{A}^*$ induces a well-defined holomorphic transformation $\phi_A \colon \Omega \to \Omega$ defined by $\phi_A(z):=\langle Az,w\rangle^{-1}Az$. The set 
\[\overline{\bigcup_{A \in \mathfrak{A}^* } \phi_A(\Omega)}\]
is a compact subset of $\Omega$.
\item\label{it:contract}
There exist a metric $\mathsf{d}$ on $\Omega$ which is bi-Lipschitz equivalent to the standard metric and a constant $\theta \in (0,1)$ such that $\mathsf{d}(\phi_A(z_1),\phi_A(z_2)) \leq \theta\mathsf{d}(z_1,z_2)$ for every $A \in \mathfrak{A}^*$.
\item\label{it:det}
Let $A \in \mathfrak{A}^*$. Then the largest eigenvalue $\lambda_1(A)$ of $A$ is algebraically simple, is real, is strictly larger in modulus than all of the other eigenvalues of $A$, and has a corresponding eigenvector $z_A \in \Omega \cap \mathbb{R}^d$ which is the unique fixed point of $\phi_A \colon \Omega \to \Omega$. The eigenvalues of the derivative $D_{z_A}\phi_A$ are precisely the numbers $\lambda_j(A)/\lambda_1(A)$ for $j=2,\ldots,d$, and in particular
\[\det (I-D_{z_A}\phi_A)= \frac{p_A'(\lambda_1(A))}{\lambda_1(A)^{d-1}} \neq 0\]
where $p_A(x):=\det (xI-A)$ denotes the characteristic polynomial of $A$ and $p_A'$ its first derivative.
\end{enumerate}
\end{theorem}
Theorem \ref{th:multicones} is trivial in the case $d=1$ and for the remainder of this section we shall ignore this case, assuming at all times that $d \geq 2$. (When $d=1$ the determinant in (\ref{it:det}) above will be interpreted as being equal to $1$.) Here and throughout the remainder of this article we use the notation $z^*$ to denote the complex conjugate of $z \in \mathbb{C}$ and reserve the notation $\overline{z}$ for the one-dimensional subspace spanned by $z$.

Using the machinery of complex cones and gauges (see \cite{Du09,Ru10}) it is possible to obtain Theorem \ref{th:multicones} by extending each real cone $\mathcal{K}_j$ to a complex cone
\[\mathcal{K}_j^{\mathbb{C}}:=\left\{\lambda((u+v)+i(u-v))\colon \lambda \in \mathbb{C}\text{ and }u,v \in \mathcal{K}_j\right\}\]
and considering the projective action on a slice through the complex extension of the union of the cones $\mathcal{K}_1,\ldots,\mathcal{K}_m$,
\[\Omega:=\left\{z \in \mathbb{C}^d \colon z \in \bigcup_{j=1}^m \Int \mathcal{K}_j^{\mathbb{C}} \text{ and }\langle z,w\rangle=1\right\}.\]
This procedure has the advantage of explicitness and may be a useful direction of research in the event that effective versions of Theorem \ref{th:main} are sought. It is on the other hand somewhat laborious to implement, and since our interest is only in establishing the correctness of the formulas in Theorem \ref{th:main} and giving a super-exponential bound for the error term, we pursue a simpler but less explicit construction along the lines of \cite[\S2]{BaJe08b}. 

\subsection{The action on the real multicone}

We begin by establishing some preliminary results concerning the action of $\mathfrak{A}$ on the real cones $\mathcal{K}_1,\ldots,\mathcal{K}_m$ and proceed to prove Theorem \ref{th:multicones} in the following subsection. 
\begin{lemma}\label{le:cones-elementary}
Let $d \geq 1$ and let $(\mathcal{K}_1,\ldots,\mathcal{K}_m)$, $(\mathcal{K}_1',\ldots,\mathcal{K}_{m}')$ be multicones in $\mathbb{R}^d$, both with transverse-defining vector $w \in \mathbb{R}^d$, such that $\mathcal{K}_j'\setminus \{0\} \subseteq  \Int\mathcal{K}_j$ for each $j=1,\ldots,m$. Define
\[\mathfrak{A}:=\left\{A \in M_d(\mathbb{R}) \colon A\left(\bigcup_{j=1}^m \mathcal{K}_j\right) \subseteq \bigcup_{j=1}^{m} \left(\mathcal{K}_j' \cup -\mathcal{K}_j'\right)\right\}\]
and observe that $\mathfrak{A}$ is a semigroup. Then there exists $\tau \in (0,1]$ such that:
\begin{enumerate}[(i)]
\item
For every $u \in \bigcup_{j=1}^m  \mathcal{K}_j$ we have $\tau \|u\| \leq \langle u,w\rangle \leq \|u\|$.
\item
For every $A \in \mathfrak{A}$ and $u \in \bigcup_{j=1}^m  \mathcal{K}_j'$ we have $\|Au\| \geq \tau \|A\|\cdot \|u\|$.
\item
For every $A_1,A_2 \in \mathfrak{A}$ we have $\|A_1 A_2\|\geq \tau \|A_1\|\cdot\|A_2\|$. In particular the set of all nonzero elements of $\mathfrak{A}$ is a subsemigroup of $\mathfrak{A}$.
\end{enumerate}
\end{lemma}
\begin{proof}
We will allow the constant $\tau>0$ to be different in each of (i),(ii) and (iii), which obviously suffices. To prove (i) it is sufficient, by homogeneity, to consider only those cases in which $\|u\|=1$. The function $u \mapsto \langle u,w\rangle$ is obviously continuous on the set of all $u \in \bigcup_{j=1}^m  \mathcal{K}_j$ such that $\|u\|=1$ and is positive everywhere on this set by the definition of a multicone. Since this set is compact this function attains its minimum, so this minimum is positive; call it $\tau$. We have $0<\tau \leq \langle u,w\rangle \leq 1$ for all $u \in \bigcup_{j=1}^m  \mathcal{K}_j$ with $\|u\|=1$ and the result follows.

By homogeneity in $A$ and $u$ it is sufficient to prove (ii) in the case $\|A\|=\|u\|=1$. By a similar compactness argument it suffices to show that $Au$ may not be zero when $A \in \mathfrak{A}$, $u \in \bigcup_{j=1}^m  \mathcal{K}_j'$ and $\|A\|=\|u\|=1$. For a contradiction suppose that we may find such $A$ and $u$ satisfying $Au=0$. Since $A \neq 0$ there exists a unit vector $v$ such that $Av \neq 0$. Since $u$ is a nonzero element of some $\mathcal{K}_j'$ it is an interior point of the corresponding cone $\mathcal{K}_j$ and therefore there exists $\epsilon>0$ such that $u+\epsilon v$ and $u-\epsilon v$ both belong to $\mathcal{K}_j$. But this implies that $A(u+\epsilon v)=\epsilon Av$ and $A(u-\epsilon v)=-\epsilon Av$ are both nonzero elements of $A\mathcal{K}_j$. Since $A\mathcal{K}_j \subseteq \mathcal{K}_i \cup -\mathcal{K}_i$ for some $i$ we deduce that $Av \in (\mathcal{K}_i \cap -\mathcal{K}_i) \setminus \{0\}$ contradicting the definition of a multicone. The result follows. To deduce (iii) we observe that for any unit vector $u \in \bigcup_{j=1}^m \mathcal{K}_j'$ we have
\[\|A_1A_2\| \geq \|A_1A_2u\| \geq \tau \|A_1\|\cdot \|A_2u\| \geq \tau^2 \|A_1\|\cdot \|A_2\|\cdot\|u\|= \tau^2 \|A_1\|\cdot \|A_2\|\]
 by repeated application of (ii).
\end{proof}
The following Perron-Frobenius result does not follow in a completely direct manner from standard statements of the Perron-Frobenius theorem for cones since it is possible for $A(\mathcal{K}_i\setminus\{0\})$ to include the zero vector, preventing the direct use of off-the-shelf results.
\begin{lemma}\label{le:pft}
Let $d$, $w$, $(\mathcal{K}_1,\ldots,\mathcal{K}_m)$, $(\mathcal{K}_1',\ldots,\mathcal{K}_m')$ and $\mathfrak{A}^*$ be as in the statement of Theorem \ref{th:multicones}. Suppose that $A \in \mathfrak{A}^*$ satisfies $A\mathcal{K}_i \subseteq \mathcal{K}_i'$ for some $i \in \{1,\ldots,m\}$. Then $\rho(A)$ is an algebraically simple eigenvalue of $A$ with corresponding eigenvector in $\mathcal{K}_i'$ and all other eigenvalues of $A$ are of strictly smaller absolute value.
\end{lemma}
\begin{proof}
Choose a cone $\mathcal{K}_i''$ such that $\mathcal{K}_i''\setminus \{0\} \subset \Int \mathcal{K}_i$  and $\mathcal{K}_i' \setminus \{0\} \subseteq \Int \mathcal{K}_i''$. We observe that the single matrix $A$, the one-element multicone $(\mathcal{K}_i)$ and the one-element multicone $(\mathcal{K}_i'')$ together satisfy the hypotheses of Lemma \ref{le:cones-elementary}, and by part (ii) of that lemma it follows that $Av$ is not the zero vector for any nonzero $v \in \mathcal{K}_i''$. In particular $A(\mathcal{K}_i''\setminus\{0\})\subseteq \mathcal{K}_i'\setminus\{0\} \subseteq \Int\mathcal{K}_i''$ and standard versions of the Perron-Frobenius Theorem such as \cite[Theorem 1.3.26]{BePl94} may be applied to the action of $A$ on $\mathcal{K}_i''$. The result follows.\end{proof}

\begin{proposition}\label{pr:real}
Let $d$, $w$, $(\mathcal{K}_1,\ldots,\mathcal{K}_m)$, $(\mathcal{K}_1',\ldots,\mathcal{K}_m')$ and $\mathfrak{A}^*$ be as in the statement of Theorem \ref{th:multicones}. Then there exist $C>0$ and $\theta \in (0,1)$ such that for all $j=1,\ldots,m$ and $n \geq 1$, for all nonzero $v_1,v_2 \in \mathcal{K}_j'$,
\[\sup_{A_1,\ldots,A_n \in \mathfrak{A}^*} \left\|\frac{A_1\cdots A_nv_1}{\langle A_1\cdots A_nv_1,w\rangle} -\frac{A_1\cdots A_nv_1}{\langle A_1\cdots A_nv_1,w\rangle}\right\| \leq C\theta^n\|v_1-v_2\|.\]
\end{proposition}
\begin{proof}
For every nonzero $v \in \mathbb{R}^d$ let $\overline{v}$ denote the one-dimensional subspace of $\mathbb{R}^d$ spanned by $v$, and let $\mathcal{K}_j/\sim$ denote the set of one-dimensional subspaces spanned by an element of $\Int \mathcal{K}_j$. For each $j=1,\ldots,m$ define
\[\alpha(v_1,v_2):=\sup\left\{\lambda \geq 0 \colon v_2-\lambda v_1 \in \mathcal{K}_j\right\}\]
and
\[\beta(v_1,v_2):=\inf\left\{\lambda \geq 0 \colon \lambda v_1- v_2 \in \mathcal{K}_j\right\}\]
for all $v_1,v_2 \in \Int\mathcal{K}_j$; then the formula
\[d_{\mathcal{K}_j}(\overline{v_1},\overline{v_2}):=\log\frac{\beta(v_1,v_2)}
{\alpha(v_1,v_2)}\]
defines a metric on $\mathcal{K}_j/\sim$ called the Hilbert projective metric.  It follows from Lemma \ref{le:cones-elementary}(i) that the set of all $v \in \bigcup_{j=1}^m \mathcal{K}_j$ such that $\langle v,w\rangle=1$ is bounded. By compactness it follows that there exists $\varepsilon \in (0,1]$ such that for every $j=1,\ldots,m$, if $v \in \mathcal{K}_j'$ with $\langle v,w\rangle=1$ then the open Euclidean $\varepsilon$-ball centred at $v$ is a subset of $\mathcal{K}_j$. We deduce that if $v_1,v_2 \in \mathcal{K}_j'$ with $\langle v_1,w\rangle=\langle v_2,w\rangle=1$ then since $\|v_1\|$, $\|v_2\|\leq \tau^{-1}$ by Lemma \ref{le:cones-elementary}(i) we have $\alpha(v_1,v_2)\geq \varepsilon\tau$ and $\beta(v_1,v_2) \leq \varepsilon^{-1}\tau^{-1}$, and hence the quantity
\[\Delta:=\max_{1 \leq j \leq m} \sup_{\substack{v_1,v_2 \in \mathcal{K}_j'\\ v_1,v_2 \neq 0}} d_{\mathcal{K}_j}(\overline{v_1},\overline{v_2}) \]
is finite. In particular if $v_1,v_2 \in \mathcal{K}_j\setminus \{0\}$ for some $j \in \{1,\ldots,m\}$, and $A \in \mathfrak{A}^*$, then $d_{\mathcal{K}_{i}}(\overline{Av_1},\overline{Av_2})\leq \Delta$ where $i$ is the unique integer such that $A\mathcal{K}_j \subseteq \mathcal{K}_{i} \cup -\mathcal{K}_{i}$. It follows by e.g. \cite[Theorem 1.1]{Li95} that if $A \in \mathfrak{A}^*$, $\overline{v_1},\overline{v_2} \in \mathcal{K}_j/\sim$ and $A\mathcal{K}_j \subseteq \mathcal{K}_i \cup -\mathcal{K}_i$ then we have $d_{\mathcal{K}_i}(\overline{Av_1},\overline{Av_2}) \leq \theta d_{\mathcal{K}_j}(\overline{v}_1,\overline{v_2})$ where $\theta:=\tanh (\Delta/4) \in (0,1)$.

We claim that there exists $C_1>0$ such that if $v_1,v_2 \in \mathcal{K}_j'$ with $\langle v_1,w\rangle=\langle v_2,w\rangle=1$ then
\[C_1^{-1} \|v_1-v_2\|\leq \left(e^{d_{\mathcal{K}_j}(\overline{v_1},\overline{v_2})}-1\right)\leq C_1\|v_1-v_2\|.\]
Indeed, given such vectors $v_1,v_2 \in \mathcal{K}_j'$ with $v_1 \neq v_2$ let $\alpha:=\alpha(v_1,v_2)$ and $\beta:=\beta(v_1,v_2)$. Since $\mathcal{K}_j$ is closed the supremum in the definition of $\alpha$ is attained, and therefore we have $v_2-\alpha v_1 \in \mathcal{K}_j$. Similarly we have $\beta v_1-v_2 \in \mathcal{K}_j$. From the maximality of $\alpha$ and the minimality of $\beta$ it follows that  $v_2-\alpha v_1$ and $\beta v_1-v_2$ are boundary points of $\mathcal{K}_j$. Since $\langle v_1-v_2,w\rangle=0$ and $v_1-v_2\neq 0$ neither $v_1-v_2$ nor $v_2-v_1$ can belong to $\mathcal{K}_j$, so neither $\alpha$ nor $\beta$ may equal $1$ and we deduce that $\alpha<1<\beta$. 

To obtain the first of the two claimed inequalities we observe that $\beta v_1-v_2$ and $(1-\alpha)v_1$ belong to $\mathcal{K}_j$, where we have used $0<\alpha<1$. Hence $(\beta-\alpha+1)v_1-v_2 \in \mathcal{K}_j$. If $\tau>0$ is as given by Lemma \ref{le:cones-elementary}(i) then we have  
\begin{align*}\tau \|v_1-v_2\| &\leq \tau \|(\beta-\alpha+1)v_1-v_2\| + \tau \|(\beta-\alpha)v_1\|\\
&\leq |\langle (\beta-\alpha+1)v_1-v_2,w\rangle| + (\beta-\alpha)|\langle v_1,w\rangle|\\
&\leq 2(\beta-\alpha)< 2\left(\frac{\beta}{\alpha}-1\right)=2\left(e^{d_{\mathcal{K}_j}(\overline{v_1},\overline{v_2})}-1\right)\end{align*}
where we have again used $0<\alpha<1$ in the final line. This yields the first inequality. To obtain the second inequality define $u_1:=\frac{1}{\beta-1}(\beta v_1-v_2)$ and $u_2:=\frac{1}{1-\alpha}(v_2-\alpha v_1)$. We observe that both $u_1$ and $u_2$ belong to the boundary of $\mathcal{K}_j$, which implies $\|u_1-v_1\|, \|u_2-v_2\| \geq \varepsilon\tau$ by the definition of $\varepsilon$ and the bound $\|v_1\|$, $\|v_2\|\leq \tau^{-1}$. We now observe that 
\begin{align*}e^{d_{\mathcal{K}_j}(\overline{v_1},\overline{v_2}) }=\frac{\beta}{\alpha}&=\frac{\|u_1-v_2\|\cdot \|u_2-v_1\|}{\|u_1-v_1\|\cdot \|u_2-v_2\|}\\
& \leq \left(\frac{\|u_1-v_1\|+\|v_1-v_2\|}{\|u_1-v_1\|}\right)\left(\frac{\|u_2-v_2\|+\|v_2-v_1\|}{\|u_2-v_2\|}\right)\\
&\leq\left(1+\varepsilon^{-1}\tau^{-1}\|v_1-v_2\|\right)^2\end{align*}
and therefore
\[e^{d_{\mathcal{K}_j}(\overline{v_1},\overline{v_2})}-1 \leq \left(\frac{2}{\varepsilon\tau} + \frac{1}{\varepsilon^2\tau^2}\|v_1-v_2\|\right) \|v_1-v_2\|\leq\left(\frac{4}{\varepsilon^2\tau^3}\right)\|v_1-v_2\|\]
where we have again used $\|v_1\|$,$\|v_2\| \leq \tau^{-1}$. The claim follows.

We may now prove the proposition. Given $n \geq 1$, nonzero $v_1,v_2 \in \mathcal{K}_{j}'$ and $A_1,\ldots,A_n \in \mathfrak{A}^*$, let $i\in \{1,\ldots,m\}$ be the integer such that $A_n\cdots A_1 \mathcal{K}_{j} \subseteq \mathcal{K}_{i}\cup -\mathcal{K}_i$. We have
\begin{align*}\left\|\frac{A_n\cdots A_1v_1}{\langle A_n \cdots A_1v_1,w\rangle}-\frac{A_n\cdots A_1v_2}{\langle A_n \cdots A_1v_2,w\rangle} \right\| &\leq C_1\left(e^{d_{\mathcal{K}_{i}}(\overline{A_n\cdots A_1v_1},\overline{A_n\cdots A_1v_2})}-1\right) \\
&\leq C_1\left(e^{\theta^n d_{\mathcal{K}_{j}}(\overline{v_1},\overline{v_2})}-1\right) \\
&\leq C_1^2\theta^n\|v_1-v_2\|\end{align*}
and the proposition is proved.
\end{proof}
While Proposition \ref{pr:real} will provide us with a vital contraction estimate for maps between specific cones $\mathcal{K}_j$, in order to apply it we will need the following combinatorial lemma which allows us to reduce the action of a specific matrix product on the multicone to that on a single cone:
\begin{lemma}\label{le:combo}
Let $d$, $w$, $(\mathcal{K}_1,\ldots,\mathcal{K}_m)$, $(\mathcal{K}_1',\ldots,\mathcal{K}_m')$ and $\mathfrak{A}^*$ be as in the statement of Theorem \ref{th:multicones}. Let $k \geq 2^m-m-1$. Then for every $A_1,\ldots,A_{k}\in\mathfrak{A}^*$ there exists $i \in \{1,\ldots,m\}$ such that
\[A_{k}\cdots A_1 \left(\bigcup_{j=1}^m \mathcal{K}_j\right) \subseteq \mathcal{K}_i \cup -\mathcal{K}_i.\]
\end{lemma}
\begin{proof}
It is clearly sufficient to consider the case $k=2^m-m-1$ only. Let $\mathcal{I}_0:=\{1,\ldots,m\}$ and for each $n=1,\ldots,2^m-m-1$ let $\mathcal{I}_n$ denote the intersection of all sets $\mathcal{I}\subseteq \{1,\ldots,n\}$ such that 
\begin{equation}\label{eq:ttt}A_{n}\cdots A_1 \left(\bigcup_{j=1}^m \mathcal{K}_j \right) \subseteq \bigcup_{i \in \mathcal{I}} \left(\mathcal{K}_i \cup -\mathcal{K}_i\right)\end{equation}
where the union over an empty set of indices $i$ is understood to be $\{0\}$. We observe that $\mathcal{I}=\mathcal{I}_n$ itself satisfies \eqref{eq:ttt}. By Lemma \ref{le:cones-elementary}(iii) the product $A_{2^m-m-1}\cdots A_1$ is not the zero matrix and therefore $\mathcal{I}_n$ is nonempty. We observe that the cardinality of $\mathcal{I}_n$ is non-increasing as a function of $n$.

We wish to prove that $\mathcal{I}_{2^m-m-1}$ has cardinality $1$, so for a contradiction let us suppose that its cardinality is at least $2$. This implies that every preceding $\mathcal{I}_n$ also has cardinality at least $2$, and also that $m\geq 2$. Since the number of subsets of $\{1,\ldots,m\}$ with cardinality at least $2$ is $2^m-m-1$, by the pigeonhole principle there exist integers $n_1,n_2$ with $0 \leq n_1<n_2 \leq 2^m-m-1$ such that $\mathcal{I}_{n_1}=\mathcal{I}_{n_2}$. The matrix $B:=A_{n_2}\cdots A_{n_1+1}$ therefore takes each cone $\mathcal{K}_i$ such that $i \in \mathcal{I}_{n_1}$ to a nontrivial subset of some cone $\mathcal{K}_j$ such that $j \in \mathcal{I}_{n_1}$, inducing a permutation on the elements of $\mathcal{I}_{n_1}=\mathcal{I}_{n_2}$. It follows that the matrix $B^{\#\mathcal{I}_{n_1}}$ induces the identity permutation on $\mathcal{I}_{n_1}$: for every $i \in \mathcal{I}_{n_1}$ we have $B^{\#\mathcal{I}_{n_1}} \mathcal{K}_i \subseteq (\mathcal{K}_i \cup -\mathcal{K}_i)$. Hence  $B^{2\#\mathcal{I}_{n_1}} \mathcal{K}_i \subseteq \mathcal{K}_i'$ for every $i \in \mathcal{I}_{n_1}$. By Lemma \ref{le:pft}, for every $i \in \mathcal{I}_{n_1}$ the matrix $B^{2\#\mathcal{I}_{n_1}}$ has a simple positive leading eigenvalue with a one-dimensional eigenspace which intersects $\mathcal{K}_i$ nontrivially: but since $\#\mathcal{I}_{n_1}\geq 2$ and distinct cones $\mathcal{K}_i$ do not intersect this implies that the leading eigenvalue is not simple, which is a contradiction.
\end{proof}

\subsection{Proof of Theorem \ref{th:multicones}}

Throughout the proof we fix $d$, $\mathfrak{A}$, $(\mathcal{K}_1,\ldots,\mathcal{K}_m)$, $(\mathcal{K}_1',\ldots,\mathcal{K}_{m}')$ and $w$ as in the statement of the theorem. Part (i) of Theorem \ref{th:multicones} follows directly from Lemma \ref{le:cones-elementary} so we concentrate on parts (ii) to (vi).

Define $H:=\{z \in \mathbb{C}^d \colon \langle z,w\rangle=1\}$ and let $K_j:=\mathcal{K}_j'\cap H$ for each $j=1,\ldots,m$. Each $K_j$ is closed by definition and is bounded as a consequence of Lemma \ref{le:cones-elementary}(i).  For each $n \geq 1$ define
\[\mathfrak{A}_n^*:=\left\{A_1\cdots A_n \colon A_1,\ldots,A_n \in \mathfrak{A}^*\right\}.\]
Define a function $M \colon H \to [0,+\infty)$ by
\[M(z):=\inf\left\{|\langle Az,w\rangle| \colon A \in \mathfrak{A}\text{ and }\|A\|=1\right\}.\]
Clearly $M(z)$ is well-defined and
\[M(z)=\inf\left\{\|A\|^{-1} |\langle Az,w\rangle| \colon A \in \mathfrak{A}^*\right\}.\]
We observe that $M$ is $1$-Lipschitz continuous: given $z_1,z_2 \in U$ and $A \in \mathfrak{A}$ with $\|A\|=1$ we have
\[|\langle Az_2,w\rangle| \geq |\langle Az_1,w\rangle|- |\langle A(z_1-z_2),w\rangle| \geq M(z_1)-\|z_1-z_2\|\]
and taking the infimum over $A$ and rearranging easily yields $M(z_1)-M(z_2) \leq \|z_1-z_2\|$. The result follows by symmetry. The set $U :=\left\{z \in H \colon M(z) \neq 0\right\}$ is consequently open. We have $\bigcup_{j=1}^m K_j\subseteq U$ by Lemma \ref{le:cones-elementary}(i) and (ii) and in particular $U$ is nonempty.

We now claim that if $A \in \mathfrak{A}^*$ and $z \in U$ then necessarily $\langle Az,w\rangle^{-1}Az \in U$. If this is not the case for some $A$ and $z$ then by compactness there exists $B \in \mathfrak{A}^*$ with $\|B\|=1$ such that $\langle B(\langle Az,w\rangle^{-1}Az),w\rangle=0$, but then necessarily $\langle BAz,w\rangle=0$ which contradicts $z \in U$ since obviously $BA \in \mathfrak{A}^*$ by the semigroup property of $\mathfrak{A}^*$. The claim is proved. We deduce that for every nonzero $A \in \mathfrak{A}$ the formula $\phi_A(z):=\langle Az,w\rangle^{-1}Az$ gives rise to a well-defined holomorphic function $\phi_A \colon U \to U$. We observe that $\phi_A \circ \phi_B = \phi_{AB}$ for all $A,B \in \mathfrak{A}^*$ and that $\phi_{tA}=\phi_A$ for all real $t>0$ and all $A \in \mathfrak{A}^*$.

Let $\tau>0$ be as given by Lemma \ref{le:cones-elementary} and observe that 
\[\sup\left\{\|z\| \colon z \in \bigcup_{j=1}^m K_j\right\}\leq \tau^{-1}\]
by Lemma \ref{le:cones-elementary}(i) and
\[\inf\left\{\|A\|^{-1}|\Re(\langle Az,w\rangle)|   \colon z \in  \bigcup_{j=1}^m K_j\text{ and }A \in \mathfrak{A}^*\right\} \geq \tau^2\]
by Lemma \ref{le:cones-elementary}(i) and (ii). For each $j=1,\ldots,m$ define
\[U_j:=\left\{z \in U \colon \inf_{\omega \in K_j} \|z-\omega\|<\epsilon\right\}\]
where $\epsilon>0$ is chosen small enough that the following properties hold: the sets $U_j$ have pairwise disjoint closures; $|\Re(\langle Az,w\rangle)|\geq \frac{\tau^2}{2}\|A\|$ and $\|z\| \leq 2\tau^{-1}$ for all $z \in \bigcup_{j=1}^m U_j$ and all $A \in \mathfrak{A}$; and
\begin{equation}\label{eq:omegaholds}(256\tau^{-10} + 4\tau^{-4})\epsilon<\frac{1}{4}.\end{equation}
The second condition is possible since the function $z \mapsto \inf\{\|A\|^{-1}|\Re(\langle Az,w\rangle)| \colon A \in \mathfrak{A}^*\}$ is $1$-Lipschitz continuous for the same reasons as $M$. Each $K_j$ is convex as a consequence of the definition of a multicone, so each $U_j$ is convex also.

Now let $C_1,\theta_1$ be the constants given by Proposition \ref{pr:real} and let $n_1 \geq 1$ be large enough that $C_1\theta^{n_1}_1<\frac{1}{4}$.  We claim that for every $A \in \mathfrak{A}_{n_1}^*$ the map $\phi_A$ satisfies $\|D_z\phi_A\|\leq \frac{1}{2}$ for all $z \in \bigcup_{j=1}^m U_j$. Fix $A \in \mathfrak{A}_{n_1}^*$ and observe that $\|D_\omega\phi_A\|\leq C_1\theta^{n_1}_1<\frac{1}{4}$ for all $\omega \in \bigcup_{j=1}^m K_j$ by Proposition \ref{pr:real}. 

By simple direct calculation, for all $v \in \mathbb{C}^d$ such that $\langle v,w\rangle=0$ and all $z \in U$ we have 
\[(D_z\phi_A)(v)=\langle Az,w\rangle^{-2} \left(\langle Az,w\rangle Av - \langle Av,w\rangle Az \right).\]
It follows that if $z_1,z_2 \in U_j$ then
\begin{eqnarray*}{\lefteqn{(D_{z_1}\phi_A-D_{z_2}\phi_A)(v)}} & &\\
&=  &  \langle Az_1,w\rangle^{-2}\langle Az_2,w\rangle^{-2}\left(\langle Az_2,w\rangle^2-\langle Az_1,w\rangle^2\right)\left(\langle Az_1,w\rangle Av - \langle Av,w\rangle Az_1\right)\\
& &+\langle Az_2,w\rangle^{-2} \left(\langle A(z_1-z_2),w\rangle Av - \langle Av,w\rangle A(z_1-z_2)\right)  
\end{eqnarray*}
for all $v$ in the tangent space $\{v \in \mathbb{C}^d \colon \langle z,w\rangle=0\}$. Since $|\langle Az,w\rangle|^{-2} \leq 4\tau^{-4}\|A\|^{-2}$ for all $z \in U_j$ by the definition of $U_j$, this yields the estimate
\begin{align*}\|D_{z_1}\phi_A - D_{z_2}\phi_A\| &\leq (16\tau^{-8}(\|z_1\|+\|z_2\|)(2\|z_1\|)  + 4\tau^{-4})\|z_1-z_2\|\\
& \leq (256\tau^{-10}+4\tau^{-4})\|z_1-z_2\|\end{align*}
where we have used the bound $\|z\| \leq 2\tau^{-1}$ which applies to all $z \in U_j$. In particular if $z \in U_j$ is arbitrary and $\omega \in K_j$ is chosen such that $\|z-\omega\|<\epsilon$, taking $z_1:=z$ and $z_2:=\omega$ and applying \eqref{eq:omegaholds} together with $\|D_\omega \phi_A\|<\frac{1}{4}$ yields $\|D_z\phi_A\|<\frac{1}{2}$. We conclude that $\max_{1 \leq j \leq m}\sup_{z \in U_j} \|D_z\phi_A\|\leq\frac{1}{2}$. Since each $U_j$ is convex it follows by the mean value inequality that for every $A \in \mathfrak{A}^*_{n_1}$ and $j=1,\ldots,m$ the map $\phi_A \colon U_j \to U$ is $\frac{1}{2}$-Lipschitz continuous with respect to the Euclidean metric. (We observe that this does not imply $\frac{1}{2}$-Lipschitz continuity on $\bigcup_{j=1}^m U_j$.) It follows that for every $k \geq 1$ and $A \in \mathfrak{A}^*_{kn_1}$ the map $\phi_A$ is $\frac{1}{2^k}$-Lipschitz on each $U_j$, which will be used later.

We next observe that for every $A \in \mathfrak{A}^*$ and $z_1,z_2 \in U$ there holds the Lipschitz continuity estimate
\begin{equation}\label{eq:ll}\|\phi_A(z_1)-\phi_A(z_2)\| \leq M(z_1)^{-1}M(z_2)^{-1} \min\{\|z_1\|,\|z_2\|\}\|z_1-z_2\|.\end{equation}
Clearly it is sufficient to prove this in the case $\|A\|=1$. In this case we observe that
\begin{align*}\left\|\phi_A(z_1)-\phi_A(z_2)\right\|&= |\langle Az_1,w\rangle\langle Az_2,w\rangle|^{-1}\left\|\langle Az_2,w\rangle Az_1 - \langle Az_1,w\rangle Az_2\right\| \\
& \leq M(z_1)^{-1}M(z_2)^{-1}\left\|\langle Az_2,w\rangle Az_1 - \langle Az_1,w\rangle Az_2\right\| \\
&= M(z_1)^{-1}M(z_2)^{-1} \left\|\langle Az_2,w\rangle A(z_1-z_2) + \langle A(z_2-z_1),w\rangle Az_2\right\|\\
&\leq M(z_1)^{-1}M(z_2)^{-1} \|z_2\|\cdot\|z_1-z_2\|,\end{align*}
and performing the same calculation with $z_1$ and $z_2$ interchanged obviously yields \eqref{eq:ll}. As a consequence we have
\begin{equation}\label{eq:lip2}
\|\phi_A(z_1)-\phi_A(z_2)\| \leq C_2\|z_1-z_2\|
\end{equation}
for every $z_1,z_2 \in \bigcup_{j=1}^m U_j$ and $A \in \mathfrak{A}^*$ where $C_2:=8\tau^{-5}$, using the inequalities $M(z) \geq \tau^2/2$ and $\|z\|\leq 2\tau^{-1}$ which follow from the definition of the sets $U_j$.

Let $n_2\geq 2^m-m-1$ be an integer such that for every $A \in \mathfrak{A}^*$ the map $\phi_A|_{U_j}$ is $1$-Lipschitz continuous for every $j=1,\ldots,m$. (Note that every sufficiently large multiple of $n_1$ has this property.) Fix $k$ large enough that $C_22^{-k} \leq \frac{1}{2}$ and define a metric $\mathsf{d}$ on $U$ by
\[\mathsf{d}(z_1,z_2):=\sum_{n=0}^{kn_1+n_2-1} 2^{\frac{n}{kn_1+n_2}} \sup_{A \in \mathfrak{A}^*_n} \left\|\phi_{A}(z_1)-\phi_{A}(z_2) \right\|\]
where the summand corresponding to $n=0$ is understood as $\|z_1-z_2\|$. It follows from \eqref{eq:ll} that $\mathsf{d}(z_1,z_2)$ is well-defined for all $z_1,z_2 \in U$ and its property of being a metric is obvious. For $z_1,z_2 \in \bigcup_{j=1}^m U_j$ we additionally have
\[\|z_1-z_2\| \leq \mathsf{d}(z_1,z_2) \leq  \left(\frac{C_2}{2^{\frac{1}{kn_1+n_2}}-1}\right)\|z_1-z_2\|\]
by applying \eqref{eq:lip2} and summing the geometric series, so $\mathsf{d}$ is bi-Lipschitz equivalent to the Euclidean distance when considered as a metric on $\bigcup_{j=1}^m U_j$. We observe that since every $A \in \mathfrak{A}^*$ is real, the metric $\mathsf{d}$ is symmetric with respect to complex conjugation: $\mathsf{d}(z_1,z_2)=\mathsf{d}(z_1^*,z_2^*)$ for all $z_1,z_2 \in U$.

We claim that for every $z_1,z_2 \in \bigcup_{j=1}^m U_j$ and $B \in \mathfrak{A}^*$ we have 
\begin{equation}\label{eq:doohickey}\mathsf{d}(\phi_B(z_1),\phi_B(z_2)) \leq 2^{-\frac{1}{kn_1+n_2}} \mathsf{d}(z_1,z_2).\end{equation}
To see this let $z_1,z_2 \in \bigcup_{j=1}^m U_j$ and $B \in \mathfrak{A}^*$. We have
\begin{eqnarray*}\mathsf{d}(\phi_B(z_1),\phi_B(z_2))&=&\sum_{n=0}^{kn_1+n_2-1} 2^{\frac{n}{kn_1+n_2}} \sup_{A \in \mathfrak{A}^*_n} \left\|\phi_{A}(\phi_B(z_1))-\phi_{A}(\phi_B(z_2)) \right\|\\
&\leq& \sum_{n=1}^{kn_1+n_2} 2^{\frac{n-1}{kn_1+n_2}} \sup_{A \in \mathfrak{A}^*_n} \left\|\phi_{A}(z) - \phi_{A}(\omega)\right\|\\
&=& 2^{-\frac{1}{kn_1+n_2}}\sum_{n=1}^{kn_1+n_2-1} 2^{\frac{n}{kn_1+n_2}} \sup_{A\in \mathfrak{A}^*_n} \left\|\phi_{A}(z_1) - \phi_{A}(z_2)\right\|\\
& & + 2^{\frac{kn_1+n_2-1}{kn_1+n_2}} \sup_{A \in \mathfrak{A}^*_{kn_1+n_2}} \|\phi_A(z_1)-\phi_A(z_2)\|.\end{eqnarray*}
To prove the claimed inequality it therefore suffices to show that
\[\|\phi_A(z_1)-\phi_A(z_2)\|\leq \frac{1}{2} \|z_1-z_2\|\]
for all $A \in \mathfrak{A}^*_{kn_1+n_2}$, since then the final term above is bounded by $2^{-\frac{1}{kn_1+n_2}} \|z_1-z_2\|$ and simple rearrangement yields \eqref{eq:doohickey}. 
Now, if $A \in \mathfrak{A}^*_{kn_1+n_2}$ let us write $A=A_1A_2$ where $A_1 \in \mathfrak{A}^*_{kn_1}$ and $A_2 \in \mathfrak{A}^*_{n_2}$. By Lemma \ref{le:combo} there exists $i$ such that $A_2(\bigcup_{j=1}^m \mathcal{K}_j) \subseteq \mathcal{K}_i \cup -\mathcal{K}_i$, and this clearly implies  $A_2(\bigcup_{j=1}^m\mathcal{K}_j') \subseteq \mathcal{K}_i' \cup -\mathcal{K}_i'$. Hence $\phi_{A_2}(\bigcup_{j=1}^m K_j) \subseteq K_i$. Choose $j_1,j_2$ such that $z_1 \in U_{j_1}$ and $z_2 \in U_{j_2}$ and choose $\omega_1\in K_{j_1}$ and $\omega_2 \in K_{j_2}$ such that $\|z_1-\omega_1\|<\epsilon$ and $\|z_2-\omega_2\|<\epsilon$. We have $\phi_{A_2}(\omega_1), \phi_{A_2}(\omega_2) \in K_i$, $\|\phi_{A_2}(z_1)-\phi_{A_2}(\omega_1)\|<\epsilon$ by the $1$-Lipschitz continuity of $\phi_{A_2}$ restricted to $U_{j_1}$, and likewise $\|\phi_{A_2}(z_2)-\phi_{A_2}(\omega_2)\|<\epsilon$. Thus 
$\phi_{A_2}(z_1)$ and $\phi_{A_2}(z_2)$ are both elements of $U_i$ and they satisfy $\|\phi_{A_2}(z_1)-\phi_{A_2}(z_2)\| \leq C_2\|z_1-z_2\|$ by \eqref{eq:lip2}. But $\phi_{A_1}$ is $\frac{1}{2^k}$-Lipschitz when restricted to $U_i$, so
\begin{align*}\|\phi_{A_1}(\phi_{A_2}(z_1))-\phi_{A_1}(\phi_{A_2}(z_2))\| &\leq 2^{-k}\|\phi_{A_2}(z_1)-\phi_{A_2}(z_2)\|\\
& \leq C_22^{-k}\|z_1-z_2\|\leq \frac{1}{2}\|z_1-z_2\|\end{align*}
and the claim follows.

For each $j=1,\ldots,m$ define $\Omega_j$ to be the unique connected component of the set 
\begin{equation}\label{eq:omegapartialdefn}\left\{z \in U \colon \inf_{\omega \in K_j} \mathsf{d}(z,\omega)<\epsilon\right\}\end{equation}
which intersects $K_j$. Obviously we have $K_j \subset \Omega_j$. Since $\|z_1-z_2\|\leq \mathsf{d}(z_1,z_2)$ for all $z_1,z_2 \in U$ we also have $\Omega_j \subseteq U_j$ for each $j=1,\ldots,m$. Define $\Omega:=\bigcup_{j=1}^m \Omega_j$ and observe that $z \in \Omega$ if and only if $z^* \in \Omega$ by the fact that $\bigcup_{j=1}^m K_j \subset \mathbb{R}^d$ and the fact that $\mathsf{d}$ is conjugation-symmetric. Since every $U_j$ is bounded, so is every $\Omega_j$ and therefore so is $\Omega$. The connected components of $\Omega$ are precisely the sets $\Omega_j$ and these have disjoint closures since this is true of the sets $U_j$ which contain them. Each $\Omega_j$ contains the corresponding set $K_j \subset \mathbb{R}^d$ and in particular intersects $\mathbb{R}^d$. This completes the proof of (\ref{it:omega}). We have
\[\frac{\tau^2}{2}\|A\|\leq \left|\Re(\langle Az,w\rangle)\right| \leq |\langle Az,w\rangle| \leq  2\tau^{-1}\|A\|\]
for all $z \in \Omega$ and $A \in \mathfrak{A}$ as a consequence of the definition of $U_1,\ldots,U_m$, and this completes the proof of (\ref{it:relower}). Since for every $A \in \mathfrak{A}^*$ the function $\phi_A$ maps $\bigcup_{j=1}^m K_j$ to a subset of itself, and $\phi_A$ contracts distances between points in $\bigcup_{j=1}^m U_j$ with respect to $\mathsf{d}$ by a factor of $\theta_2:=2^{-1/(kn_1+n_2)}$, it follows that
\[\overline{\bigcup_{A \in \mathfrak{A}^*} \phi_A(\Omega)} \subseteq \left\{z \in U \colon \inf_{\omega \in \bigcup_{j=1}^m K_j} \mathsf{d}(z,\omega)\leq \theta_2\epsilon\right\}\]
which is a compact subset of the set defined in \eqref{eq:omegapartialdefn}. Each $\phi_A(\Omega_j)$ is a connected subset of the set defined above and intersects one of the sets $K_i$, hence it is a subset of the set defined in \eqref{eq:omegapartialdefn} and intersects $K_i$, hence is a subset of the corresponding set $\Omega_i$, hence is a subset of $\Omega$. We conclude that $\overline{\bigcup_{A \in \mathfrak{A}^*} \phi_A(\Omega)}$ is a compact subset of $\Omega$. This completes the proof of (\ref{it:induce}) and (\ref{it:contract}). 

It remains only to prove (\ref{it:det}). Fix $A \in \mathfrak{A}^*$. Since $A^{n_2} \in \mathfrak{A}_{n_2}^*$ the matrix $A^{n_2}$ maps $\bigcup_{j=1}^m \mathcal{K}_j$ into $(\mathcal{K}_i\cup -\mathcal{K}_i)$ for some $i \in \{1,\ldots,m\}$ and in particular $A^{2n_2}$ maps $\mathcal{K}_i$ into $\mathcal{K}_i'$. It follows by Lemma \ref{le:pft} that $A^{2n_2}$ has an algebraically simple leading eigenvalue which is real and positive, has corresponding eigenvector $v_A$ in $\mathcal{K}_i'$ and is the unique eigenvalue with maximal modulus. Hence $A$ has an algebraically simple leading eigenvalue $\lambda_1(A)$ which is real (but may be negative), is the unique eigenvalue of maximal modulus, and satisfies $A v_A =\lambda_1(A)v_A$. Defining $z_A:=\langle v_A,w\rangle^{-1}v_A$ we have $z_A \in K_i \subset \Omega \cap \mathbb{R}^d$. Obviously $\phi_Az_A = z_A$ and $\langle A z_A,w\rangle = \lambda_1(A)$. By (\ref{it:contract}) there can be no other fixed points for $\phi_A$ in $\Omega$. 

Let us now calculate the eigenvalues of the derivative $D_{z_A}\phi_A$. Let $u_1,\ldots,u_d \in \mathbb{C}^d$ be a Jordan basis for  $A$ with basis elements listed in descending order of the absolute value of the corresponding eigenvalue, and with $u_1=z_A$. Since $|\lambda_1(A)|>|\lambda_2(A)|$ we have $A u_1=\lambda_1(A)u_1$ and $A u_2=\lambda_2(A)u_2$. For each $j \in \{3,\ldots,d\}$, let $\delta_j \in \{0,1\}$ such that $A u_j=\lambda_j(A )u_j+\delta_ju_{j-1}$. 

For every $v$ in the tangent space $\{v \in \mathbb{C}^d \colon \langle v,w\rangle=0\}$ to $\Omega$ at $z_A$ we have
\begin{align*}\left(D_{z_A}\phi_A\right) v&:= \lim_{\varepsilon \to 0} \frac{1}{\varepsilon}\left(\frac{A(u_1+\varepsilon v)}{\langle A(u_1 + \varepsilon v),w\rangle } - \frac{Au_1}{\langle A u_1,w\rangle } \right)\\
&=\frac{\langle A u_1,w\rangle\cdot A v- \langle A v,w\rangle\cdot A u_1 }{\langle Au_1 ,w\rangle \langle A u_1,w\rangle } \\
&=\frac{1}{\lambda_1(A)} \left(Av - \langle A v,w\rangle u_1 \right).\end{align*}
Clearly the vectors $v_j:=u_j-\langle u_j,w\rangle u_1$, where $j$ runs from $2$ to $d$, form a basis of the tangent space $\{z \in \mathbb{C}^d \colon \langle z,w\rangle=0\}$. We have 
\begin{align*}\left(D_{z_A}\phi_A\right)v_2
&=\frac{1}{\lambda_1(A)} \left(Av_2 - \langle Av_2,w \rangle u_1\right)\\
&=\frac{1}{\lambda_1(A)} \left(\lambda_2(A)u_2 -  \lambda_1(A)\langle u_2,w \rangle u_1  - \lambda_2(A)\langle u_2,w \rangle u_1 + \lambda_1(A)\langle u_2,w \rangle u_1\right)\\
&=\frac{1}{\lambda_1(A)} \left(\lambda_2(A)u_2 - \lambda_2(A)\langle u_2,w \rangle u_1\right)\\
&=\frac{\lambda_2(A)}{\lambda_1(A)} v_2,\end{align*}
and for $j=3,\ldots,d$ we similarly have
\begin{align*}\left(D_{z_A}\phi_A\right) v_j 
=&\frac{1}{\lambda_1(A)} \left(Av_j - \langle Av_j,w \rangle u_1\right)\\
=& \frac{1}{\lambda_1(A)} \big(\lambda_j(A)u_j + \delta_ju_{j-1} - \lambda_1(A)\langle u_j,w\rangle u_1\\
& - \lambda_j(A)\langle u_j,w\rangle u_1 -\delta_j\langle u_{j-1},w \rangle u_1 + \lambda_1(A)\langle u_j,w\rangle u_1\big)\\
=& \frac{1}{\lambda_1(A)} \left(\lambda_j(A)u_j -\lambda_j(A)\langle u_j,w\rangle u_1 + \delta_ju_{j-1}- \delta_j\langle u_{j-1},w \rangle u_1\right)\\
=&\frac{\lambda_j(A)}{\lambda_1(A)} v_j + \frac{\delta_j}{\lambda_1(A)} v_{j-1}.\end{align*}
It follows that with respect to the basis $v_2,\ldots,v_d$ the matrix of $D_{z_A}\phi_A$ is upper triangular with the values $\lambda_j(A)/\lambda_1(A)$ along the diagonal. In particular its eigenvalues are precisely the numbers $\lambda_j(A)/\lambda_1(A)$ for $j=2,\ldots,d$ as claimed. Since $p_A(x)=\det(xI-A)=\prod_{j=1}^d (x-\lambda_j(A))$ we have
\[p_A'(x)=\sum_{\ell=1}^d \prod_{\substack{1\leq j \leq d \\ j \neq \ell}} (x-\lambda_j(A))\]
and therefore
\[\frac{p_A'(\lambda_1(A))}{\lambda_1(A)^{d-1}} = \frac{\prod_{j=2}^d (\lambda_1(A)-\lambda_j(A))}{\lambda_1(A)^{d-1}} =\prod_{j=2}^d \left(1-\frac{\lambda_j(A)}{\lambda_1(A)}\right)=\det(I-D_{z_A}\phi_A).\]
Since $1-\lambda_j(A)/\lambda_1(A)$ is nonzero for all $j=2,\ldots,d$ this quantity is nonzero. This completes the proof of (\ref{it:det}) and hence of the theorem.


\section{Operator-theoretic preliminaries}\label{se:prelims}

In this section we collect some preliminary results which will underpin the construction of the operators $\mathscr{L}_s$ defined in Theorem \ref{th:opter}. 

\subsection{Bergman spaces}

If $\Omega \subset \mathbb{C}^k$ is open and nonempty the Bergman space $\mathcal{A}^2(\Omega)$ is defined to be the set of all holomorphic functions $f\colon \Omega \to \mathbb{C}$ such that the integral $\int_\Omega |f(z)|^2dV(z)$ is finite, where $V$ denotes $2k$-dimensional Lebesgue measure on $\mathbb{C}^k\simeq \mathbb{R}^{2k}$. The space $\mathcal{A}^2(\Omega)$ is a Hilbert space when equipped with the inner product $\langle f,g\rangle_{\mathcal{A}^2(\Omega)}:=\int_\Omega f(z)g(z)^*dV(z)$. In particular it is a closed subspace of the Hilbert space $L^2(\Omega)$ and is therefore separable. We note the following elementary estimate:
\begin{lemma}\label{le:supbound}
Let $\Omega \subseteq \mathbb{C}^k$ be a nonempty open set and let $K \subseteq \Omega$ be compact. Then there exists $C_K>0$ depending on $K$ such that $\sup_{z \in K}|f(z)| \leq C_K\|f\|_{\mathcal{A}^2(\Omega)}$ for every $f \in \mathcal{A}^2(\Omega)$.\end{lemma}
\begin{proof}
Choose $\varepsilon>0$ small enough that for every $z \in K$ the open ball $B_\varepsilon(z_0)$ is a subset of $\Omega$. By harmonicity we have
\begin{align*}|f(z_0)|^2 &=\left|\frac{1}{V(B_\varepsilon(z_0))} \int_{B_\varepsilon(z_0)} f(z)^2dV(z)\right|\\
&\leq \frac{1}{V(B_\varepsilon(z_0))}\int_{\Omega} |f(z)|^2dV(z)=\frac{1}{V(B_\varepsilon(z_0))} \|f\|_{\mathcal{A}(\Omega)}^2 = \frac{k!}{\pi^k\varepsilon^k} \|f\|_{\mathcal{A}(\Omega)}^2 \end{align*}
for all $f \in \mathcal{A}^2(\Omega)$ and $z_0 \in K$. 
\end{proof}
We observe in particular that for every $z \in \Omega$ the evaluation map $f \mapsto f(z)$ is a continuous linear functional $\mathcal{A}^2(\Omega)\to \mathbb{C}$.

In practice we will be interested in the case where $\Omega$ is a bounded open subset of an affine subspace of $\mathbb{C}^d$ rather than of $\mathbb{C}^d$ itself. Clearly the results of this section will apply equally well in that context with $k$ being equal to the dimension of the affine subspace of $\mathbb{C}^d$ of which $\Omega$ is an open subset.
\subsection{Trace-class operators}

We define the  \emph{singular values}  or \emph{approximation numbers} $\mathfrak{s}_n(\mathcal{L})$ of a bounded linear operator $\mathcal{L} \colon \mathcal{H}\to \mathcal{H}$ acting on a separable complex Hilbert space $\mathcal{H}$ to be the quantities
\[\mathfrak{s}_n(\mathcal{L}):=\inf\left\{\|\mathcal{L}-\mathcal{F}\|\colon \mathcal{F}\colon \mathcal{H} \to \mathcal{H} \text{ is bounded with rank at most }n-1\right\},\]
where $n$ ranges over the positive integers. If $\mathcal{L}$ is compact then the values $\mathfrak{s}_n(\mathcal{L})^2$ coincide with the sequence of eigenvalues of the self-adjoint operator $\mathcal{L}^*\mathcal{L}$ (see e.g. \cite[Theorem IV.2.5]{GoGoKr00}). If $\mathcal{L}$ satisfies $\sum_{n=1}^\infty \mathfrak{s}_n(\mathcal{L})<\infty$ then $\mathcal{L}$ is called \emph{trace-class}. Any trace-class operator is obviously the limit in the operator norm of a sequence of finite-rank operators and in particular is compact. It follows easily from the definition of $\mathfrak{s}_n$ that if $\mathcal{L}_1$ and $\mathcal{L}_2$ are bounded operators then $\mathfrak{s}_n(\mathcal{L}_1\mathcal{L}_2)$ and $\mathfrak{s}_n(\mathcal{L}_2\mathcal{L}_1)$ are both bounded by $\|\mathcal{L}_1\|\mathfrak{s}_n(\mathcal{L}_2)$ for every $n \geq 1$, and in particular the composition of a trace-class operator with a bounded operator is trace-class. In particular every power of a trace-class operator is trace-class.

The fundamental properties of the trace are summarised in the following result which combines several statements from \cite[\S3]{Si79}:
\begin{theorem}\label{th:lid}
Let $\mathcal{L}$ be a trace-class operator acting on a complex separable Hilbert space $\mathcal{H}$ and let $(\lambda_n)_{n=1}^M$ be a complete enumeration of the nonzero eigenvalues of $\mathcal{L}$, listed with repetition according to algebraic multiplicity, where $M \in \mathbb{N} \cup \{0,+\infty\}$. Then for every orthonormal basis $(e_n)_{n=1}^\infty$ of $\mathcal{H}$ we have
\begin{equation}\label{eq:lid}\sum_{n=1}^\infty \langle \mathcal{L}e_n,e_n\rangle = \sum_{n=1}^M \lambda_n\end{equation}
with both series being absolutely convergent. 
The common value of \eqref{eq:lid} is defined to be the \emph{trace} of $\mathcal{L}$ and is denoted $\tr \mathcal{L}$. 
\end{theorem}
 It is clear from the definition that $\mathfrak{s}_{2n-1}(\mathcal{L}_1+\mathcal{L}_2)\leq \mathfrak{s}_n(\mathcal{L}_1)+\mathfrak{s}_n(\mathcal{L}_2)$ for every pair of bounded linear operators $\mathcal{L}_1,\mathcal{L}_2 \colon \mathcal{H} \to \mathcal{H}$ and every $n \geq 1$. It follows easily that if $\mathcal{L}_1,\ldots,\mathcal{L}_k$ are trace-class operators on $\mathcal{H}$ then any finite linear combination $\sum_{i=1}^k a_i \mathcal{L}_i$ is also trace-class and satisfies
 \[\tr \sum_{i=1}^k a_i \mathcal{L}_i = \sum_{i=1}^k a_i \tr \mathcal{L}_i\]
 as a consequence of \eqref{eq:lid}.
 
 The following result also combines several statements from \cite[\S3]{Si79}, with the exception of the determinant formula for $a_n$ which may be found instead in, for example, \cite[Theorem 6.8]{Si77} or \cite[Theorem IV.5.2]{GoGoKr00}.
 \begin{theorem}\label{th:ofundity}
Let $\mathcal{L}$ be a trace-class operator on a separable complex Hilbert space $\mathcal{H}$ and let $(\lambda_n)_{n=1}^\infty$ be an enumeration of the nonzero eigenvalues of $\mathcal{L}$, repeated according to algebraic multiplicity. (If only $M<\infty$ nonzero eigenvalues exist then we define $\lambda_n:=0$ for all $n>M$.) For every $n \geq 1$ define
\[a_n:=(-1)^n \sum_{i_1<\cdots<i_n} \lambda_{i_1}(\mathcal{L})\cdots \lambda_{i_n}(\mathcal{L})\]
and define also $a_0:=1$.  Then the function
\[\det(I-z\mathcal{L}):=\sum_{n=0}^\infty a_nz^n\]
is well-defined and entire, and is equal to the absolutely convergent infinite product $\prod_{n=1}^\infty (1-z\lambda_n)$. The zeros of $z \mapsto \det(I-z\mathcal{L})$ are precisely the reciprocals of the nonzero eigenvalues of $\mathcal{L}$ and the order of each zero is equal to the algebraic  multiplicity of the corresponding eigenvalue. The coefficients $a_n$ satisfy
\[a_n=\frac{(-1)^n}{n!}\det\begin{pmatrix}
\tr \mathcal{L} & n -1&  0&\cdots &0 &0\\
\tr \mathcal{L}^2&\tr \mathcal{L}&n-2 &\cdots &0 &0\\
\tr \mathcal{L}^3&\tr \mathcal{L}^2 &\tr \mathcal{L} &\ddots &0 &0\\
\vdots & \vdots & \vdots & \ddots  &\ddots& \vdots\\
\tr \mathcal{L}^{n-1} &\tr \mathcal{L}^{n-2}&\tr \mathcal{L}^{n-3}&\cdots &\tr \mathcal{L} &1\\
\tr \mathcal{L}^n &\tr \mathcal{L}^{n-1}&\tr \mathcal{L}^{n-2}&\cdots &\tr \mathcal{L}^2 &\tr \mathcal{L}
\end{pmatrix},\]
and
\[\left|a_n\right| \leq \sum_{i_1<\dots<i_n}\mathfrak{s}_{i_1}(\mathcal{L})\cdots \mathfrak{s}_{i_n}(\mathcal{L})\]
for all $n\geq 1$.
\end{theorem}

 \subsection{Weighted composition operators on Bergman spaces}
 
 It has long been known that composition operators on Bergman spaces, and on other Banach spaces of holomorphic functions, are trace-class under mild conditions (see e.g. \cite{Gr55}). Historically most results in this context have assumed the set $\Omega \subset \mathbb{C}^k$ to be bounded and connected but in this article we will need to work with sets having multiple connected components. We will use the notation $\Omega_0 \Subset \Omega$ to mean that the closed set $\overline{\Omega_0}$ is a compact subset of the open set $\Omega$.
 
 The following result is a special case of \cite[Theorem 5.9]{BaJe08a}.
\begin{theorem}\label{th:banjen}
Let $\Omega \subseteq \mathbb{C}^k$ be a nonempty open set and let $\Omega_0\Subset \Omega$ be nonempty. Suppose that $\phi_1,\ldots,\phi_m \colon \Omega \to \Omega_0$ are holomorphic and $\psi_1,\ldots,\psi_m \colon \Omega \to \mathbb{C}$ are holomorphic and bounded. Then the operator $\mathcal{L} \colon \mathcal{A}^2(\Omega) \to \mathcal{A}^2(\Omega)$ given by
\[\left(\mathcal{L}f\right)(z):=\sum_{j=1}^m \psi_j(z) f(\phi_j(z))\]
is a well-defined bounded linear operator on $\mathcal{A}^2(\Omega)$, and there exist $C,\gamma>0$ depending only on $\Omega$ and $\Omega_0$ such that
\[\mathfrak{s}_n(\mathcal{L}) \leq C\left(\sum_{j=1}^m \sup_{z \in \Omega} |\psi_j(z)|\right)\exp\left(-\gamma n^{\frac{1}{k}}\right)\]
 for every $n \geq 1$. In particular $\mathcal{L}$ is trace class.
\end{theorem}

In this article we will need  to calculate explicitly the traces of a family of operators. The following result is a minor variation on a type of result appearing in work of D. Ruelle (\cite[Lemma 1]{Ru76}), D. Mayer (\cite[\S{III}]{Ma80} and remark following \cite[Corollary 7.11]{Ma91}), D. Fried (\cite[Lemma 5]{Fr86}) and other authors. The result may be proved easily by following the second, third and fourth paragraphs of the proof of \cite[Theorem 4.2]{BaJe08b}. 
\begin{theorem}\label{pr:banjo}
Let $\Omega \subset \mathbb{C}^k$ be a bounded, connected, nonempty open set and suppose that $\phi \colon \Omega \to \Omega$ is a holomorphic function such that $\phi(\Omega)\Subset \Omega$.  Let $\psi \colon \Omega \to \mathbb{C}$ be holomorphic and bounded. Then $\phi$ has a unique fixed point $z_0 \in \Omega$, the eigenvalues of the derivative $D_{z_0}\phi$ are all strictly less than $1$ in modulus, and the operator $\mathcal{L} \colon \mathcal{A}^2(\Omega) \to \mathcal{A}^2(\Omega)$ defined by $(\mathcal{L}f)(z):=\psi(z)f(\phi(z))$ is trace-class and has trace equal to $\psi(z_0)/ \det (I-D_{z_0}\phi)$.
\end{theorem}
Since we will in general need to study operators on Bergman spaces $\mathcal{A}^2(\Omega)$ for which $\Omega$ is not connected, we prove the following simple extension of Theorem \ref{pr:banjo} which does not seem to have been previously stated elsewhere:
\begin{theorem}\label{t:thrace}
Let $\Omega \subseteq \mathbb{C}^k$ be a bounded nonempty open set and suppose that $\phi \colon \Omega \to \Omega$ is a holomorphic function such that $\phi(\Omega)\Subset\Omega$.  Let $\psi \colon \Omega \to \mathbb{C}$ be holomorphic and bounded. Then the set of fixed points $\Fix \phi := \{z \in \Omega \colon \phi(z)=z\}$ is either finite or empty, and each connected component of $\Omega$ contains at most one fixed point of $\phi$. At each fixed point $z \in \Fix \phi$ the eigenvalues of the derivative $D_z\phi$ are all strictly less than $1$ in modulus. The operator $\mathcal{L} \colon \mathcal{A}^2(\Omega) \to \mathcal{A}^2(\Omega)$ defined by $(\mathcal{L}f)(z):=\psi(z)f(\phi(z))$ is trace-class and satisfies
\begin{equation}\label{eq:slart}\tr \mathcal{L} = \sum_{z \in \Fix \phi} \frac{\psi(z)}{ \det (I-D_{z}\phi)}.\end{equation}
Additionally, if $\Omega$ is connected then $\Fix \phi$ is a singleton.
\end{theorem}
\begin{proof}
The number of connected components of $\Omega$ is at most countably infinite since otherwise the separability of $\mathbb{C}^k$ would be contradicted. Let $(\Omega_m)_{m=1}^M$ be an enumeration of the connected components of $\Omega$ where $M \in \mathbb{N} \cup \{\infty\}$. For each $m$, by connectedness and continuity we have either $\phi(\Omega_m) \Subset \Omega_m$ or $\phi(\Omega_m) \cap \Omega_m=\emptyset$. In the former case there is a unique fixed point of $\phi$ in $\Omega_m$ and the derivative of $\phi$ at the fixed point has all eigenvalues strictly less than $1$ in modulus by Theorem \ref{pr:banjo}. In the latter case there is obviously no fixed point in $\Omega_m$. It follows in particular that if $m=1$ then $\Fix\phi$ is a singleton as required. Moreover we observe that $\Fix \phi$ consists entirely of isolated points, is closed, and is compact since it is contained in $\overline{\phi(\Omega)}$; it is therefore finite or empty, as required. 

The operator $\mathcal{L}$ meets the hypotheses of Theorem \ref{th:banjen} and hence is trace-class, so it remains to calculate its trace. 
For each integer $m$ such that $1 \leq m \leq M$ let $(f_{m,n})_{n=1}^\infty$ be an orthonormal basis for $\mathcal{A}^2(\Omega_m)$. Extend each $f_{m,n}$ to a function $\tilde{f}_{m,n}\colon \Omega \to \mathbb{C}$ by defining $\tilde{f}_{m,n}(z):=f_{m,n}(z)$ when $z\in \Omega_m$ and $\tilde{f}_{m,n}(z):=0$ otherwise. Clearly $(\tilde{f}_{m,n})$ is an orthonormal basis for $\mathcal{A}^2(\Omega)$, so by Theorem \ref{th:lid} we have
\begin{align}\label{eq:drsmelly}\tr \mathcal{L} = \sum_{m=1}^M\sum_{n=1}^\infty \langle \mathcal{L}\tilde{f}_{m,n},\tilde{f}_{m,n}\rangle_{\mathcal{A}^2(\Omega)}&=\sum_{m=1}^M \sum_{n=1}^\infty \int_{\Omega} \psi(z)\tilde{f}_{m,n}(\phi(z))\tilde{f}_{m,n}(z)^*dV(z)\\\nonumber
&=\sum_{m=1}^M \sum_{n=1}^\infty \int_{\Omega_m} \psi(z)\tilde{f}_{m,n}(\phi(z))\tilde{f}_{m,n}(z)^*dV(z)
 \end{align}
 using the fact that each $\tilde{f}_{n,m}$ is supported on $\Omega_m$, and these series are absolutely convergent. 
 
Let us evaluate the final term of \eqref{eq:drsmelly} by considering the contribution of each $m$. For $m$ such that $\phi(\Omega_m) \cap \Omega_m=\emptyset$ the integrand is clearly identically zero for every $n \geq 1$ and the contribution of that $m$ to the total is zero. On the other hand for each $m$ such that $\phi(\Omega_m) \Subset \Omega_m$ let us define $\mathcal{L}_m \colon \mathcal{A}^2(\Omega_m) \to \mathcal{A}^2(\Omega_m)$ by $(\mathcal{L}_mf)(z):=\psi(z)f(\phi(z))$. By Theorem \ref{pr:banjo} there is a unique fixed point $z_m$ of $\phi$ in $\Omega_m$, the operator $\mathcal{L}_m$ is trace-class and
\begin{align*}\frac{\psi(z_m)}{\det(I-D_{z_m}\phi)}=\tr \mathcal{L}_m &=\sum_{n=1}^\infty \langle \mathcal{L}_mf_{m,n},f_{m,n}\rangle_{\mathcal{A}^2(\Omega_m)} \\
&= \sum_{n=1}^\infty \int_{\Omega_m} \psi(z)f_{m,n}(\phi(z))f_{m,n}(z)^*dV(z)\\
&=\sum_{n=1}^\infty \int_{\Omega_m} \psi(z)\tilde{f}_{m,n}(\phi(z))\tilde{f}_{m,n}(z)^*dV(z).
 \end{align*}
We have shown that for all $m$
\[\sum_{n=1}^\infty \int_{\Omega_m} \psi(z)\tilde{f}_{m,n}(\phi(z))\tilde{f}_{m,n}(z)^*dV(z)=\sum_{z \in \Omega_m \cap \Fix \phi} \frac{\psi(z)}{\det(I-D_z\phi)}\]
and combining this with \eqref{eq:drsmelly} yields \eqref{eq:slart}.
\end{proof}

\subsection{An operator Perron-Frobenius theorem}

The last general functional-analytic result which we will require is the following: 
\begin{theorem}[Krasnoselski\u{\i}]\label{th:kr}
Let $\mathcal{X}$ be a real Banach space and $\mathcal{C} \subseteq \mathcal{X}$ a subset such that:
\begin{enumerate}[(i)]
\item
$\mathcal{C}$ is closed and convex and satisfies $\lambda\mathcal{C}=\mathcal{C}$ for all real $\lambda>0$,
\item
$\mathcal{C} \cap -\mathcal{C}=\{0\}$,
\item
The interior of $\mathcal{C}$ is nonempty.
\end{enumerate}
Suppose that $\mathcal{L} \colon \mathcal{X} \to \mathcal{X}$ is a compact linear operator which is \emph{strongly positive}: for every nonzero $x \in \mathcal{C}$ there exists $n \geq 1$ such that $\mathcal{L}^nx \in \Int\mathcal{C}$. Then $\rho(\mathcal{L})$ is nonzero and is a simple eigenvalue of $\mathcal{L}$ whose corresponding eigenspace intersects $\Int \mathcal{C}$. Moreover there exist no other eigenvalues of $\mathcal{L}$ with modulus $\rho(\mathcal{L})$.
\end{theorem}
\begin{proof}
The strong positivity of the operator permits the application of \cite[Theorem 2.5]{Kr64} which implies that there exists an eigenvector in the cone $\mathcal{C}$ with positive real eigenvalue $\lambda$; by strong positivity this eigenvector must belong to $\Int \mathcal{C}$.   In the terminology of \cite[\S2.1.1]{Kr64} the strong positivity of the operator $\mathcal{L}$ implies that $\mathcal{L}$ is $u_0$-positive for every $u_0 \in \Int \mathcal{C}$, so \cite[Theorem 2.10]{Kr64} may be applied to show that the eigenvalue $\lambda$ is simple and \cite[Theorem 2.13]{Kr64} shows that it is maximal in absolute value (hence equal to $\rho(\mathcal{L})$) and that no other eigenvalues of the same absolute value exist.
\end{proof}


\section{Proof of Theorem \ref{th:opter}}\label{se:operators}

We will follow \cite{BoMo18} in analysing the singular value function
\[\varphi^s(A_\iii)=\left\|A_\iii^{\wedge\lfloor s\rfloor}\right\|^{1+\lfloor s\rfloor-s} \left\|A_\iii^{\wedge\lceil s\rceil}\right\|^{s-\lfloor s\rfloor}\]
 by treating it as a product of the form $\prod_{j=1}^\ell \|A_\iii^{(j)}\|^{t_j}$ where $(A_1^{(1)},\ldots,A_N^{(1)})$, \ldots, $(A_1^{(\ell)},\ldots,A_N^{(\ell)})$ are \emph{a priori} unrelated tuples of matrices with respective dimensions $d_1,\ldots,d_\ell\geq 1$, essentially ignoring the fact that the two tuples $(A_1^{\wedge \lfloor s\rfloor},\ldots,A_N^{\wedge \lfloor s\rfloor})$ and $(A_1^{\wedge \lceil s\rceil},\ldots,A_N^{\wedge \lceil s\rceil})$ are related by the property of being exterior powers of the same tuple. Besides the established utility of this approach in \cite{BoMo18,Mo18}, we suspect that other results of a similar character such as \cite{FeSh14,GuLe04} could in principle be rewritten in these terms.

Theorem \ref{th:opter} is a special case of the following more general statement which will also be applied in \cite{Mo20}: 
\begin{theorem}\label{th:topaff}
Let $\ell \geq 1$ and $t=(t_1,\ldots,t_\ell) \in \mathbb{C}^\ell$, for $j=1,\ldots,\ell$ let $m_j,d_j \geq 1$, let $(A_1^{(j)},\ldots,A_N^{(j)}) \in M_{d_j}(\mathbb{R})$, and let $(\mathcal{K}_1^{(j)},\ldots,\mathcal{K}_{m_j}^{(j)})$ be a multicone with transverse-defining vector $w_j \in \mathbb{R}^{d_j}$. Suppose that not every $d_j$ is equal to $1$. Then there exists a bounded open subset $\Omega$ of the $\sum_{j=1}^{\ell} (d_j-1)$-dimensional affine space
\begin{equation}\label{eq:bun}\left\{z =\bigoplus_{j=1}^\ell z_j \in \bigoplus_{j=1}^\ell \mathbb{C}^{d_j} \colon \langle z_j,w_j\rangle=1\text{ for all }j=1,\ldots,\ell\right\}\end{equation}
such that the operator
\[\left(\mathscr{L}_{t}f\right)\left(\bigoplus_{j=1}^\ell z_j\right):=\sum_{i=1}^N \prod_{j=1}^\ell \left(\frac{\langle A_i^{(j)}  z_j,w_j\rangle}{\sign \Re(\langle A_i^{(j)}  z_j,w_j\rangle)}\right)^{t_j}  f\left(\bigoplus_{j=1}^\ell \langle A_i^{(j)}z_j,w_j\rangle^{-1}A_i^{(j)}z_j \right)\]
is a well-defined bounded linear operator on $\mathcal{A}^2(\Omega)$ and:
\begin{enumerate}[(i)]
\item
There exist constants $C,\kappa,\gamma>0$ such that the approximation numbers $\mathfrak{s}_n(\mathscr{L}_t)$ satisfy $\mathfrak{s}_n(\mathscr{L}_t) \leq C\exp (\kappa\|t\|-\gamma n^{1/(d_1+\cdots +d_\ell-\ell)})$ for every $n \geq 1$ and $t=(t_1,\ldots,t_\ell) \in \mathbb{C}^\ell$.
\item
For each $n \geq 1$ the trace of the operator $\mathscr{L}_t^n$ is equal to
\begin{equation}\label{eq:trarse}\sum_{|\iii|=n} \prod_{j=1}^\ell  \frac{\lambda_1\left(A_\iii^{(j)}\right)^{d_j-1} \rho\left(A_\iii^{(j)}\right)^{t_j}} {p_{A_\iii^{(j)}}'\left(\lambda_1\left(A_\iii^{(j)}\right)\right)}\end{equation}
where $p_B(x):=\det(xI-B)$ denotes the characteristic polynomial of $B \in M_d(\mathbb{R})$ and $p'_B(x_0)$ its derivative evaluated at $x_0$.
\item
If $t \in \mathbb{R}^\ell$ then  
\begin{equation}\label{eq:limit}\rho(\mathscr{L}_t)=\lim_{n \to \infty} \left(\sum_{|\iii|=n} \prod_{j=1}^m \left\|A_\iii^{(j)}\right\|^{t_j}\right)^{\frac{1}{n}}\end{equation}
and in particular this limit exists. Furthermore in this case $\rho(\mathscr{L}_t)$ is a simple eigenvalue of $\mathscr{L}_t$, and $\mathscr{L}_t$ has no other eigenvalues with modulus equal to $\rho(\mathscr{L}_t)$. 
\end{enumerate}
\end{theorem}
\begin{proof}[Proof of Theorem \ref{th:topaff}]
Fix $i \in \{1,\ldots,\ell\}$. Since $(A_1^{(i)},\ldots,A_N^{(i)})$ strictly preserves the multicone $(\mathcal{K}_1^{(i)},\ldots,\mathcal{K}_{m_i}^{(i)})$ with transverse-defining vector $w_i$ we may choose a multicone $(\hat{\mathcal{K}}_1^{(i)},\ldots,\hat{\mathcal{K}}_1^{(i)})$ with the same transverse-defining vector such that $\hat{\mathcal{K}}_j^{(i)} \setminus \{0\}\subseteq \Int\mathcal{K}_j^{(i)}$ for each $j=1,\ldots,m_i$ and such that $A_k^{(i)}(\bigcup_{j=1}^{m_i} \mathcal{K}_j^{(i)}) \subseteq \bigcup_{j=1}^{m_i} (\hat{\mathcal{K}}_j^{(i)}\cup -\hat{\mathcal{K}}_j^{(i)})$ for every $k=1,\ldots,N$. Let
\[\mathfrak{A}_{(i)}:=\left\{A \in M_{d_i}(\mathbb{R}) \colon A\left(\bigcup_{j=1}^m \mathcal{K}_j^{(i)}\right) \subseteq \bigcup_{j=1}^{m} \left(\hat{\mathcal{K}}_j^{(i)} \cup -\hat{\mathcal{K}}_j^{(i)}\right)\right\}\]
and let $\mathfrak{A}_{(i)}^*$ denote the set of all nonzero elements of $\mathfrak{A}_{(i)}$. By Theorem \ref{th:multicones} $\mathfrak{A}_{(i)}^*$ is a semigroup. Since obviously $A_1^{(i)},\ldots,A_N^{(i)} \in \mathfrak{A}_{(i)}^*$ we have $A_\iii^{(i)} \in \mathfrak{A}_{(i)}^*$ for every $\iii \in \Sigma_N^*$. 

Theorem \ref{th:multicones} implies that there exists a bounded open set $\Omega_i \subset \{z_i \in \mathbb{C}^{d_i} \colon \langle z_i,w_i\rangle=1\}$ such that for every $\iii \in \Sigma_N^*$ the map $\phi_\iii^{(i)} \colon \Omega_i \to \Omega_i$ defined by $\phi_\iii^{(i)}(z):=\langle A_\iii^{(i)}z_i,w_i\rangle^{-1}A_\iii^{(i)}z_i$ is well-defined. For each $\iii \in \Sigma_N^*$ we have $\Re(\langle A_\iii^{(i)}z,w_i\rangle)\neq 0$ for all $z_i \in \Omega_i$ by Theorem \ref{th:multicones}(\ref{it:relower}), so for each $\iii \in \Sigma_N^*$ the function $z_i \mapsto \sign \Re(\langle A_\iii^{(i)}z_i,w_i\rangle) \in \{\pm1\}$ is well-defined and is constant on every connected component of $\Omega_i$. In particular it is a holomorphic function on $\Omega_i$. For each $\iii \in \Sigma_N^*$ define
\[\psi_{\iii,t}^{(i)}(z_i):=\left(\frac{\langle A_\iii^{(i)}z_i,w_i\rangle}{\sign \Re(\langle A_\iii^{(i)}z_i,w_i\rangle)}\right)^{t_i}:= \exp\left(t_i \log \left(\frac{\langle A_\iii^{(i)}z_i,w_i\rangle}{\sign \Re(\langle A_\iii^{(i)}z_i,w_i\rangle)}\right)\right).\]
Since $\langle A_\iii^{(i)}z_i,w_i\rangle/\sign \Re(\langle A_\iii^{(i)}z_i,w_i\rangle)$ has positive real part for all $z_i \in \Omega_i$ its logarithm is a well-defined holomorphic function of $z_i \in \Omega_i$ and has imaginary part confined to the range $(-\frac{\pi}{2},\frac{\pi}{2})$ throughout $\Omega_i$. 

For all $z_i \in \Omega_i$ and $\iii \in \Sigma_N^*$ we have
\begin{equation}\label{eq:herpderp} \left|\Re\left( \log\left(\frac{\langle A_\iii^{(i)}z_i,w_i\rangle}{\sign \Re(\langle A_\iii^{(i)}z_i,w_i\rangle)}\right)\right)-\log \left\| A_\iii^{(i)}\right\|\right|\leq \log C_1\end{equation}
for some constant $C_1>1$ using Theorem \ref{th:multicones}(\ref{it:relower}), where $C_1$ may be chosen independent of $i \in \{1,\ldots,m\}$ by taking the maximum of its possible values as $i$ varies. Hence
 \begin{align*}
\Re\left( t_i\log\left(\frac{\langle A_\iii^{(i)}z_i,w_i\rangle}{\sign \Re(\langle A_\iii^{(i)}z_i,w_i\rangle)}\right)\right)=&\,\,\Re(t_i)\Re\left( \log\left(\frac{\langle A_\iii^{(i)}z_i,w_i\rangle}{\sign \Re(\langle A_\iii^{(i)}z_i,w_i\rangle)}\right)\right)\\
& - \Im(t_i)\Im\left( \log\left(\frac{\langle A_\iii^{(i)}z_i,w_i\rangle}{\sign \Re(\langle A_\iii^{(i)}z_i,w_i\rangle)}\right)\right)\\
\leq&\,\,\Re(t_i) \log\left\|A_\iii^{(i)}\right\| +|\Re(t_i)|\log C_1 + \frac{\pi}{2}\left|\Im(t_i)\right|\end{align*}
for all $z_i \in \Omega_i$ and $\iii \in \Sigma_N^*$ and therefore
\[\sup_{z_i \in \Omega_i} \left|\psi_{\iii,t}^{(i)}(z_i)\right| \leq \left(C_1e^{\pi/2}\right)^{|t_i|}\left\|A_\iii^{(i)}\right\|^{\Re(t_i)}\]
for all $\iii \in \Sigma_N^*$. 

Now define $\Omega:=\Omega_1 \times \cdots \times \Omega_\ell$. We observe that each $\Omega_i$ is a bounded, open subset of the hyperplane $\left\{z_i \in \mathbb{C}^{d_i} \colon \langle z_i,w_i\rangle=1\right\}$ which is symmetric with respect to complex conjugation and therefore $\Omega$ also has these properties as a subset of the affine space defined in \eqref{eq:bun}. For each $\iii \in \Sigma_N^*$ define a holomorphic function $\phi_\iii \colon \Omega \to \Omega$ by
\[\phi_\iii\left(z_1,\ldots,z_\ell\right):=\left(\phi_\iii^{(1)}(z_1),\ldots,\phi_\iii^{(\ell)}(z_\ell)\right)\]
in accordance with the statement of the theorem. As a consequence of Theorem \ref{th:multicones}(\ref{it:induce}) the set
\[  \overline{\bigcup_{\iii \in \Sigma_N^*}\phi_\iii(\Omega)} =  \overline{\bigcup_{\iii \in \Sigma_N^*} \phi_\iii^{(1)}(\Omega_1) \times \cdots \times \phi_\iii^{(\ell)}(\Omega_\ell) } \]
is a compact subset of $\Omega$. For each $\iii \in \Sigma_N^*$ define also $\psi_{\iii,t} \colon \Omega \to \mathbb{C}$ by
\[\psi_{\iii,t}(z_1,\ldots,z_\ell):=\prod_{i=1}^\ell \psi_{\iii,t}^{(i)}(z_i)\]
and observe that
\begin{equation}\label{eq:strongpsibound}\sup_{z \in \Omega} \left|\psi_{\iii,t}(z)\right| \leq C_2^{\|t\|}\prod_{i=1}^\ell \left\|A_\iii^{(i)}\right\|^{\Re(t_i)}\end{equation}
for all $\iii \in \Sigma_N^*$ and $t \in \mathbb{C}^\ell$, where $C_2:=(C_1e^{\pi/2})^{\sqrt{\ell}}$. In particular
\begin{equation}\label{eq:weakpsibound}
\sum_{j=1}^N \sup_{z \in \Omega} \left|\psi_{j,t}(z)\right| \leq C_2^{\|t\|}\sum_{j=1}^N \prod_{i=1}^\ell \left\|A_j^{(i)}\right\|^{\Re(t_i)} \leq NC_3^{\|t\|},
\end{equation}
say, for every $t \in \mathbb{C}^\ell$.

We may now define the operator $\mathscr{L}_t$ by
\[\left(\mathscr{L}_tf\right)(z):=\sum_{j=1}^N \psi_{j,t}(z)f(\phi_j(z))\]
for all $f \in \mathcal{A}^2(\Omega)$ and $z \in \Omega$ in accordance with the statement of the theorem. By Theorem \ref{th:banjen} it follows that each $\mathscr{L}_t$ is a well-defined bounded linear operator acting on $\mathcal{A}^2(\Omega)$ and that there exist $C_4,\gamma>0$  such that for all $t \in \mathbb{C}^\ell$ we have
\begin{align*}\mathfrak{s}_n\left(\mathscr{L}_t\right) &\leq C_4\left(\sum_{j=1}^N\sup_{z \in \Omega}\left|\psi_{j,t}(z)\right| \right)\exp\left(-\gamma n^{1/(\sum_{i=1}^\ell (d_i-1))}\right)\\
&  \leq C_4N\exp\left(\kappa\|t\|-\gamma n^{1/\sum_{i=1}^\ell (d_i-1))}\right)\end{align*}
as a consequence of \eqref{eq:weakpsibound}, where $\kappa:=\log C_3$. We have proved (i).

It follows from (i) that $\mathscr{L}_t$ is a trace-class operator. For each $\iii \in\Sigma_N^*$ and $t \in \mathbb{C}^\ell$ let us define an auxiliary operator $\mathscr{L}_{\iii,t}$ by
\[\left(\mathscr{L}_{\iii,t}f\right)(z):= \psi_{\iii,t}(z)f(\phi_\iii(z)).\]
Theorem \ref{th:banjen} shows in the same manner as before that each $\mathscr{L}_{\iii,t}$ is a well-defined trace-class operator on $\mathcal{A}^2(\Omega)$. The reader may easily verify the equations
\[\psi_{\jjj \iii,t}(z)=\psi_{\iii,t}(\phi_\jjj(z)) \psi_{\jjj,t}(z),\qquad \phi_{\jjj\iii}(z)=\phi_\iii(\phi_\jjj(z))\]
and therefore $\mathscr{L}_{\jjj\iii,t}=\mathscr{L}_{\iii,t}\mathscr{L}_{\jjj,t}$ for all $\iii,\jjj \in \Sigma_N^*$ and $t \in \mathbb{C}^\ell$. It follows by a simple inductive argument that $\mathscr{L}_t^n=\sum_{|\iii|=n}\mathscr{L}_{\iii,t}$ for every $n\geq 1$ and $t \in \mathbb{C}^\ell$, so in particular
\begin{equation}\label{eq:trace-is-linear}\tr \mathscr{L}_t = \tr \sum_{|\iii|=n}\mathscr{L}_{\iii,t}= \sum_{|\iii|=n}\tr \mathscr{L}_{\iii,t}\end{equation}
 for every $n\geq 1$ and $t \in \mathbb{C}^\ell$ by the linearity of the trace.
 
Let us now compute $\tr \mathscr{L}_{\iii,t}$ for fixed $\iii \in \Sigma_N^*$ and $t \in \mathbb{C}^\ell$. By Theorem \ref{th:multicones}(\ref{it:det}) each $\phi_\iii^{(i)}$ has a unique fixed point $z_\iii^{(i)} \in \Omega_i$ and it follows directly that $z_\iii:=(z_\iii^{(1)},\ldots,z_\iii^{(\ell)}) \in \Omega$ is the unique fixed point of $\phi_\iii$ in $\Omega$. Since $\langle A_\iii^{(i)}z_\iii^{(i)},w_i\rangle=\lambda_1(A_\iii^{(i)})$ for each $i=1,\ldots,\ell$ it follows easily that $\psi_{\iii,t}(z_\iii)=\prod_{i=1}^\ell \psi_{\iii,t}^{(i)}(z_\iii^{(i)})=\prod_{i=1}^\ell \rho(A_\iii^{(i)})^{t_i}$. By Theorem \ref{th:multicones}(\ref{it:det}) the derivative $D_{z_\iii^{(i)}}\phi_\iii^{(i)}$ of $\phi_\iii^{(i)}$ at $z_\iii^{(i)}$ satisfies
\[\det\left(I-D_{z_\iii^{(i)}}\phi_\iii^{(i)}\right) = \frac{p'_{A_\iii^{(i)}}\left(\lambda_1\left(A_\iii^{(i)}\right)\right)}{\lambda_1\left(A_\iii^{(i)}\right)^{d_i-1}}\]
where both sides of this expression are interpreted as $1$ if $d_i=1$, and since clearly $D_{z_\iii}\phi_\iii = D_{z_\iii^{(1)}}\phi_{\iii}^{(1)} \oplus \cdots \oplus  D_{z_\iii^{(\ell)}}\phi_{\iii}^{(\ell)}$ we easily obtain
\[\det \left(I-D_{z_\iii}\phi_\iii\right) = \prod_{i=1}^\ell \det\left(I-D_{z_\iii^{(i)}}\phi_\iii^{(i)}\right)
=\prod_{i=1}^\ell \frac{p'_{A_\iii^{(i)}}\left(\lambda_1\left(A_\iii^{(i)}\right)\right)}{\lambda_1\left(A_\iii^{(i)}\right)^{d_i-1}}.\]
It follows by Theorem \ref{t:thrace} that
\begin{equation}\label{eq:trace-individual}\tr \mathscr{L}_{\iii,t} = \prod_{i=1}^\ell \frac{\lambda_1\left(A_\iii^{(i)}\right)^{d_i-1} \rho\left(A_\iii^{(i)}\right)^{t_i}}{p_{A_\iii^{(i)}}'\left(\lambda_1\left(A_\iii^{(i)}\right)\right)}\end{equation}
 for every $t =(t_1,\ldots,t_\ell)\in \mathbb{C}^\ell$ and every $\iii \in \Sigma_N^*$, and combining \eqref{eq:trace-is-linear} with \eqref{eq:trace-individual} yields \eqref{eq:trarse} which completes the proof of (ii).
 
The proof of (iii) requires some preparatory steps. For the remainder of the proof we fix $t=(t_1,\ldots,t_\ell) \in \mathbb{R}^\ell$. We begin with the existence of the limit in \eqref{eq:limit}. By Theorem \ref{th:multicones}(\ref{it:aly}) there exists $\tau\in (0,1]$ such that $\tau  \|A_\iii ^{(i)}\|\cdot\|A_\jjj^{(i)}\| \leq \|A_\iii ^{(i)}A_\jjj^{(i)}\| \leq   \|A_\iii ^{(i)}\|\cdot\|A_\jjj^{(i)}\|$ for all $i=1,\ldots,\ell$ and $\iii,\jjj \in \Sigma_N^*$, which clearly implies
\[ \left\|A_{\jjj\iii}^{(i)} \right\|^{t_i}  \geq \tau^{|t_i|}\left\|A_\iii^{(i)} \right\|^{t_i}\left\|A_\jjj^{(i)} \right\|^{t_i}  \]
for all $i$, $\iii$ and $\jjj$. The inequality
\[\sum_{|\iii|=n+m} \prod_{i=1}^\ell \left\|A_\iii^{(i)}\right\|^{t_i}\geq\tau^{\sum_{i=1}^\ell |t_i|}\left(\sum_{|\iii|=n} \prod_{i=1}^\ell \left\|A_\iii^{(i)}\right\|^{t_i}\right)\left(\sum_{|\iii|=m} \prod_{i=1}^\ell \left\|A_\iii^{(i)}\right\|^{t_i}\right)  \]
for all $n,m \geq 1$ follows, so by superadditivity the limit
\[e^{\mathscr{P}(t)}:= \lim_{n \to \infty} \left(\sum_{|\iii|=n}  \prod_{i=1}^\ell \left\|A_\iii^{(i)}\right\|^{t_i}\right)^{\frac{1}{n}} = \sup_{n \geq 1}\left(\tau^{\sum_{i=1}^\ell |t_i|}\sum_{|\iii|=n}  \prod_{i=1}^\ell \left\|A_\iii^{(i)}\right\|^{t_i}\right)^{\frac{1}{n}} \]
is well-defined. We obtain in particular the inequality 
\begin{equation}\label{eq:pbup}
\sum_{|\iii|=n}  \prod_{i=1}^\ell \left\|A_\iii^{(i)}\right\|^{t_i}\leq \tau^{-\sum_{i=1}^\ell |t_i|} e^{n\mathscr{P}(t)} \leq C_5^{\|t\|}e^{n\mathscr{P}(t)},
\end{equation}
say, for all $n \geq 1$.

We next introduce a subset of $\Omega$ which will be useful in describing the behaviour of the eigenfunctions of $\mathscr{L}_t$. By Theorem \ref{th:multicones}(\ref{it:contract}) there exist for each $i=1,\ldots,\ell$ a metric $\mathsf{d}_i$ on $\Omega$ which is equivalent to the standard metric and a real number $\theta_i \in (0,1)$ such that every $\phi_\iii^{(i)}$ is a $\theta_i$-contraction with respect to $\mathsf{d}_i$. Clearly if $\mathsf{d}$ is the product metric derived from $\mathsf{d}_1,\ldots,\mathsf{d}_\ell$ and $\theta:=\max_i \theta_i$ then every $\phi_\iii$ is a $\theta$-contraction on $\Omega$ with respect to $\mathsf{d}$ and $\mathsf{d}$ is equivalent to the standard metric on $\Omega$. It follows that $\phi_1,\ldots,\phi_N$ defines an iterated function system on the compact set $\Omega':=\bigcup_{\iii \in \Sigma_N^*} \overline{\phi_\iii(\Omega)}$ in the sense of J.E. Hutchinson \cite{Hu81} and therefore there exists a unique nonempty compact set $\Lambda \subseteq \Omega'$ with the property $\Lambda=\bigcup_{j=1}^N \phi_j(\Lambda)$. Clearly $\Lambda=\bigcup_{|\iii|=n}\phi_\iii(\Lambda)$ for every $n \geq 1$ by a straightforward induction and it follows easily by contractivity that $\Lambda=\bigcap_{n=1} \bigcup_{|\iii|=n}\overline{\phi_\iii(\Omega)}$. On the other hand $\phi_1,\ldots,\phi_N$ clearly also defines an iterated function system on $\Omega' \cap \bigoplus_{i=1}^\ell \mathbb{R}^{d_i}$ since each map $z_i \mapsto \langle A_\iii^{(i)}z,w_i\rangle^{-1}A_\iii^{(i)}z_i$ obviously preserves $\Omega_i \cap \mathbb{R}^{d_i}$. There therefore exists a unique nonempty compact set $\Lambda' \subseteq \Omega'  \cap \bigoplus_{i=1}^\ell \mathbb{R}^{d_i}$ with the same property $\Lambda'=\bigcup_{j=1}^N \phi_j(\Lambda')$. By uniqueness we have $\Lambda=\Lambda'$ and we deduce that $\Lambda \subseteq \Omega \cap \bigoplus_{i=1}^\ell \mathbb{R}^{d_i}$.

We claim that $\Lambda$ has the following transitivity property: for every open set $U\subset \Omega$ having nonempty intersection with $\Lambda$ there exists $\jjj \in \Sigma_N^*$ such that $\phi_\jjj(\Omega) \subseteq U$. To demonstrate this choose $\omega \in U \cap \Lambda$ arbitrarily, let $\varepsilon>0$ be small enough that the ball of centre $\omega$ and radius $\varepsilon$ in the metric $\mathsf{d}$ is a subset of $U$, and let $n$ be large enough that $\theta^n\diam \Omega<\varepsilon$ in the sense of the metric $\mathsf{d}$. Since $\omega \in \Lambda=\bigcup_{|\iii|=n}\phi_\iii(\Lambda)$ there exists $\jjj \in \Sigma_N^*$ with length $n$ such that $\omega \in \phi_\jjj(\Lambda)$. Clearly $\omega \in \phi_\jjj(\Omega)$ and every other point of $\phi_\jjj(\Omega)$ is within distance $\theta^n\diam \Omega <\varepsilon$ of $\omega$, so $\phi_\jjj(\Omega)$ is contained in the $\varepsilon$-ball around $\omega$ and is therefore a subset of $U$ as required. The claim is proved.
 
We make one final preliminary claim: there exists $C_6>1$ such that for every $f \in \mathcal{A}^2(\Omega)$,
\begin{equation}\label{eq:soup}\limsup_{n\to \infty} \sup_{z \in \Omega} e^{-n\mathscr{P}(t)}\left|(\mathscr{L}_t^nf)(z)\right| \leq C_6^{\|t\|} \sup_{z \in \Lambda}|f(z)|.\end{equation}
To prove the claim let $z_0 \in \Omega$ and $n \geq 1$ be arbitrary: we have
\begin{align*}e^{-n\mathscr{P}(t)}\left|\left(\mathscr{L}_t^nf\right)(z_0)\right|&=e^{-n\mathscr{P}(t)}\left|\sum_{|\iii|=n}\psi_{\iii,t}(z_0)f\left(\phi_\iii(z_0)\right)\right|\\
&\leq C_2^{\|t\|}e^{-n\mathscr{P}(t)}\sum_{|\iii|=n} \left(\prod_{i=1}^\ell \left\|A_\iii^{(i)}\right\|^{t_i}\right) \left|f\left(\phi_\iii(z_0)\right)\right|\\
&\leq C_2^{\|t\|}C_5^{\|t\|}\sup_{z \in \overline{\bigcup_{|\iii|=n} \phi_\iii(\Omega)}} |f(z)|\end{align*}
using \eqref{eq:strongpsibound} and \eqref{eq:pbup} and the result follows easily since $\Lambda=\bigcap_{n=1} \bigcup_{|\iii|=n}\overline{\phi_\iii(\Omega)}$. We observe immediately that $\rho(\mathscr{L}_t)\leq e^{\mathscr{P}(t)}$ since obviously \eqref{eq:soup} prevents $\mathscr{L}_t$ from having an eigenfunction which corresponds to an eigenvalue of modulus strictly greater than $e^{\mathscr{P}(t)}$.
We also observe that as a consequence of \eqref{eq:soup} an eigenfunction of $\mathscr{L}_t$ with eigenvalue of modulus $e^{\mathscr{P}(t)}$ cannot vanish identically on $\Lambda$.

In order to apply Theorem \ref{th:kr} we wish to study the action of $\mathscr{L}_t$ on a real Hilbert space. Let us define
\[\mathcal{H}:=\left\{f \in \mathcal{A}^2(\Omega) \colon f(z)\in \mathbb{R} \text{ for all }z \in \Omega \cap \bigoplus_{i=1}^\ell \mathbb{R}^{d_i}\right\}\]
 and note that $\mathcal{H}$ is a closed subset of $\mathcal{A}^2(\Omega)$ as a consequence of Lemma \ref{le:supbound}. It is clear that $\mathcal{H}$ is also a real Hilbert space when equipped with the norm $\|\cdot\|_{\mathcal{A}^2(\Omega)}$. We observe that the complexification $\mathcal{H}^{\mathbb{C}}$ is precisely $\mathcal{A}^2(\Omega)$. Indeed, since $z \in  \Omega$ if and only if $z^* \in \Omega$ by Theorem \ref{th:multicones}, for every $f \in \mathcal{A}^2(\Omega)$ the holomorphic function $f^*$ defined by $f^*(z):=f(z^*)^*$ is also an element of $\mathcal{A}^2(\Omega)$; thus every $f \in \mathcal{A}^2(\Omega)$ can be written as $f = \frac{1}{2}(f + f^*) + \frac{1}{2}(f-f^*) =g+ih$, say, where $f,g \in \mathcal{H}$. This decomposition is moreover unique since if $g+ih$ is the zero function with $g,h \in \mathcal{H}$ then $g$ and $h$ are identically zero on $\Omega \cap \bigoplus_{i=1}^\ell \mathbb{R}^{d_i}$, hence all of their derivatives vanish there, hence they are zero on every connected component of $\Omega$ which intersects $ \bigoplus_{i=1}^\ell \mathbb{R}^{d_i}$, hence they are zero throughout $\Omega$ by Theorem \ref{th:multicones}(\ref{it:omega}). 
  
We wish to apply Theorem \ref{th:kr} in order to study the spectrum of $\mathscr{L}_t$ on $\mathcal{H}$ and hence on its complexification $\mathcal{A}^2(\Omega)$. The natural mechanism for doing this is to consider the cone of elements of $\mathcal{H}$ which are non-negative on a convenient compact subset such as $\bigcup_{j=1}^N \overline{\phi_j(\Omega)} \cap \bigoplus_{i=1}^\ell \mathbb{R}^{d_i}$ and show that every nonzero element of the cone is eventually mapped to an interior point (which is precisely a function which is positive throughout the compact subset) by some power of $\mathscr{L}_t$. However, in the full generality of Theorem \ref{th:main} it is possible that $\phi_\iii(\Omega)$ may be extremely small, indeed even a singleton set. In such cases it is not necessarily the case that every nonzero holomorphic function on $\Omega$ is eventually mapped to a function which is positive on a prescribed set and Theorem \ref{th:kr} may not be directly applicable. To resolve this issue we will pass to a suitable quotient Hilbert space.

It is clear from the definition of $\psi_{j,t}$ and $\phi_j$ that $(\mathscr{L}_tf)(z)$ is real  when $f \in \mathcal{H}$ and $z \in \Omega \cap \bigoplus_{i=1}^\ell \mathbb{R}^{d_i}$, so $\mathscr{L}_t$ acts on $\mathcal{H}$. Define $\mathcal{Z}:=\{f \in \mathcal{H} \colon f(z)=0\text{ for all }z \in \Lambda\}$ and note that $\mathcal{Z}$ is a vector subspace of $\mathcal{H}$ and is closed as a consequence of Lemma \ref{le:supbound}. We observe that by similar reasoning $\mathscr{L}_t$ preserves the subspace $\mathcal{Z}$. The quotient space $\mathcal{H}/\mathcal{Z}$ is a Hilbert space when equipped with norm $\|[f]\|_{\mathcal{H}/\mathcal{Z}}:=\inf\{\|f-g\|_{\mathcal{A}^2(\Omega)}\colon g \in \mathcal{Z}\}$, being isometrically isomorphic to the orthogonal complement of $\mathcal{Z}$ in $\mathcal{H}$. It is not difficult to see that the operator $\mathscr{L}_t$ induces a compact operator on the real Hilbert space $\mathcal{H}/\mathcal{Z}$ which we also denote by $\mathscr{L}_t$.

We observe that for each $z \in \Lambda$ the functional $[f]\mapsto f(z)$ is a well-defined continuous linear functional $\mathcal{H}/\mathcal{Z} \to \mathbb{R}$. Indeed, if $[f] \in \mathcal{H}/\mathcal{Z}$ and $g \in \mathcal{Z}$ then we have $f(z)=(f+g)(z)$ and
\[|f(z)|=|(f+g)(z)| \leq C_\Lambda\|f+g\|_{\mathcal{A}^2(\Omega)}\]
where $C_\Lambda>0$ is the constant given by Lemma \ref{le:supbound} in respect of the nonempty compact set $\Lambda$. In particular $f(z)$ is independent of the choice of representative $f \in [f]$ and
\begin{equation}\label{eq:mog-can-do-one}|f(z)|\leq C_\Lambda\inf\{\|f+g\|_{\mathcal{A}^2(\Omega)} \colon g \in \mathcal{Z}\}=C_\Lambda\|[f]\|_{\mathcal{H}/\mathcal{Z}}\end{equation}
so that the functional $[f] \mapsto f(z)$ is continuous as claimed. Now define
\[\mathcal{C}:=\{[f] \in \mathcal{H}/\mathcal{Z} \colon f(z)\geq 0\text{ for all }z \in \Lambda\} = \bigcap_{z \in \Lambda}\left\{[f]\in \mathcal{H}/\mathcal{Z} \colon f(z) \geq 0\right\}.\]
This set is clearly well-defined, positively homogenous, convex, and closed. If $[f] \in \mathcal{C} \cap -\mathcal{C}$ then $f(z)=0$ for all $z \in \Lambda$ so that $f \in \mathcal{Z}$ and therefore the only element of $\mathcal{C}\cap-\mathcal{C}$ is $[0]$. Since the function $[f]\mapsto \inf_{z \in \Lambda} f(z)$ is continuous as a consequence of \eqref{eq:mog-can-do-one} it is not difficult to see that $[f] \in \mathcal{H}/\mathcal{Z}$ is an interior point of $\mathcal{C}$ if and only if $\inf_{z \in \Lambda} f(z)>0$. In particular the set $\mathcal{C}$ satisfies conditions (i)--(iii) of Theorem \ref{th:kr}. We observe also that $\mathscr{L}_t\mathcal{C}\subseteq \mathcal{C}$ since by construction each $\psi_{j,t}$ is positive on $\Omega \cap \bigoplus_{i=1}^\ell \mathbb{R}^{d_i}$ and in particular on $\Lambda$.

In order to be able to apply Theorem \ref{th:kr} we must show that for every $[f] \in \mathcal{C}$ with $[f]\neq [0]$ there exists $n\geq 1$ such that $\inf_{z \in \Lambda} (\mathcal{L}_t^nf)(z)>0$. Given $[f] \in \mathcal{C}$ with $[f]\neq [0]$ there necessarily exists $z_0 \in \Lambda$ such that $f(z_0)>0$ and hence there exists an open set $U \subset \Omega$ intersecting $\Lambda$ such that $f(z)>0$ for all $z \in U\cap \Lambda$. By the transitivity property of $\Lambda$ remarked earlier there exists a word $\jjj\in \Sigma_N^*$ with some length $n$ such that $\overline{\phi_{\jjj}(\Omega)} \subset U$, so in particular $\overline{\phi_{\jjj}(\Lambda)} \subset U \cap \Lambda$ and therefore $f(\phi_\jjj(z))>0$ for all $z \in \Lambda$. Hence
\[(\mathscr{L}_t^nf)(z) =\sum_{|\iii|=n}\psi_{\iii,t}(z)f(\phi_\iii(z)) \geq \psi_{\jjj,t}(z)f(\phi_\jjj(z))>0\]
for all $z \in \Lambda$ since each $\psi_{\iii,t}$ is real and positive throughout $\Lambda$, each $f\circ\phi_\iii$ is real and non-negative throughout $\Lambda$, and $f\circ\phi_\jjj$ is real and positive throughout $\Lambda$. We have obtained $\inf_{z \in \Lambda} (\mathcal{L}_t^nf)(z)>0$ and therefore $[\mathcal{L}_t^nf] \in \mathcal{C}$ as required. 

We may now apply Theorem \ref{th:kr} to the action of $\mathscr{L}_t$ on $\mathcal{H}/\mathcal{Z}$. By that theorem the spectral radius $R$ of $\mathscr{L}_t$ on $\mathcal{H}/\mathcal{Z}$ is positive and there exists $[\xi_t] \in \Int \mathcal{C}$ such that $[\mathscr{L}_t\xi_t]=R[\xi_t]$. Now let $z_0 \in \Lambda$ be arbitrary. It follows from \eqref{eq:herpderp} that
\[C_1^{-\|t\|}\prod_{i=1}^\ell \left\|A_\iii^{(i)}\right\|^{t_i} \leq \psi_{\iii,t}(z_0) \leq C_1^{\|t\|}\prod_{i=1}^\ell \left\|A_\iii^{(i)}\right\|^{t_i}\]
for every $\iii \in \Sigma_N^*$, so
\[\left(\mathscr{L}_t^n\xi_t\right)(z_0) = \sum_{|\iii|=n}\psi_{\iii,t}(z_0) \xi_t(\phi_\iii(z_0)) \geq C_1^{-\|t\|}\sum_{|\iii|=n}\prod_{i=1}^\ell \left\|A_\iii^{(i)}\right\|^{t_i} \left(\inf_{z \in \Lambda}\xi_t(z)\right)\]
and in a similar fashion
\[\left(\mathscr{L}_t^n\xi_t\right)(z_0) \leq C_1^{\|t\|}\leq \sum_{|\iii|=n}\prod_{i=1}^\ell \left\|A_\iii^{(i)}\right\|^{t_i} \left(\sup_{z \in \Lambda}\xi_t(z)\right).\]
Since $[\mathscr{L}_t^n\xi]=R^n[\xi_t]$ and the function $[f]\mapsto f(z_0)$ is continuous, the left-hand side of each of the last two displayed equations is simply $R^n\xi_t(z_0)>0$. Taking the power $1/n$ and letting $n \to \infty$ it follows that $R = e^{\mathscr{P}(t)}$, and we previously observed that $e^{\mathscr{P}(t)}\geq \rho(\mathscr{L}_t)$. On the other hand it is clear that necessarily $\rho(\mathscr{L}_t) \geq  \lim_{n \to \infty} |(\mathscr{L}_t^n\xi_t)(z_0)|^{1/n}=R$ and we conclude that  $R = e^{\mathscr{P}(t)}= \rho(\mathscr{L}_t)$.

If an eigenvalue of $\mathscr{L}_t$ acting on $\mathcal{A}^2(\Omega)$ has absolute value $\rho(\mathscr{L}_t)$ then by \eqref{eq:soup} its corresponding eigenfunction $\eta_t$ cannot be identically zero on $\Lambda$. Consequently $[\eta_t] \neq [0]$ and therefore $[\eta_t]$ is an eigenfunction of $\mathscr{L}_t$ on $\mathcal{H}/\mathcal{Z}$ (or its complexification) with the same eigenvalue. But by Theorem \ref{th:kr} this is only possible if the eigenvalue is $\rho(\mathscr{L}_t)$ itself, and we conclude that $\rho(\mathscr{L}_t)$ is the only eigenvalue of $\mathscr{L}_t$ on $\mathcal{A}^2(\Omega)$ which has maximum modulus. Moreover this eigenvalue is simple: if two linearly independent eigenfunctions $\xi^1_t, \xi^2_t \in \mathcal{A}^2(\Omega)$ exist then by \eqref{eq:soup} neither function can be identically zero on $\Lambda$; by multiplying each by a complex unit if necessary, we may assume that each takes a nonzero real value somewhere on $\Lambda$; and replacing $\xi_t^1$ and $\xi_t^2$ with the functions $\xi_t^1+(\xi_t^1)^*$ and $\xi_t+(\xi_t^2)^*$ if necessary we may assume that $\xi_t^1,\xi_t^2 \in \mathcal{H}$ and $[\xi_t^1],[\xi_t^2]\neq 0$. Since $\mathscr{L}_t$ acting on $\mathcal{H}/\mathcal{Z}$ has a simple eigenvalue at $e^{\mathscr{P}(t)}$ by Theorem \ref{th:kr}, the equivalence classes $[\xi^1_t]$ and $[\xi^2_t]$ must be exact, nonzero scalar multiples of one another. This is precisely to say that some linear combination of $\xi_t^1$ and $\xi^2_t$ vanishes identically on $\Lambda$ but is not the zero element of $\mathcal{H}$; but since that linear combination is an eigenfunction with eigenvalue $e^{\mathscr{P}(t)}$ this contradicts \eqref{eq:soup}.

To complete the proof it remains only to show that $\rho(\mathscr{L}_t)$ is an algebraically simple eigenvalue. Let $\xi_t\in\mathcal{A}^2(\Omega)$ be an eigenfunction corresponding to this eigenvalue and observe that by \eqref{eq:soup} $\xi_t$ is not identically zero on $\Lambda$. If $e^{\mathscr{P}(t)}$ is not algebraically simple, there exists nonzero $\eta_t \in \mathcal{A}^2(\Omega)$ such that $\mathscr{L}_t\eta_t=e^{\mathscr{P}(t)}(\eta_t+\xi_t)$ and therefore $\mathscr{L}_t^n\eta_t=e^{n\mathscr{P}(t)}(\eta_t+n\xi_t)$ for every $n \geq 1$, but this is only compatible with \eqref{eq:soup} if $\xi_t$ is identically zero on $\Lambda$, a contradiction. The proof is complete.
\end{proof}


\section{Proof of Theorem \ref{th:main} }\label{se:proofs}

Before starting the proof of Theorem \ref{th:main} we require two preliminary lemmas, one concerning the behaviour of the leading eigenvalue of the operator $\mathscr{L}_s$ of Theorems \ref{th:opter} and \ref{th:topaff} and one an abstract result concerning sequences of implicit functions in two complex variables. 
\begin{lemma}\label{le:do-not-press}
Let $(A_1,\ldots,A_N)\in M_d(\mathbb{R})^N$ be $k$- and $(k+1)$-multipositive with $N,d \geq 2$ and $0 \leq k<d$, and for each $s \in \mathbb{C}$ let $\mathscr{L}_s \colon \mathscr{H} \to \mathscr{H}$ be as given by Theorem \ref{th:opter}. 
Define
\[p(s):=\log\rho(\mathscr{L}_s)=\lim_{n\to \infty}  \frac{1}{n}\log \left(\sum_{|\iii|=n}\left\|A_\iii^{\wedge k}\right\|^{k+1-s} \left\|A_\iii^{\wedge (k+1)}\right\|^{s-k}\right)\]
for all $s \in \mathbb{R}$. Then $p \colon \mathbb{R} \to \mathbb{R}$ is convex. If additionally there exists a norm $\vertle{\cdot}$ on $\mathbb{R}^d$ with respect to which $\max_{1 \leq i \leq N}\vertle{A_i}<1$, then there exists $c>0$ such that
\[\frac{p(s_2)-p(s_1)}{s_2-s_1}\leq -c<0\]
for all pairs of distinct points $s_1,s_2 \in \mathbb{R}$.
\end{lemma}
\begin{proof}
If $s_1,s_2 \in \mathbb{R}$, $\lambda \in (0,1)$ and $n \geq 1$ then 
\begin{eqnarray*}\lefteqn{\sum_{|\iii|=n}\left\|A_\iii^{\wedge k}\right\|^{k+1-\lambda s_1 -(1-\lambda)s_2} \left\|A_\iii^{\wedge (k+1)}\right\|^{\lambda s_1 + (1-\lambda )s_2-k}}& &\\
& &=\sum_{|\iii|=n}\left(\left\|A_\iii^{\wedge k}\right\|^{k+1-s_1}\left\|A_\iii^{\wedge (k+1)}\right\|^{s_1-k }\right)^\lambda  \left(\left\|A_\iii^{\wedge k}\right\|^{k+1-s_2}\left\|A_\iii^{\wedge (k+1)}\right\|^{s_2-k }\right)^{1-\lambda} \\
& & \leq \left(\sum_{|\iii|=n}\left\|A_\iii^{\wedge k}\right\|^{k+1-s_1}\left\|A_\iii^{\wedge (k+1)}\right\|^{s_1-k }\right)^\lambda  \left(\sum_{|\iii|=n}\left\|A_\iii^{\wedge k}\right\|^{k+1-s_2}\left\|A_\iii^{\wedge (k+1)}\right\|^{s_2-k }\right)^{1-\lambda} \end{eqnarray*}
using H\"older's inequality with $p=1/\lambda$ and $q=1/(1-\lambda)$. Taking $n^{\mathrm{th}}$ roots and letting $n \to \infty$ it follows directly that $\rho(\mathscr{L}_{\lambda s_1+(1-\lambda)s_2}) \leq \rho(\mathscr{L}_{s_1})^\lambda \rho(\mathscr{L}_{s_2})^{1-\lambda}$ and the convexity of $p$ follows by taking logarithms. 

To complete the proof suppose that there exists a norm $\vertle{\cdot}$ on $\mathbb{R}^d$ with respect to which $\max_{1 \leq i \leq N}\vertle{A_i}<1$, and choose $C>0$ such that $\|B\| \leq C\vertle{B}$ for all $B \in M_d(\mathbb{R})$. Observe that in particular $\sigma_{k+1}(A_\iii)\leq \sigma_1(A_\iii)=\|A_\iii\|\leq C\vertle{A_\iii}$ for all $\iii \in \Sigma_N^*$. If $s_1<s_2 \in \mathbb{R}$ and $n \geq 1$ then
\begin{eqnarray*}\lefteqn{\sum_{|\iii|=n}\left\|A_\iii^{\wedge k}\right\|^{k+1-s_2} \left\|A_\iii^{\wedge (k+1)}\right\|^{s_2-k}}& &\\
& =&\sum_{|\iii|=n} \sigma_1(A_\iii)\cdots \sigma_k(A_\iii) \sigma_{k+1}(A_\iii)^{s_2-k}\\
& =&\sum_{|\iii|=n} \sigma_1(A_\iii)\cdots \sigma_k(A_\iii) \sigma_{k+1}(A_\iii)^{s_1-k}\sigma_{k+1}(A_\iii)^{s_2-s_1}\\
& \leq& \left(\max_{|\iii|=n}\sigma_{k+1}(A_\iii)\right)^{s_2-s_1} \sum_{|\iii|=n} \sigma_1(A_\iii)\cdots \sigma_k(A_\iii) \sigma_{k+1}(A_\iii)^{s_1-k}\\
& \leq& \left(\max_{|\iii|=n}C\vertle{A_\iii}\right)^{s_2-s_1} \sum_{|\iii|=n} \sigma_1(A_\iii)\cdots \sigma_k(A_\iii) \sigma_{k+1}(A_\iii)^{s_1-k}\\
& \leq& C^{s_2-s_1}\left(\max_{1 \leq i \leq N}\vertle{A_i}\right)^{n(s_2-s_1)} \sum_{|\iii|=n}\left\|A_\iii^{\wedge k}\right\|^{k+1-s_1} \left\|A_\iii^{\wedge (k+1)}\right\|^{s_1-k}\end{eqnarray*}
so that by taking the $n^{\mathrm{th}}$ root and letting $n \to \infty$ we obtain
\[\rho(\mathscr{L}_{s_2})\leq \left(\max_{1 \leq i \leq N}\vertle{A_i}\right)^{s_2-s_1}\rho(\mathscr{L}_{s_1})  \]
for all such $s_1$ and $s_2$. Taking logarithms and rearranging yields the claim with $c:=-\log \max_{1\leq i \leq N} \vertle{A_i}>0$. 
\end{proof}
Similarly to \S\ref{se:prelims} we shall say that $X_1$ is compactly contained in $X_2$ if the closure of $X_1$ is a compact subset of the interior of $X_2$, and express this relation with the notation $X_1 \Subset X_2$.
\begin{lemma}\label{le:disco}
Let $D_1,D_2 \subset \mathbb{C}$ be open discs, let  $f_n  \colon D_1 \times D_2 \to \mathbb{C}$ be a bounded holomorphic function for each $n\geq 1$, and let $f \colon D_1 \times D_2 \to \mathbb{C}$ be bounded and holomorphic. Suppose that there exists a holomorphic function $g \colon D_1 \to D_2$ such that  for all $s \in D_1$, $g(s)$ is a simple zero of the function $z \mapsto f(s,z)$ and is the unique zero of that function in $D_2$. Suppose also that
\[\lim_{n \to \infty} \sup_{s \in D_1} \sup_{z \in D_2} \left|f_n(s,z)-f(s,z)\right| =0.\]
Let $D_1'$ be any open disc which is compactly contained in $D_1$. Then there exist a disc $D_2' \subseteq D_2$, which may be chosen concentric with $D_2$ and with radius arbitrarily close to that of $D_2$, an integer $n_0\geq 1$ and holomorphic functions $g_n \colon D_1' \to D_2'$ defined for all $n\geq n_0$ such that:
\begin{enumerate}[(i)]
\item
For all $n \geq n_0$ and $s \in D_1'$, $g_n(s)$ is a simple zero of $z \mapsto f_n(s,z)$ and is the unique zero of that function in $D_2'$.
\item
For every integer $\ell \geq 0$ there exists $C_\ell>0$ such that
\[\sup_{s \in D_1'} \left|g_n^{(\ell)}(s)-g^{(\ell)}(s)\right| \leq C_\ell\sup_{s \in D_1}\sup_{z \in D_2}|f_n(s,z)-f(s,z)|\]
for all $n \geq n_0$, where $h^{(\ell)}$ denotes the $\ell^{\mathrm{th}}$ derivative of the function $h$.
\end{enumerate}
\end{lemma}
\begin{proof}
Throughout the proof let $D_3$ be an open disc such that $D_1' \Subset D_3 \Subset D_1$. By compactness and continuity we have $g(\overline{D_3})\Subset D_2$. Let $D_2' \Subset D_2$ be any disc which is concentric with $D_2$ and has radius large enough that $g(\overline{D_3})\Subset D_2'$. By compactness and continuity we obtain
\[\inf_{s \in \overline{D_3}} \inf_{z \in \partial D_2'} |f(s,z)|>0\]
and hence by uniform convergence there exists $n_1 \geq 1$ such that for all $n\geq n_1$
\[\sup_{s \in \overline{D_3}} \sup_{z \in \partial D_2'} |f(s,z)-f_n(s,z)|<\inf_{s \in \overline{D_3}} \inf_{z \in \partial D_2'} |f(s,z)|.\]
It follows by Rouch\'e's theorem that for every $s \in \overline{D_3}$ and $n\geq n_1$ there exists a unique zero $g_n(s)$ of the function $z \mapsto f_n(s,z)$ in $D_2'$ and this zero is simple. Since each $f_n$ is holomorphic it follows by the holomorphic implicit function theorem (see e.g. \cite[p.34]{FrGr02}) that each $g_n \colon \overline{D_3} \to D_2'$ is holomorphic on $D_3$. 

We claim now that 
\[\lim_{n \to \infty}\sup_{s \in \overline{D_3}} |g_n(s)-g(s)|=0.\]
Indeed, let $\varepsilon>0$ be any number which is small enough that for every $s \in \overline{D_3}$ the closed $\varepsilon$-ball centred at $g(s)$ is a subset of $D_2'$. By compactness and the absence of zeros of $z \mapsto f(s,z)$ in $D_2 \setminus \{g(s)\}$ we have
\[\inf_{s \in \overline{D_3}} \inf_{|z-g(s)|=\varepsilon} |f(s,z)|>0\]
so that in the same manner if $n$ is large enough 
\[\sup_{s \in \overline{D_3}} \sup_{|z-g(z)|=\varepsilon} |f(s,z)-f_n(s,z)|<\inf_{s \in\overline{D_3}} \inf_{|z-g(s)|=\varepsilon} |f(s,z)|.\]
Applying Rouch\'e's theorem again it follows that if $n$ is sufficiently large then for all $s \in \overline{D_3}$ there is a unique zero of the function $z \mapsto f_n(s,z)$ in the region $0 \leq |z-g(s)|<\varepsilon$. This zero belongs to $D_2'$ and hence is necessarily equal to $g_n(s)$, and we therefore have $\sup_{s \in \overline{D_3}}|g_n(s)-g(s)|\leq \varepsilon$. Since $\varepsilon$ was arbitrary we conclude that 
\begin{equation}\label{eq:clams}\lim_{n \to \infty}\sup_{s \in \overline{D_3}} |g_n(s)-g(s)|=0\end{equation}
as claimed.

For each $s \in D_1$ the value $g(s)$ is a simple zero of the function $z \mapsto f(s,z)$, so we have $\frac{\partial f}{\partial z}(s,g(s))\neq0$ for all $s \in D_1$. Define
\[c:=\inf_{s \in \overline{D_3}} \left|\frac{\partial f}{\partial z}(s,g(s))\right|>0.\]
Since $g(\overline{D_3}) \Subset D_2' \Subset D_2$ we may choose $\tau>0$ small enough that for every $z \in \partial D_2'$ the closed ball of radius $2\tau$ centred at $z$ is a subset of $D_2$ which does not intersect $g(\overline{D_3})$. Using \eqref{eq:clams} take $n_2 \geq n_1$ large enough that
\[\sup_{s \in \overline{D_3}} |g_n(s)-g(s)|<\tau\]
for all $n \geq n_2$. Observe that if $s \in \overline{D_3}$ and $n \geq n_2$ then $|g_n(s)-g(s)|<\tau$ and $|g(s)-\omega|>2\tau$ and therefore $|g_n(s)-\omega|>\tau$ for all $\omega \in \partial D_2'$. Using Cauchy's integral formula, for any two distinct points $z_1, z_2 \in D_2'$ we have
\begin{eqnarray*}\lefteqn{\frac{f(s,z_1)-f(s,z_2)}{z_1-z_2}-\frac{\partial f}{\partial z}(s,z_2)}& &\\
&  =& \frac{1}{2\pi i}\int_{\partial D_2'} \frac{f(s,\omega)}{(z_1-z_2)(\omega-z_1)}- \frac{f(s,\omega)}{(z_1-z_2)(\omega-z_2)} - \frac{f(s,\omega)}{(\omega-z_2)^2}d\omega\\
&  =&\frac{1}{2\pi i}\int_{\partial D_2'} \frac{f(s,\omega)((\omega-z_2)^2-(\omega-z_1)(\omega-z_2)-(z_1-z_2)(\omega-z_1))}{(z_1-z_2)(\omega-z_1)(\omega-z_2)^2}d\omega\\
& =&\frac{1}{2\pi i}\int_{\partial D_2'} \frac{f(s,\omega)(z_1^2-2z_1z_2+z_2^2)}{(z_1-z_2)(\omega-z_1)(\omega-z_2)^2}d\omega\\
&=&\frac{1}{2\pi i}\int_{\partial D_2'} \frac{f(s,\omega)(z_1-z_2)}{(\omega-z_1)(\omega-z_2)^2}d\omega.\end{eqnarray*}
Hence if $s \in \overline{D_3}$, $n \geq n_2$ and $g_n(s) \neq g(s)$ then since $g(s),g_n(s)\in D_2'$ 
\[\left|\frac{f(s,g_n(s))-f(s,g(s))}{g_n(s)-g(s)}-\frac{\partial f}{\partial z}(s,g(s))\right|\leq \frac{R|g_n(s)-g(s)|}{\tau^3} \cdot \sup_{t \in D_1} \sup_{z \in D_2} |f(t,z)|\]
where $R$ denotes the radius of $D_2'$. Now take $n_3 \geq n_2$ large enough that additionally
\[\left(\sup_{s \in \overline{D_3}} |g_n(s)-g(s)|  \right)\left(\frac{R}{\tau^3} \sup_{s \in D_1} \sup_{z \in D_2} |f(s,z)|\right)<\frac{c}{2}.\]
If $n \geq n_3$, $s \in \overline{D_3}$ and $g_n(s) \neq g(s)$ then since $f_n(s,g_n(s))=0=f(s,g(s))$ we have
\begin{eqnarray*}
\lefteqn{\left|\frac{f(s,g_n(s))-f_n(s,g_n(s))}{g_n(s)-g(s)}\right|}& &\\
& =&\left|\frac{f(s,g_n(s))-f(s,g(s))}{g_n(s)-g(s)}\right|\\
&\geq& \left|\frac{\partial f}{\partial z}(s,g(s))\right|-\left|\frac{f(s,g_n(s))-f(s,g(s))}{g_n(s)-g(s)} -\frac{\partial f}{\partial z}(s,g(s))\right|>\frac{c}{2}.
\end{eqnarray*}
It follows that when $n \geq n_3$ 
\[\sup_{s \in \overline{D_3}} |g_n(s)-g(s)| \leq \frac{2}{c} \sup_{s \in \overline{D_3}}\sup_{z \in D_2}|f_n(s,z)-f(s,z)|.\]
To complete the proof of the lemma let $\delta>0$ be small enough that for every $s \in D_1'$ the closed $\delta$-ball centred at $s$ is a subset of $D_3$. By the Cauchy integral formula we have for each integer $\ell\geq 0$ and every $n \geq n_3$
\begin{align*}\sup_{s \in D_1'} \left|g^{(\ell)}_n(s)-g^{(\ell)}(s)\right| &\leq \sup_{s \in D_1'} \left|\frac{\ell!}{2\pi i} \int_{|s-t|=\delta} \frac{g_n(t)-g(t)}{(t-s)^{\ell+1}} dt \right|\\
&\leq \delta^{-\ell}\ell! \sup_{s \in \overline{D_3}} |g_n(s)-g(s)|\\
&\leq \frac{2\ell!}{c\delta^\ell} \sup_{s \in D_1}\sup_{z \in D_2} |f_n(s,z)-f(s,z)|\end{align*}
as required. The proof is complete.
\end{proof}

\begin{proof}[Proof of Theorem \ref{th:main}]
Let $(A_1,\ldots,A_N)\in M_d(\mathbb{R})^N$ be $k$- and $(k+1)$-multipositive where $N,d \geq 2$ and $0\leq k<d$. For all $s \in \mathbb{C}$ let $\mathscr{L}_s \colon\mathscr{H}\to \mathscr{H}$ be as given by Theorem \ref{th:opter}.
Let $t_n(s)$ and $a_n(s)$ be as defined in the statement of Theorem \ref{th:main}. We claim that there exist $\tilde K, \tilde \gamma,\kappa>0$ such that 
\begin{equation}\label{eq:astrosheep-6000}|a_n(s)|\leq \tilde{K}^n e^{n \kappa|s|} \exp\left(-\tilde\gamma n^\alpha\right)\end{equation}
for all $n \geq 1$ and $s \in \mathbb{C}$, where $\tilde{K},\tilde{\gamma}$ and $\kappa$ do not depend on $s$ or $n$ and where
\[\alpha:= \frac{{d+1 \choose k+1}-1}{{d+1\choose k+1}-2} = 1+\frac{1}{{d\choose k}+{d\choose k+1}-2}.\]
By Theorem \ref{th:opter} there exist constants $C,\gamma,\kappa>0$ such that
\begin{equation}\label{eq:todos-os-garotos-e-todos-as-garotas-querem}\mathfrak{s}_n(\mathscr{L}_s)\leq C\exp\left(\kappa|s| - \gamma n^\beta\right) \end{equation}
for all $n \geq 1$ and $s \in \mathbb{C}$ where $\beta:=({d +1\choose k+1}-2)^{-1}=\alpha-1$, and $\mathscr{L}_s$ is trace-class with $\tr \mathscr{L}_s^n=t_n(s)$ for all $s \in \mathbb{C}$. By Theorem \ref{th:ofundity}  we have
\begin{equation}\label{eq:se-voce-procura-a-amy}|a_n(s)| \leq \sum_{i_1<\cdots<i_n}\mathfrak{s}_{i_1}(\mathscr{L}_s) \cdots \mathfrak{s}_{i_n}(\mathscr{L}_s)\end{equation}
for all $n \geq 1$. In order to proceed further we require two elementary inequalities. We first note that for every integer $m\geq 2$
\begin{align}\label{eq:bohnenstadt}\sum_{\ell=m}^\infty e^{-\gamma \ell^\beta} \leq \int_{m-1}^\infty e^{-\gamma t^\beta}dt&=\frac{1}{\beta}\int_{\left(m-1\right)^\beta}^\infty u^{\frac{1}{\beta}-1}e^{-\gamma u}du\\\nonumber
 &\leq \frac{K}{\beta}\int_{\left(m-1\right)^\beta}^\infty e^{-\frac{\gamma}{2}u}du\\\nonumber
&\leq \frac{2K}{\beta\gamma} e^{-\frac{\gamma}{2} \left(m-1\right)^\beta}\leq \frac{2K}{\beta\gamma}e^{-\frac{\gamma}{2^{1+\beta}}m^\beta} \end{align}
where $K:=\sup\{ x^{\frac{1}{\beta}-1} e^{-\gamma x/2}\colon x\geq \frac{1}{2}\}>0$ depends only on $\beta$ and $\gamma$, and by increasing $K$ if necessary we have $\sum_{\ell=m}^\infty e^{-\gamma \ell^\beta} \leq \frac{2K}{\beta\gamma}e^{-\frac{\gamma}{2^{1+\beta}}m^\beta}$ also for $m=1$. Secondly we notice that 
\begin{equation}\label{eq:haricotville}\sum_{\ell=1}^m \ell^\beta \geq  \int_0^m t^\beta dt = \frac{m^{1+\beta}}{1+\beta}\end{equation}
for all integers $m\geq 1$ since the series is an upper Riemann sum for the integral. Combining \eqref{eq:todos-os-garotos-e-todos-as-garotas-querem}, \eqref{eq:se-voce-procura-a-amy}, \eqref{eq:bohnenstadt} and \eqref{eq:haricotville} we may now obtain
\begin{align*}|a_n(s)|&\leq  \sum_{i_1<\cdots<i_n} \prod_{\ell=1}^n C\exp(\kappa |s|-\gamma i_\ell^\beta)\\
& = \left(Ce^{\kappa |s|}\right)^n\sum_{i_1<\cdots<i_n}\exp\left(-\gamma \left(i_1^\beta+\cdots + i_n^\beta\right)\right)\\
&\leq\left(Ce^{\kappa |s|}\right)^n\sum_{i_1=1}^\infty \sum_{i_2=2}^\infty \cdots \sum_{i_n=n}^\infty \exp\left(-\gamma \left(i_1^\beta+\cdots + i_n^\beta\right)\right)\\
&=\left(Ce^{\kappa |s|}\right)^n\prod_{m=1}^n \sum_{\ell=m}^\infty \exp\left(-\gamma \ell^\beta\right) \\
&\leq \left(\frac{2KCe^{\kappa |s|}}{\beta\gamma}\right)^n\prod_{m=1}^n  \exp\left(-\frac{\gamma}{2^{1+\beta}} m^\beta\right) \\
&\leq \left(\frac{2KCe^{\kappa |s|}}{\beta\gamma}\right)^n  \exp\left(-\frac{\gamma}{(1+\beta)2^{1+\beta}} n^{1+\beta}\right)\end{align*}
which establishes the claimed inequality \eqref{eq:astrosheep-6000} with $\tilde\gamma :=\gamma/(2^{1+\beta}(1+\beta))$ and $\tilde{K}:=2KC/\beta\gamma$.

Now define a function $d_n \colon \mathbb{C}^2 \to \mathbb{C}$ for each $n\geq 1$ by $d_n(s,z):=\sum_{m=0}^n a_m(s)z^m$, and define also $d_\infty(s,z):=\sum_{m=0}^\infty a_m(s)z^m$, the convergence of the series being guaranteed by \eqref{eq:astrosheep-6000}. As a consequence of \eqref{eq:astrosheep-6000} it is clear that
\begin{equation}\label{eq:electric-head-part-196883}
\left|d_n(s,z)-d_\infty(s,z)\right| =\left|\sum_{m=n+1}^\infty a_m(s)z^m\right|=O\left(\exp\left(-\frac{\tilde\gamma}{2}n^\alpha\right)\right)
\end{equation}
uniformly on compact subsets of $\mathbb{C}^2$. It is clear by inspection that each $d_n$ is holomorphic, and using the convergence of $d_n$ to $d_\infty$ uniformly on compact sets together with Cauchy's theorem and Morera's theorem it follows easily that $d_\infty \colon \mathbb{C}^2 \to \mathbb{C}$ is holomorphic.  By Theorem \ref{th:ofundity} we have $d_\infty(s,z)=\det (I-z\mathscr{L}_s)$ for every $(s,z) \in \mathbb{C}^2$. In particular for every $s \in \mathbb{C}$ the zeros of $z \mapsto d_\infty(s,z)$ are precisely the reciprocals of the nonzero eigenvalues of $\mathscr{L}_s$, with the degree of each zero being equal to the algebraic multiplicity of the corresponding eigenvalue.

For each $s \in \mathbb{R}$ define $r_\infty(s):=\rho(\mathscr{L}_s)^{-1} \in (0,+\infty)$. We observe that $p(s)=-\log r_\infty(s)$ is a continuous function of $s$ by Lemma \ref{le:do-not-press} since it is a convex function of $s \in \mathbb{R}$, so $r_\infty(s) \colon \mathbb{R} \to (0,+\infty)$ is continuous. By the combination of Theorem \ref{th:opter} and Theorem \ref{th:ofundity}, for each $s \in \mathbb{R}$ the function $z\mapsto d_\infty(s,z)$ has a simple zero at $r_\infty(s)$ and has no zeroes with equal or smaller absolute value. We claim that there exist $n_0 \geq 1$, an open set $U \subset \mathbb{C}$ containing $[k,k+1]$, a holomorphic extension of $r_\infty|_{[k,k+1]}$ to $U$ and a sequence of holomorphic functions $r_n \colon U \to \mathbb{C}$ defined for all $n \geq n_0$ such that
\begin{equation}\label{eq:fun-to-funky}\sup_{s \in U} \left|r_n^{(\ell)} (s)-r_\infty^{(\ell)}(s)\right| = O\left(\exp\left(-\frac{\tilde \gamma}{2} n^\alpha\right)\right)\end{equation}
for all integers $\ell \geq 0$ and such that for all $n \geq n_0$ and $s \in [k,k+1]$, $r_n(s)$ is the smallest positive real number $x$ such that $d_n(s,x)=0$.

To prove the claim it is clearly sufficient, by the compactness of $[k,k+1]$, to show that every $s_0 \in [k,k+1]$ admits an open neighbourhood $U(s_0)$ such that $r_\infty$ extends holomorphically from $U(s_0)\cap [k,k+1]$ to all of $U(s_0)$, such that there exists a sequence of functions $r_n \colon U(s_0) \to \mathbb{C}$ defined for all large enough $n$ such that for all $s \in [k,k+1] \cap U(s_0)$, $r_n(s)$ is the smallest positive real number $x$ such that $d_n(s,x)=0$, and such that
\[\sup_{s \in U(s_0)} \left|r_n^{(\ell)} (s)-r_\infty^{(\ell)}(s)\right| = O\left(\exp\left(-\frac{\tilde \gamma}{2} n^\alpha\right)\right)\]
for all integers $\ell \geq 0$. The open set $U$ can then be taken equal to the union of a finite cover of $[k,k+1]$ by different sets $U(s)$, and the characterisation of $r_n(s)$ as the smallest positive root of $d_n(s,x)=0$ ensures that for each $n$ the local functions $r_n \colon U(s) \to \mathbb{C}$ extend consistently to a single well-defined function $r_n \colon U \to \mathbb{C}$. 

Let us therefore prove this local version of the preceding claim. Fix $s_0 \in [k,k+1]$. Since $z \mapsto d_\infty(s_0,z)$ has a unique zero in the closed disc with centre $0$ and radius $r_\infty(s_0)$, and all of its zeros are isolated, we may choose an open disc $D_2(s_0)$ with centre $z_0 \in \mathbb{R}$ and radius $R>0$ such that $[0,r_\infty(s_0)] \subset D_2(s_0)$ and such that $\overline{D_2(s_0)}$ contains no other zeros of $z \mapsto d_\infty(s_0,z)$.  A simple argument using compactness shows that we may choose a small open disc $D_1(s_0)$ centred at $s_0$ such that 
\[\sup_{s \in D_1(s_0)} \sup_{|z-z_0|=R} \left|d_\infty(s,z)-d_\infty(s_0,z)\right|<\inf_{|z-z_0|=R} \left|d_\infty(s_0,z)\right|\]
and by shrinking the neighbourhood $D_1(s_0)$ further if necessary we may assume using continuity that additionally $r_\infty(s) \in D_2(s_0)$ for all $s \in D_1(s_0)\cap [k,k+1]$.

By Rouch\'e's theorem, for all $s \in D_1(s_0)$ the function $z \mapsto d_\infty(s,z)$ has a unique zero in $D_2(s_0)$ and this zero is simple. When $s \in D_1(s_0) \cap[k,k+1]$ this zero must be equal to $r_\infty(s) \in D_2(s_0)$ by uniqueness. Extend $r_\infty \colon D_1(s_0)\cap [k,k+1] \to \mathbb{R}$ to a function $D_1(s_0)\to \mathbb{C}$ by defining $r_\infty(s)$ to be the unique zero of $z \mapsto d_\infty(s,z)$ in $D_2(s_0)$ for each $s \in D_1(s_0)$. By the holomorphic implicit function theorem and the simplicity of the zero $r_\infty \colon D_1(s_0)\to D_2(s_0)$ is holomorphic. Applying Lemma \ref{le:disco} we find, shrinking $D_1(s_0)$ and $D_2(s_0)$ if necessary, that there exist constants $C_\ell>0$, an integer $n_1 \geq 1$ and holomorphic functions $r_n \colon D_1(s_0) \to D_2(s_0)$ defined for all $n \geq n_1$ such that
\begin{align*}\sup_{s \in D_1(s_0)} \left|r_n^{(\ell)}(s)-r_\infty^{(\ell)}(s)\right| &\leq C_\ell \sup_{s \in D_1(s_0)}\sup_{z \in D_2(s_0)} \left| d_n(s,z)-d_\infty(s,z)\right|\\
&= O\left(\exp\left(-\frac{\tilde \gamma}{2} n^\alpha\right)\right)\end{align*}
for every integer $\ell \geq 0$, such that $r_n(s)$ is the unique zero of $z \mapsto d_n(s,z)$ in $D_2(s_0)$ for all $s \in D_1(s_0)$ and $n \geq n_1$ and is a simple zero for all such $s$ and $n$, such that $[0,r_\infty(s_0)] \subseteq D_2(s_0)$, and such that $D_2(s_0)$ is an open disc centred on the real axis. For all $s \in D_1(s_0) \cap [k,k+1]$ and $n \geq n_0$ the numbers $r_n(s)$ and $r_n(s)^*$ both lie in $D_2(s_0)$ and are both zeros of the polynomial $d_n(s,z)=\sum_{m=0}^n a_n(s)z^m$ since the coefficients of that polynomial are real and since $D_2(s_0)$, being a disc centred on the real axis, is symmetric with respect to complex conjugation. By the uniqueness of the zero $r_n(s)$ in $D_2(s_0)$ this is possible only if $r_n(s)=r_n(s)^*$, which is to say if $r_n(s)$ is real. Since $D_2(s_0)$ contains the interval from $0$ to $r_n(s)$, it follows that if $r_n(s)$ is positive then it is the smallest positive real root of $\sum_{m=0}^n a_n(s)x^m$ for all $s \in D_1(s_0)\cap [k,k+1]$. To complete the proof of the claim it therefore suffices to show that if $n$ is sufficiently large then $r_n(s)>0$ for all $s \in D_1(s_0)$. To see this choose $\delta \in (0,r_\infty(s_0))$ small enough that the open $\delta$-ball centred at $r_\infty(s_0)$ is contained in $D_2(s_0)$, and observe that by shrinking $D_1(s_0)$ further if necessary we may obtain
\[\inf_{s \in D_1(s_0)} \inf_{|z-r_\infty(s_0)|=\delta} |d_\infty(s,z)|>0\]
and hence for all large enough $n$
\[\sup_{s \in D_1(s_0)} \sup_{|z-r_\infty(s_0)|=\delta} |d_n(s,z)-d_\infty(s,z)|< \inf_{s \in D_1(s_0)} \inf_{|z-r_\infty(s_0)|=\delta} |d_\infty(s,z)|.\]
By Rouch\'e's theorem this implies that there exists $n_0\geq n_1$ such that for all $n \geq n_0$ and all $s \in D_1(s_0)$ there is a unique zero of $z \mapsto d_n(s,z)$ inside the circle of radius $\delta$ and centre $r_\infty(s_0)$, and since this region is a subset of $D_2(s_0)$ this root must equal $r_n(s)$ by the uniqueness of that root in $D_2(s_0)$. In particular for all $n \geq n_0$ and $s \in D_1(s_0) \cap [k,k+1]$ we have $r_n(s)>r_\infty(s_0)-\delta>0$ and no other root lies in $(0,r_n(s)) \subset D_2(s_0)$. Hence $r_n(s)$ is the smallest positive real root of $\sum_{m=0}^n a_n(s)x^m$ for all $s \in D_1(s_0)\cap [k,k+1]$ as required to prove the local version of the claim with $U(s_0):=D_1(s_0)$. The full statement of the claim follows.

We may now complete the proof of the theorem. Define $P_n(s):=r_n(s)^{-1}>0$ for all $s \in [k,k+1]$ and $n \geq n_0$, and $P(s):=r_\infty(s)^{-1}>0$ for all $s \in \mathbb{R}$. Observe that by Theorem \ref{th:opter} we have $e^{P(A_1,\ldots,A_N;s)}=P(s)$ for all $s \in [k,k+1]$. Since $r_\infty \colon U \to \mathbb{C}$ is holomorphic, $P$ is real-analytic at least on a neighbourhood of $[k,k+1]$. Since $r_\infty(s)$ is positive for all real $s$ and $[k,k+1]$ is compact it follows that
\begin{equation}\label{eq:cheesefish}\inf_{s \in [k,k+1]}r_\infty(s)>0\end{equation}
and by the case $\ell=0$ of \eqref{eq:fun-to-funky} we deduce that 
\begin{equation}\label{eq:toastfish}\lim_{n\to \infty} \inf_{s \in [k,k+1]}r_n(s)>0.\end{equation}
Using \eqref{eq:fun-to-funky}, \eqref{eq:cheesefish}, \eqref{eq:toastfish} and the expressions
\[\left|P_n(s)-P(s)\right| = \left|\frac{1}{r_n(s)}-\frac{1}{r_\infty(s)}\right|,\]
\[\left|P'_n(s)-P'(s)\right| = \left|\frac{r_n'(s)}{r_n(s)^2}-\frac{r_\infty'(s)}{r_\infty(s)^2}\right|\]
and
\[\left|P''_n(s)-P''(s)\right| = \left|\frac{r_n''(s)r_n(s)-r_n'(s)^2}{r_n(s)^4}-\frac{r_\infty''(s)r_\infty(s)-r_\infty'(s)^2}{r_\infty(s)^4}\right|\]
it follows by elementary manipulations that
\begin{equation}\label{eq:cnut}\sup_{s \in [k,k+1]}\left|P_n(s)-P(s)\right| = O\left(\exp\left(-\frac{\tilde\gamma}{2}n^\alpha\right)\right),\end{equation}
\begin{equation}\label{eq:wnaker}\sup_{s \in [k,k+1]}\left|P'_n(s)-P'(s)\right| = O\left(\exp\left(-\frac{\tilde\gamma}{2}n^\alpha\right)\right)\end{equation}
and
\begin{equation}\label{eq:avocado-toast-is-the-end-of-civilisation-according-to-the-daily-telegraph}\sup_{s \in [k,k+1]}\left|P''_n(s)-P''(s)\right| = O\left(\exp\left(-\frac{\tilde\gamma}{2}n^\alpha\right)\right).\end{equation}
In the case where we do not assume that $\max_{1 \leq i \leq N}\vertle{A_i}<1$ for some norm on $\mathbb{R}^d$ the estimate \eqref{eq:cnut} already completes the proof of Theorem \ref{th:main}. Otherwise, we claim that $\inf_{s \in [k,k+1]}P''(s)>0$ and $\sup_{s \in [k,k+1]}P'(s)<0$. Let $p(s):=\log P(s)$ for $s \in \mathbb{R}$ so that $P'(s)=p'(s)P(s)$ and $P''(s)=p''(s)P(s)+p'(s)^2P(s)$. Obviously $p$ is real-analytic on $[k,k+1]$ since $P$ is positive and real-analytic there, and $p$ is convex by Lemma \ref{le:do-not-press}, so necessarily $p''(s)\geq 0$ for all $s \in [k,k+1]$. By Lemma \ref{le:do-not-press} we have $p'(s)<0$ for all $s \in [k,k+1]$ and therefore
\begin{equation}\label{eq:hi-ms-coutts}\sup_{s \in [k,k+1]}P'(s)=\sup_{s \in [k,k+1]} p'(s)P(s)<0.\end{equation}
Similarly we observe that $\inf_{s \in [k,k+1]}|p'(s)|>0$, and since $P''(s)=p''(s)P(s)+p'(s)^2P(s)\geq p'(s)^2P(s)$ we likewise deduce that $\inf_{s \in [k,k+1]}P''(s)>0$ as claimed. 

Combining the previous claim with \eqref{eq:avocado-toast-is-the-end-of-civilisation-according-to-the-daily-telegraph} we find in particular that $\inf_{s \in [k,k+1]}P''_n(s)>0$ for all large enough $n$, which proves that each such function $P_n \colon [k,k+1] \to \mathbb{R}$ is convex. By the hypothesis $\dimaff(A_1,\ldots,A_N) \in (k,k+1)$ of Theorem \ref{th:main} there exists a solution $s \in (k,k+1)$ to $P(s)=1$, and since $P$ has negative derivative on $[k,k+1]$ this implies that $P(k)>1>P(k+1)$. Combining this observation with  \eqref{eq:cnut} we find that $P_n(k)>1>P_n(k+1)$ for all large enough $n$, and by the combination of \eqref{eq:hi-ms-coutts} and \eqref{eq:wnaker} we find that $\sup_{s \in [k,k+1]}P'_n(s)\leq -c<0$ for all large enough $n$ where $c>0$ is some positive constant. It follows that for all large enough $n$ there exists a unique $s_n \in [k,k+1]$ such that $P_n(s_n)=1$. Let $s_\infty:=\dimaff(A_1,\ldots,A_N) \in [k,k+1]$ be the unique solution to $P(s_\infty)=1$. If $s_n \neq s_\infty$ then by the Mean Value Theorem there exists $t$ strictly between $s_n$ and $s_\infty$ such that
\[P'(t)=\frac{P(s_n)-P(s_\infty)}{s_n-s_\infty}\]
and therefore since $P_n(s_n)=1=P(s_\infty)$ we obtain
\[\left|s_n-s_\infty\right|=\frac{|P(s_n)-P(s_\infty)|}{|P'(t)|} =\frac{|P(s_n)-P_n(s_n)|}{|P'(t)|}\\\leq c^{-1}|P(s_n)-P_n(s_n)|.\]
The inequality $|s_n-s_\infty| \leq c^{-1}|P(s_n)-P_n(s_n)|$ obviously also holds when $s_n=s_\infty$, so
\[\left|s_n-s_\infty\right|=O\left(\exp\left(-\frac{\tilde\gamma}{2}n^\alpha\right)\right)\]
as $n \to \infty$ using \eqref{eq:cnut}.  The proof of the theorem is complete.
\end{proof}


\section{Examples}\label{se:apps}

\subsection{Methodology}\label{ss:urk}

There are two intuitively natural mechanisms by which to make the approximations given in Theorem \ref{th:main} yield an approximation to the affinity dimension. On the one hand since $e^{P(A_1,\ldots,A_n;s)}$ is decreasing in $s$ and since the affinity dimension is the unique $s \in [k,k+1]$ such that $1$ is the leading eigenvalue of $\mathscr{L}_s$, the affinity dimension corresponds to the smallest $s \in [k,k+1]$ such that $\det(I-\mathscr{L}_s)=0$, which is to say the smallest $s \in [k,k+1]$ such that $\sum_{m=0}^\infty a_m(s)=0$. One might therefore attempt to approximate the affinity dimension by looking for the smallest solution $s$ to the equation $\sum_{m=0}^na_m(s)=0$ for each fixed $n$. In practice this is impractical since $\mathscr{L}_s$ may in general have infinitely many positive real eigenvalues and the number of solutions to $\sum_{m=0}^na_m(s)=0$ may therefore be extremely large and the function itself highly oscillatory. 

\begin{table}
\begin{center}
\begin{tabular}{*3c}
\toprule
$n$  & Approximation to affinity dimension&CPU time\\
\midrule
2& 1.14341 79598 76019 95000 60486 91827 85789 60135 &0.043s \\
3&\underline{1.1}1827 23247 08006 28499 89060 66409 13091 47143&0.044s\\
4 &\underline{1.11}538 89736 67461 99644 51849 00512 18003 54788&0.053s\\
5 &\underline{1.115}60 42107 66261 56209 11669 09958 04069 77087&0.075s\\
6 &\underline{1.11560} 31850 39305 08475 98379 83168 80085 68510&0.11s\\
7 &\underline{1.11560} \underline{3}2522 24751 03699 38823 87724 66623 37012&0.16s\\
8 & \underline{1.11560} \underline{325}79 27402 64806 11546 27227 11083 45893&0.30s\\
9 &\underline{1.11560} \underline{3257}7 86505 71154 77556 50836 85812 53178&0.39s\\
10 &\underline{1.11560} \underline{32577} \underline{8}7028 88533 65835 00045 83936 61000& 0.67s\\
11 &\underline{1.11560} \underline{32577} \underline{870}30 91898 36777 33249 49956 17495&1.2s\\
12 &\underline{1.11560} \underline{32577} \underline{87030} 89197 97928 71446 51257 73313&2.0s\\
13 & \underline{1.11560} \underline{32577} \underline{87030} \underline{89}218 88050 96492 48585 23429 &4.3s\\
14 &\underline{1.11560} \underline{32577} \underline{87030} \underline{89218} \underline{8}4942 17623 75680 33697&8.8s\\
15 &\underline{1.11560} \underline{32577} \underline{87030} \underline{89218} \underline{849}37 14660 75123 27001 &20s\\
16 & \underline{1.11560} \underline{32577} \underline{87030} \underline{89218} \underline{84937} \underline{14}840 85419 85122 &44s\\
17& \underline{1.11560} \underline{32577} \underline{87030} \underline{89218} \underline{84937} \underline{14840} 24544 08248&100s\\
18 & \underline{1.11560} \underline{32577} \underline{87030} \underline{89218} \underline{84937} \underline{14840} \underline{245}74 24137&210s\\
19 &\underline{1.11560} \underline{32577} \underline{87030} \underline{89218} \underline{84937} \underline{14840} \underline{24574} \underline{2}5551&440s\\
20 &\underline{1.11560} \underline{32577} \underline{87030} \underline{89218} \underline{84937} \underline{14840} \underline{24574} \underline{25551}&990s\\
\bottomrule
\end{tabular}\bigskip\medskip
\caption{Approximations to the affinity dimension of Example 1 calculated using Theorem \ref{th:main} and the secant method as described in \S\ref{ss:urk}, implemented in Wolfram Mathematica. The CPU time used in each approximation is as reported by Mathematica's {\tt{Timing}} function. For $n=1$ the approximation to the pressure function has no root in $(1,2)$ and this line is therefore omitted from the table. Digits which are empirically observed to have converged to a stable value are underlined.
}\label{ta:blen}
\end{center}
\end{table}

In practice we therefore adopt the following alternative approach. For large $n$ the smallest positive real root $x=r_n(s)$ of $\sum_{m=0}^n a_m(s)x^m$ approximates the reciprocal of the leading eigenvalue of $\mathscr{L}_s$. Moreover, for large $n$ the function $s \mapsto r_n(s)^{-1}$ is convex and strictly decreasing with a unique root in $[k,k+1]$ by virtue of Theorem \ref{th:main}. Computing the unique root of a convex decreasing function is a far more tractable enterprise than finding the smallest root of an oscillating function, and for this reason our application of Theorem \ref{th:main} follows the approach of solving $r_n(s)=1$. For this problem we use the secant method. Since $r_n^{-1}$ is convex and decreasing the convergence of the sequence of approximations generated by the secant method is guaranteed with super-exponential rate $O(\theta^{m^{(1+\sqrt{5})/2}})$ for some $\theta \in (0,1)$. In practical instances we found that the sequence $(s_m)$ consistently converged empirically to 40 decimal places by around $m\simeq 12$ independently of $n$. The results of this procedure applied to some examples of two- and three-dimensional affine iterated function systems are presented in this section.

For large $n$ one may show that the trace $t_n(s)$ appearing in Theorem \ref{th:main} approximates the value $e^{nP(A_1,\ldots,A_N;s)}$ whereas the coefficients $a_n(s)$ are shown in Theorem \ref{th:main} to decrease to zero with super-exponential speed. The small size of $a_n(s)$ is thus attributable to additive cancellation between potentially very large summands. It is therefore likely to be necessary in implementation to record the traces $t_n(s)$ to significantly more decimal places than are desired for the ultimate approximation. In the computations which follow the traces $t_n(s)$ were calculated in arbitrary precision, reducing to finite precision only for the outcome of the calculation of the coefficients $a_n(s)$.

\subsection{Example 1: a pair of dominated matrices}\label{ss:slary}
Define
\[A_1:=
\begin{pmatrix}
-\frac{4}{7}&\frac{5}{7}\\
0&\frac{1}{7}
\end{pmatrix},\qquad
A_2:=\begin{pmatrix}
\frac{1}{7}&0\\
-\frac{5}{7}&-\frac{4}{7}
\end{pmatrix}.
\]
We claim that the pair $(A_1,A_2)$ is $1$-dominated. Indeed, define
\[\mathcal{C}_1:=\left\{\begin{pmatrix}x\\y\end{pmatrix} \in \mathbb{R}^2 \colon |x|\geq 2|y|\right\},\]
\[\mathcal{C}_2:=\left\{\begin{pmatrix}x\\y\end{pmatrix} \in \mathbb{R}^2 \colon |y|\geq 2|x|\right\}.\]
If $(x,y)^\top \in \mathcal{C}_1$ then
\[\left|\frac{5}{7}y-\frac{4}{7}x\right|\geq \frac{4}{7}|x|-\frac{5}{7}|y|\geq \frac{3}{7}|y|\geq \left|\frac{2}{7}y\right|\]
and equality of the first and last terms is only possible if $y=0$ and consequently $x=0$. In particular if $(x,y)^\top \in \mathcal{C}_1$ is nonzero we obtain $A_1(x,y)^\top \in \Int \mathcal{C}_1$. Moreover for $(x,y)^\top \in \mathcal{C}_1$ we also have 
\[\left|\frac{4}{7}y+\frac{5}{7}x\right|\geq \frac{5}{7}|x|-\frac{4}{7}|y|\geq \frac{3}{7}|x| \geq \left|\frac{2}{7}x\right|\]
which yields $A_2(x,y)^\top \in \Int \mathcal{C}_2$ when $(x,y)^\top$ is nonzero. In a similar manner, if $(x,y)^\top \in \mathcal{C}_2$ then
\[\left|\frac{5}{7}y-\frac{4}{7}x\right|\geq \frac{5}{7}|y|-\frac{4}{7}|x|\geq \frac{3}{7}|y|\geq \left|\frac{2}{7}y\right|\]
and
\[\left|\frac{4}{7}y+\frac{5}{7}x\right|\geq \frac{4}{7}|y|-\frac{5}{7}|x|\geq \frac{3}{7}|x|\geq \left|\frac{2}{7}x\right|\]
which respectively give $A_1(x,y)^\top \in \Int \mathcal{C}_1$ and $A_2(x,y)^\top \in \Int \mathcal{C}_2$ when $(x,y)^\top$ is nonzero. 

If we now let $w=(1,1)^\top$ then $\langle u,w\rangle$ is never zero for any nonzero $u \in \mathcal{C}_1 \cup \mathcal{C}_2$, so defining 
\[\mathcal{K}_i:=\left\{u \in \mathcal{C}_i\colon \langle u,w\rangle>0\right\}\]
for $i=1,2$ it is not difficult to see that $(\mathcal{K}_1,\mathcal{K}_2)$ is a multicone for $(A_1,A_2)$. In particular Theorem \ref{th:main} may be applied to estimate the affinity dimension of the pair $(A_1,A_2)$. Let $(B_1,B_2):=(A_1,-A_2)$. Since
\begin{align*}e^{P(A_1,A_2;1)}=e^{P(B_1,B_2;1)}&=\lim_{n \to \infty} \left(\sum_{|\iii|=n} \left\|B_\iii\right\|\right)^{\frac{1}{n}}  \\
&\geq \lim_{n \to \infty} \left\|\sum_{|\iii|=n} B_\iii\right\|^{\frac{1}{n}}\\
&=\lim_{n \to \infty}  \left\|(B_1+B_2)^n\right\|^{\frac{1}{n}}=\rho(B_1+B_2)=\frac{\sqrt{50}}{7}>1\end{align*}
and 
\[e^{P(A_1,A_2;2)}=|\det A_1|+|\det A_2|=\frac{8}{49}<1\]
 we infer that $\dimaff(A_1,A_2) \in (1,2)$. The first 20 approximations to the affinity dimension of $(A_1,A_2)$ are tabulated in Table \ref{ta:blen}. 


\subsection{Example 2: a three-dimensional iterated function system}

\begin{table}
\begin{center}
\begin{tabular}{*3c}
\toprule
n & Approximation to affinity dimension&CPU time\\
\midrule
3&1.74010 38961 34544 64381 66016 57752 82592 79145&0.067s\\
4 & \underline{1}.53612 13489 34570 18769 13237 56458 61628 45041&0.10s\\
5 &\underline{1.5}8779 31446 44939 17928 98900 28708 16065 92496&0.15s\\
6 & \underline{1.58}459 23810 06597 43285 21249 54866 32813 68839&0.22s\\
7 &\underline{1.584}77 97771 44149 34557 48903 92413 22985 52229&0.33s\\
8 & \underline{1.58477} 17757 07488 53767 71488 42424 52891 52003&0.63s\\
9 & \underline{1.58477} 20386 65944 76377 72361 85895 44529 09738&0.80s\\
10 & \underline{1.58477} \underline{203}18 53062 52952 58955 36166 25319 46959&1.4s\\
11 & \underline{1.58477} \underline{2031}9 95110 47059 43620 26740 31575 13317&2.4s\\
12 & \underline{1.58477} \underline{20319} \underline{9}2686 60697 00747 19778 01115 41015&5.4s\\
13 & \underline{1.58477} \underline{20319} \underline{92}720 93370 05697 62846 36869 58071&12s\\
14 & \underline{1.58477} \underline{20319} \underline{92720} 52545 02878 00445 78535 74528&27s\\
15 & \underline{1.58477} \underline{20319} \underline{92720} \underline{52}956 88351 89418 63989 50927&59s\\
16 & \underline{1.58477} \underline{20319} \underline{92720} \underline{5295}3 32862 81715 84179 24019&130s\\
17& \underline{1.58477} \underline{20319} \underline{92720} \underline{52953} \underline{3}5507 79078 84111 41677&270s\\
18 & \underline{1.58477} \underline{20319} \underline{92720} \underline{52953} \underline{35}490 71502 87276 30757&560s\\
19 & \underline{1.58477} \underline{20319} \underline{92720} \underline{52953} \underline{35490} 81124 12318 84553&1200s\\
20 & \underline{1.58477} \underline{20319} \underline{92720} \underline{52953} \underline{35490} \underline{81}076 56294 07542&2800s\\
21 & \underline{1.58477} \underline{20319} \underline{92720} \underline{52953} \underline{35490} \underline{81076} 77018 06325&5900s\\ 
\bottomrule
\end{tabular}\bigskip\medskip
\caption{Approximations to the affinity dimension of Example 3 calculated using Theorem \ref{th:main} and the secant method as described in \S\ref{ss:urk}, implemented in Wolfram Mathematica. The CPU time used in each approximation is as reported by Mathematica's {\tt{Timing}} function. Digits which are empirically observed to have converged to a stable value are underlined. Convergence is noticeably slower than for two-dimensional examples: in this context our bound for the error in the $n^{\mathrm{th}}$ approximation is $O(\exp(-\gamma n^{5/4}))$ as opposed to $O(\exp(-\gamma n^{2}))$ in the other examples. For $n=1,2$ the approximation to the pressure function has no root in $(1,2)$ and these lines are therefore omitted.
}\label{ta:aff2}
\end{center}
\end{table}
\begin{figure}
    \centering
    \includegraphics[width=0.7\linewidth]{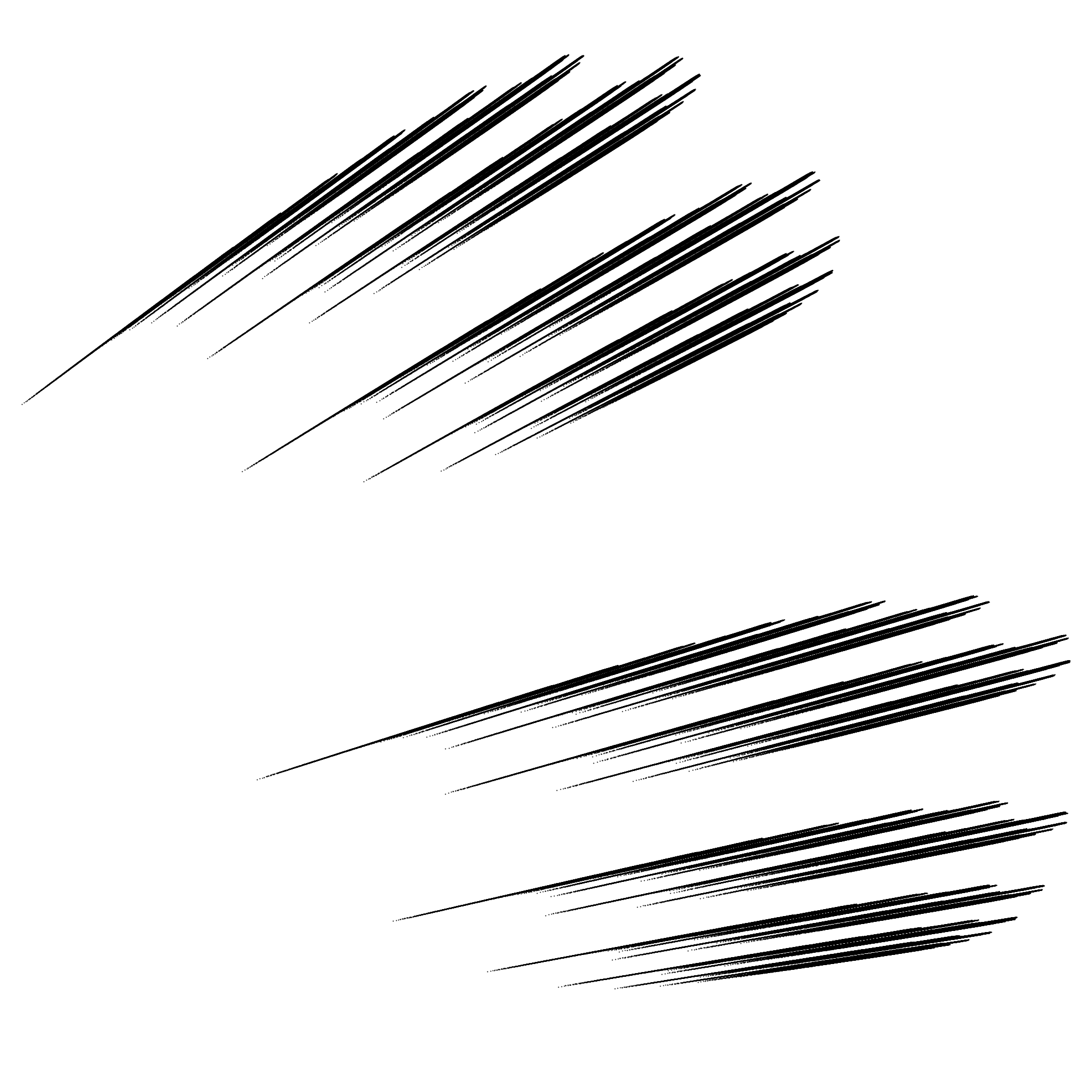}
    \caption{A projection of the attractor of the iterated function system defined by Example 3. Approximations to the affinity dimension computed using Theorem \ref{th:main} are listed in Table \ref{ta:aff2}.  It is known from work of Falconer \cite[\S5]{Fa88} that the upper box dimension $\overline{\dim}_BX$ is bounded above by $\dimaff(A_1,A_2)$, but unlike the case of planar affine iterated function systems current techniques are not powerful enough to determine whether or not $\dim_H X= \dimaff (A_1,A_2)$.}\label{fi:sting}
    \end{figure}

Consider $(A_1,A_2)$ where
\[A_1:=\frac{1}{12}\begin{pmatrix}5 & 4 & 1\\ 5& 5 & 4\\ 0 &1 &5\end{pmatrix},\qquad A_2:=\frac{1}{12}\begin{pmatrix}5 & 5 & 0\\ 4& 5 & 1\\ 1 &4 &5\end{pmatrix}=A_1^\top \]
and note that $A_1$ and $A_2$ are contractions in the Euclidean norm. It is easily checked that $(A_1A_1,A_1A_2,A_2A_1,A_2A_2)$ is a tuple of positive invertible matrices and is therefore $1$-dominated. By the characterisation of domination in terms of singular values this clearly implies that $1$-domination holds also for $(A_1,A_2)$.

We identify each $A_i$ with the corresponding linear map $\mathbb{R}^3 \to \mathbb{R}^3$ defined by $A_i$ with respect to the standard basis $e_1,e_2,e_3$ of $\mathbb{R}^3$. With respect to the basis $e_1\wedge e_2, e_1\wedge e_3,e_2\wedge e_3$ for $\wedge^2\mathbb{R}^3$ we have
\[A_1^{\wedge 2}=\frac{1}{144}\begin{pmatrix}5&15&11 \\5&25&19\\5&25&21\end{pmatrix},\qquad A_2^{\wedge 2}=\frac{1}{144}\begin{pmatrix}5&5&5\\15&25&25\\11&19&21\end{pmatrix}.\]
Since  $(A_1^{\wedge 2},A_2^{\wedge 2})$ is thus representable by a pair of positive matrices we see that $(A_1,A_2)$ is both $1$-and $2$-dominated. Using non-negativity it follows by a theorem of Yu. V. Protasov (\cite{Pr10}) that
\[\lim_{n \to \infty} \left(\sum_{|\iii|=n} \|A_\iii\| \right)^{\frac{1}{n}}=\rho(A_1+A_2)>1\]
and
\[\lim_{n \to \infty} \left(\sum_{|\iii|=n} \left\|A_\iii^{\wedge 2}\right\| \right)^{\frac{1}{n}}=\rho\left(A_1^{\wedge 2}+A_2^{\wedge 2}\right)<1.\]
Thus $P(A_1,A_2;1)>0>P(A_1,A_2;2)$ and consequently $\dimaff (A_1,A_2) \in (1,2)$, and we conclude that Theorem \ref{th:main} is applicable to the computation of $\dimaff (A_1,A_2)$. The first 21 approximations to $\dimaff (A_1,A_2)$ are presented in Table \ref{ta:aff2}. An illustration of the attractor of the iterated function system
\[T_1\begin{pmatrix}x\\y\\z\end{pmatrix} := \frac{1}{12}\begin{pmatrix}5 & 4 & 1\\ 5& 5 & 4\\ 0 &1 &5\end{pmatrix}\begin{pmatrix}x\\y\\z\end{pmatrix}+\begin{pmatrix}1\\0\\0\end{pmatrix}\]
 \[T_2\begin{pmatrix}x\\y\\z\end{pmatrix} :=\frac{1}{12}\begin{pmatrix}5 & 5 & 0\\ 4& 5 & 1\\ 1 &4 &5\end{pmatrix}\begin{pmatrix}x\\y\\z\end{pmatrix}+\begin{pmatrix}0\\0\\1\end{pmatrix}\]
is given in Figure \ref{fi:sting}.

\begin{figure}
    \centering
    \includegraphics[width=0.7\linewidth]{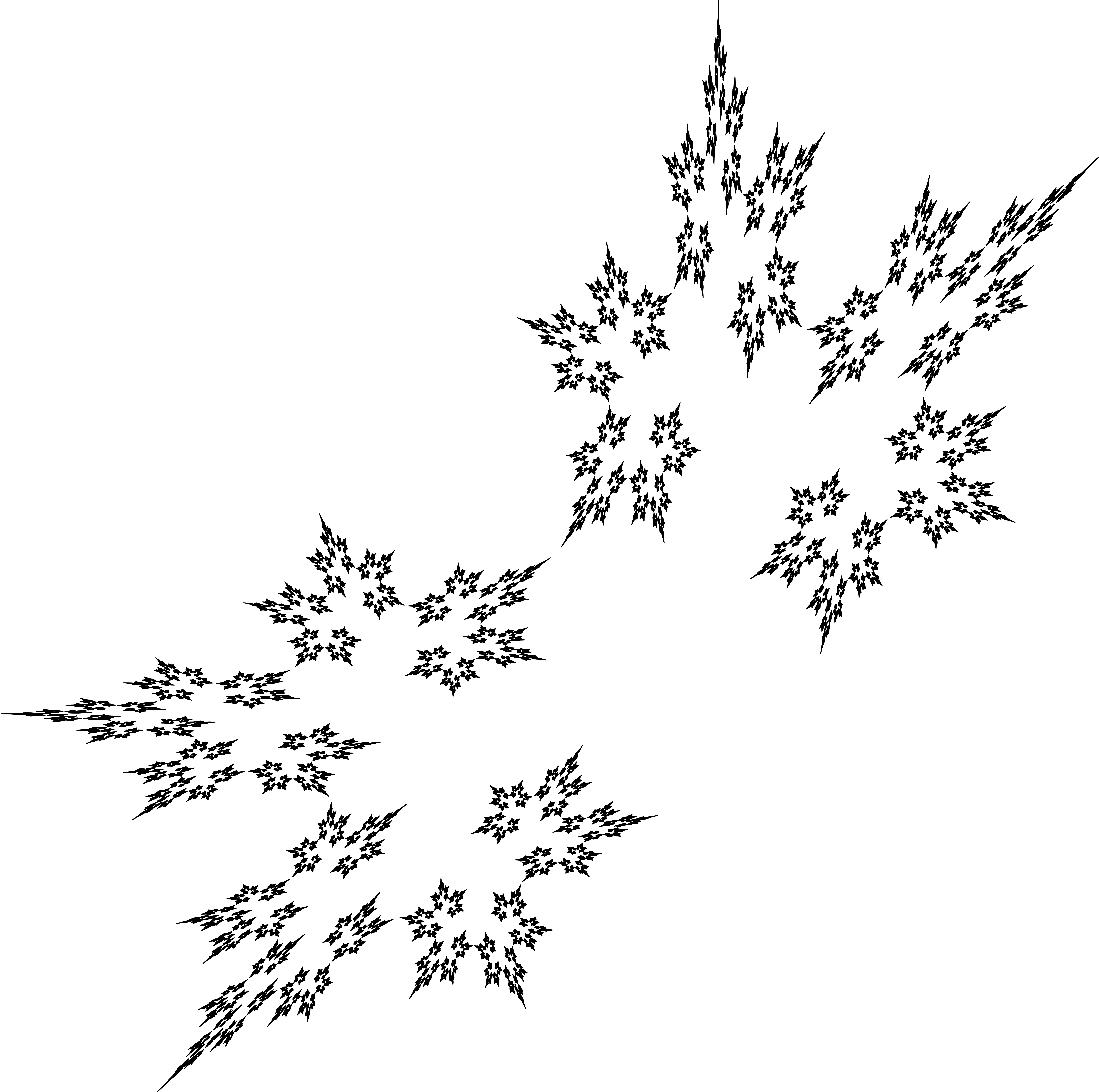}
    \caption{This self-affine set was shown in \cite[\S6.6]{MoSh17} to have Hausdorff dimension equal to the affinity dimension of the defining iterated function system. However, the linear parts of the defining affine transformations have non-real eigenvalues and Theorem \ref{th:main} is not applicable. Non-rigorous estimates using the discretisation method described in \S\ref{se:nondom} as tabulated in Table \ref{ta:pdance} suggest that the affinity dimension is equal to approximately 1.522688.}\label{fi:gs}
    \end{figure}

\section{Non-dominated matrices}\label{se:nondom}

If $(A_1,\ldots,A_N) \in M_2(\mathbb{R})^N$ is a tuple of invertible matrices which is not $1$-dominated then by a line of reasoning due to A. Avila \cite{Yo04} there exist tuples $(A_1',\ldots,A_N')$ arbitrarily close to $(A_1,\ldots,A_N)$ with the property that some product $A_{i_1}'\cdots A_{i_n}'$ has complex eigenvalues. For such matrices the formula for $t_n(s)$ in  Theorem \ref{th:main} has no clear meaning, and also for such matrices no open subset of $\mathbb{RP}^1$ may be found which is mapped strictly inside itself by the action of the matrices $A_i'$, preventing the construction of a trace-class transfer operator in direct mimicry of Theorem \ref{th:main}. For such matrices it is therefore difficult to see how any reasonable adaptation of Theorem \ref{th:main} could be made. In this sense we believe that $1$-domination, or multipositivity, is the weakest open condition on the matrices $A_1,\ldots,A_N$ which permits a version of Theorem \ref{th:main} to be proved.

\begin{table}
\begin{center}
\begin{tabular}{*3c}
\toprule
Mesh size & \makecell{Approximation to \\affinity dimension} & CPU time\\
\midrule
$2$&\underline{1}.02591849& 0.010s\\
$2^2$&\underline{1}.07532743 & 0.0065s\\
$2^3$&\underline{1}.\underline{11}171266 &  0.018s\\
$2^4$&\underline{1}.\underline{11}715797& 0.036s\\
$2^5$&\underline{1}.\underline{11}608327 & 0.053s\\
$2^6$&\underline{1}.\underline{11}557816 & 0.80s\\
$2^7$&\underline{1}.\underline{115}37306& 0.46s\\
$2^8$&\underline{1}.\underline{115}61123 & 0.35s\\
$2^9$&\underline{1}.\underline{115}59940 & 0.65s\\
$2^{10}$&\underline{1}.\underline{115}61053 & 1.8s\\
$2^{11}$&\underline{1}.\underline{115}58601& 2.7s\\
$2^{12}$&\underline{1}.\underline{115}60216 & 4.8s\\
$2^{13}$&\underline{1}.\underline{11560}441 & 24s\\
$2^{14}$&\underline{1}.\underline{11560}185& 21s\\
$2^{15}$&\underline{1.}\underline{11560}275 & 67s\\
$2^{16}$&\underline{1}.\underline{11560}321 & 270s\\
$2^{17}$&\underline{1}.\underline{115603}15& 4100s\\
\bottomrule
\end{tabular}\bigskip\medskip
\caption{Estimates of the affinity dimension of Example 1 calculated using the non-rigorous discretisation method described in \S\ref{se:nondom}. Even at small mesh sizes the first few decimal places show good agreement with Table \ref{ta:blen} but convergence in subsequent decimal places is markedly slower. Digits which are empirically observed to have converged to a stable value are underlined.}\label{ta:nkie}
\end{center}
\end{table}

However, for non-dominated matrices it is still possible to obtain non-rigorous estimates of the affinity dimension by other techniques. Given $A_1,\ldots,A_N \in GL_2(\mathbb{R})$ and $s \in [0,1]$ we may define an operator $\mathscr{L}_s \colon C^\alpha(\mathbb{RP}^1) \to C^\alpha(\mathbb{RP}^1)$ by
\[\left(\mathscr{L}_sf\right)(\overline{u}):=\sum_{i=1}^N \left(\frac{\|A_iu\|}{\|u\|}\right)^s f\left(\overline{A_iu}\right),\]
and for $s \in [1,2]$ by
\[\left(\mathscr{L}_sf\right)(\overline{u}):=\sum_{i=1}^N \left(\frac{\|A_iu\|}{\|u\|}\right)^{2-s} |\det A_i|^{s-1} f\left(\overline{A_iu}\right),\]
in such a manner that
\[\rho(\mathscr{L}_s)=\lim_{n\to\infty}\left(\sum_{i_1,\ldots,i_n=1}^N \varphi^s\left(A_{i_1}\cdots A_{i_n}\right)\right)^{\frac{1}{n}}\]
and such that $\rho(\mathscr{L}_s)$ is a simple eigenvalue of $\mathscr{L}_s$, as long as $\alpha \in (0,1)$ is chosen suitably small (in a manner which in general will depend on $s$) and mild algebraic non-degeneracy conditions on $(A_1,\ldots,A_N)$ are met. (These spectral properties are guaranteed by, for example, \cite[Th\'eor\`eme 8.8]{GuLe04}.) We could then hope to estimate the spectral radius of $\mathscr{L}_s$ for different values of $s$ by discretising the phase space $\mathbb{RP}^1$, constructing a large matrix representing a discretised action of $\mathscr{L}_s$, and working on the supposition that the spectral radius of the matrix is a good approximation to $\rho(\mathscr{L}_s)$ and hence to $e^{P(A_1,\ldots,A_N;s)}$. In practical experiments we were able to obtain around five decimal places of accuracy for the affinity dimension by discretising $\mathbb{RP}^1$ into approximately $10^4$ evenly-spaced mesh points: see Tables \ref{ta:nkie} and \ref{ta:pdance}. We observe in particular that the results obtained in Table \ref{ta:nkie} show good agreement with Theorem \ref{th:main} when tested on the multipositive matrix set described in Example 2. However, we have not been able to make this method of estimation rigorous. This approach could also be applied to higher-dimensional affine iterated function systems but we have not investigated the matter of finding suitable discretisations of the more complicated phase spaces required in this context.

\begin{table}
\begin{center}
\begin{tabular}{*3c}
\toprule
Mesh size & \makecell{Approximation to \\affinity dimension} & CPU time\\
\midrule
$2$&1.50000000& 0.0028s\\
$2^2$&\underline{1}.\underline{5}1578683 & 0.0025s\\
$2^3$&\underline{1}.\underline{5}1254065  &  0.0047s\\
$2^4$&\underline{1}.\underline{52}070716 & 0.033s\\
$2^5$&\underline{1}.\underline{52}415711 & 0.059s\\
$2^6$&\underline{1}.\underline{52}305542 & 0.079s\\
$2^7$&\underline{1}.\underline{52}290806 & 0.13s\\
$2^8$&\underline{1}.\underline{522}62668 & 0.26s\\
$2^9$&\underline{1}.\underline{522}69395 & 0.61s\\
$2^{10}$&\underline{1}.\underline{522}70408 & 1.1s\\
$2^{11}$&\underline{1}.\underline{522}69152 & 2.2s\\
$2^{12}$&\underline{1}.\underline{5226}8717  & 4.5s\\
$2^{13}$&\underline{1}.\underline{52268}810 & 7.7s\\
$2^{14}$&\underline{1}.\underline{52268}795 & 18s\\
$2^{15}$&\underline{1}.\underline{522687}80 & 55s\\
$2^{16}$&\underline{1}.\underline{5226878}0 & 220s\\
$2^{17}$&\underline{1}.\underline{5226878}2 & 1400s\\
\bottomrule
\end{tabular}\bigskip\medskip
\caption{Estimates of the affinity dimension of the iterated function system defined in \cite[\S6.6]{MoSh17} and illustrated in Figure \ref{fi:gs}, calculated using the non-rigorous discretisation method described in \S\ref{se:nondom}. Digits which are empirically observed to have converged to a stable value are underlined. No rigorous estimate of the affinity dimension of this IFS is currently available.}\label{ta:pdance}
\end{center}
\end{table}

\section{Acknowledgements}
This research was supported by the Leverhulme Trust (Research Project Grant number RPG-2016-194). The author thanks O. Bandtlow for helpful comments and suggestions. The author additionally thanks an anonymous reviewer for suggesting several economies of argument.

\bibliographystyle{acm}
\bibliography{polecat}
\end{document}